\documentclass[10pt,a4paper,openany]{book}
\usepackage{graphicx}
 \usepackage[reqno, namelimits, sumlimits]{amsmath}
   \usepackage{amsfonts}
   \usepackage{amssymb}
  \usepackage{amsthm}
  \usepackage{xcolor}
\usepackage{makeidx}
\newtheorem{thm}{Theorem}[section]
\newtheorem{cor}[thm]{Corollary}
\newtheorem{lem}[thm]{Lemma}
\newtheorem{prop}[thm]{Proposition}

\newtheorem{example}[thm]{Example}

\usepackage{times}

\newtheorem{defn}[thm]{Definition}
\newtheorem{rem}[thm]{Remark}
\def\R{{\mathbb R}}

\def\N{{\mathbb N}}

\def\C{{\mathbb C}}

\def\cE{{\mathcal E}}
\def\cF{{\mathcal F}}

\def\cG{{\mathcal G}}

\def\cL{{\mathcal L}}

\makeindex

\begin{document}
\title{The convergence problem for dissipative autonomous systems: classical methods and recent advances}
\author{Alain Haraux \& Mohamed Ali Jendoubi}
\maketitle

\begin{center}
{\bf{PREFACE}}
\end{center}\bigskip\bigskip

Our  initial motivation was to provide an up to date translation of the monograph \cite {MR1084372} written in french by the first author, taking account of more recent developments of infinite dimensional dynamics based on the {\L}ojasiewicz gradient inequality. \\

While preparing the project it appeared that it would not be easy to cover the entire scope of the french version in a reasonable amount of time, due to the fact that the non-autonomous systems require sophisticated tools which underwent major improvement during the last decade. \\

In order to keep the present work within modest size bounds and to make it available to the readers without too much delay, we decided to make a first volume entirely dedicated to the so-called convergence problem for autonomous systems of dissipative type. We hope that this volume will help the interested reader to make the connection between the rather simple background developed in the french monograph and the rather technical specialized literature on the convergence problem which grew up rather fast in the recent years. \\
\bigskip

\begin{center}
Alain HARAUX\\
{\em Laboratoire Jacques-Louis Lions\\
  Universit\'e Pierre et Marie Curie et CNRS\\
  Bo\^{\i}te courrier 187\\
  75252 Paris Cedex 05,  France}\\
haraux@ann.jussieu.fr \\
\end{center}

\medskip
\begin{center}
Mohamed Ali JENDOUBI \\
{\em Universit\'e de Carthage,\\
Institut Pr\'eparatoire aux Etudes Scientifiques et Techniques,\\
B.P. 51 2070 La Marsa, Tunisia} \\
ma.jendoubi@fsb.rnu.tn \\
\end{center}


 \tableofcontents
\chapter [Introduction and basic tools]{Introduction and basic tools}

\section{Introduction}The present text is devoted to a rather specific subject: convergence to \index{equilibrium point}equilibrium, as $t$ tends to infinity,  of the solutions to differential equations on the positive halfline $\{t \ge 0\}$ of the general form  $$ U'(t)+ {\cal A} U(t)= 0$$ where ${\cal A}$ is a nonlinear, time independent, possibly unbounded operator on some Banach space $X$. By \index{equilibrium point} equilibrium we mean a solution of the so-called stationary problem 
$${\cal A} U= 0.$$ By the equation, taken at a formal level for the moment, it is clear that if a solution tends to an equilibrium \index{equilibrium point} and if ${\cal A}$ is continuous : $X\rightarrow Y$ for some Banach space Y having $X$ as a topologically imbedded subspace, the "velocity'' $U'(t) $ tends to $0$ in $Y$. If the trajectory $U$ is precompact in $X$, it will follow that this means some strong asymptotic flatness of $U(t)$ for $t$ large. Conversely, systems having this property do not necessarily enjoy the convergence property since trajectories might oscillate (slower and slower at infinity) between several stationary solutions. \\

A  well known convenient way to study the asymptotic behavior of solutions is to associate to the differential equation a semi-group $S(t)$ of (nonlinear) operators on some closed subset $Z$ of the Banach space $X$ , defined as follows: for each $t\ge 0$ and each $z\in X$ for which the initial value problem is well-posed, $S(t)z$ is the value at $t$ of the solution with initial value $z$. Since the initial value problem does not need to be well-posed for every $z\in X$, in general $Z$ will just be some closed set containing the trajectory 
$$ \Gamma(z) = \overline{\bigcup_{t\ge0} S(t)z}^X $$ For some results the consideration of $ \Gamma(z) $ will be enough, for some others (for instance stability properties) it will be preferable to take $Z$ as large as possible. The standard terminology used in the Literature for such semi-groups is ``Dynamical systems" and we shall adopt it. Since the operator ${\cal A}$ does not depend on time, both equation and dynamical system are called \index{autonomous} autonomous. According to the context, the word "trajectory" will mean either a solution of the equation $ u(t+s) = S(t) u(s) $ on the halfline, or the closure of its range. \\

The present work concerns dissipative \index{autonomous} autonomous systems. In the Literature the term ``dissipative" has been used in many different contexts. Here, dissipative refers to the existence of a scalar function $\Phi$ of the solution $U$ which is dissipated by the system, in the sense that it is nonincreasing:

$$ \forall s\ge 0, \,\,\forall t\ge s, \quad \Phi(U(t)) \le \Phi(U(s)).$$
If in addition $\Phi$ is coercive , this implies that $U(t)$ is bounded in $X$. The problem of asymptotic behavior becomes therefore natural. Such non-increasing functions of the solution play an important role in the theory of stability  initiated by \index{Liapunov} Liapunov. For this reason, in this text, they will be called Liapunov functions \index{Liapunov function} (resp. Liapunov  functionals if $X$ is a function space).\\

Let us now define more precisely the main theme of the present text.  The structure of trajectories to dynamical systems tends to become more and more complicated as the dimension of the ambient space $X$ increases. When $X= \R$, ${\cal A}$ is just a scalar function of the scalar variable $U$ and if ${\cal A}$ is locally Lipschitz, as a consequence of local uniqueness, no trajectory other than a stationary solution can cross the set of equilibria. As a consequence all bounded solutions are monotonic, hence convergent.  In higher dimensions, what remains true is that convergent trajectory have to converge to a stationary solution. But the equation $u"+ u = 0$ , which can be represented as a first order differential equation in $X= \R^2$ exhibits oscillatory solutions, and even when a strictly decreasing Liapunov \index{Liapunov} functions exists, two-dimensional systems can have some non-convergent trajectories. Our main purpose is to find sufficient conditions for convergence and exhibit some counterexamples showing the optimality of the convergence theorems. Finding sufficient conditions for convergence is a program which was initiated by S.  ${\L}$ojasiewicz when $X= \R^N$ and ${\cal A}= \nabla F $ with $F$ a real valued  function. By relying on the so-called ${\L}$ojasiewicz gradient inequality, \index{{\L}ojasiewicz gradient inequality} he showed that convergence of bounded solutions is insured whenever $F$ is analytic. From the point of view of a sufficient condition expressed in terms of regularity, this result is optimal: there are $C^\infty$ functions on  $X= \R^2$  for which the equation $ U'(t)+ \nabla F U(t)= 0$  has bounded non-convergent solutions. An explicit example was given by Palis \& De Melo in \cite{MR0669541}, and in this text we extend their example in such a way that any Gevrey regularity condition weaker than analytic appears unsufficient for convergence.\\

This text is divided in 11 chapters: the first 3 chapters contain some basic material   useful either to set properly the convergence question, or as a technical background for the proofs of the main results. In Chapter 4 we fix the main general concepts or notation concerning dynamical systems. In chapter 5 a general asymptotic stability criterion is given, generalizing the well known Liapunov \index{Liapunov} stability theorem (Liapunov's \index{Liapunov} first method) in a framework applicable to infinite dimensional dynamical systems and in the same vein, a finite-dimensional method used by R. Bellman to derive instability from linearized instability is applied to some infinite dimensional dynamical systems. Chapter 6 is devoted to the definition and main properties of a class of ``gradient-like  systems" \index{gradient-like} in which the question of convergence appears fairly natural. Chapter 7 concerns the general invariance principle and its connection with Liapunov's second method.  After  Chapter 8,  in which simple particular cases are treated by specific methods,  Chapter 9 and 10 are devoted to convergence theorems based on the \index{{\L}ojasiewicz gradient inequality} ${\L}$ojasiewicz gradient inequality, respectively in finite dimensions and infinite dimensional setting with applications to semilinear parabolic and hyperbolic problems in bounded domains. Chapter 11 is devoted to a somewhat informal description of more recent or technically more elaborate results which are too difficult to fall within the scope of a brief monograph. \\

We hope that this text may help the reader to build a bridge between the now classical ${\L}$ojasiewicz  convergence theorem and the more recent results on second order equations and infinite dimensional systems. 

\section{Some important lemmas}

The first lemma is classical and is recalled only for easy reference in the main text. 
\begin{lem}\label{Gronwall Lemma}  (Gronwall Lemma)\index{Gronwall Lemma} Let $T>0,\ \lambda\in L^1(0,T),\ \lambda \geq0
$ a.e. on $(0, T)$ and $C\geq0.$ Let $\varphi \in L^{\infty}(0,T),\   \varphi  \geq 0 $ a.e. on $(0, T)$, such
that $$ \varphi(t) \leq C + \int_0^t \lambda (s)\varphi(s)ds, \quad  \hbox{\rm a.e. on}\   (0, T)
$$ Then we have $$ \varphi(t) \leq C exp \Bigl(\int_0^t \lambda (s)ds\Bigr),
\quad \hbox{\rm a.e. on}\ (0, T)$$ \end{lem}
\begin{proof}[{\bf Proof.}] We set $$ \psi(t) = C + \int_0^t \lambda (s)\varphi(s)ds, \quad\forall t\in[0, T]$$ Then $ \psi $ is absolutely continuous, hence differentiable a.e. on (0,T), and we have
$$  \psi'(t) = \lambda(t)\varphi(t) \leq \lambda(t)\psi(t)   \quad \hbox{a.e.  on}\ (0, T).$$
 Consequently, a.e. on $(0,T)$ we find :
$$\frac{d}{dt}[\psi(t)exp\Bigl(-\int_0^t \lambda(s)ds\Bigr)]\leq 0.$$ Hence by
integrating
$$ \psi(t) \leq C exp \Bigl(\int_0^t \lambda (s)ds\Bigr),
\quad \forall t\in[0, T].$$
The result follows, since  $\varphi\leq
\psi$  a.e. on $(0, T)$
\end{proof}
 
 The next lemmas will be useful in the study of convergence and decay rates
 \begin{lem} \label{L1-conv} (cf. e.g.  \cite{MR0069338}. )
Let $X$ be a Banach space, $t_0\in\R$   and  $z\in C((t_0,\infty) ;X )$.   
Assume that the following conditions are satisfied
\begin{equation} \label{L1}
  z\in L^1((t_0,\infty) ;X )
\end{equation}
\begin{equation} \label{unifc}
\text{ z is uniformly continuous on } [t_0,\infty) \text{with values in X}.
\end{equation}
Then $$\lim_{t\rightarrow\infty} \| z(t) \|_X= 0  $$ 
\end{lem}
 \begin{proof}[{\bf Proof.}]  Let $\varepsilon>0$ be arbirary and let $\delta >0$ be such that  $$\sup_{ t\in[t_0,\infty), h\in[0, \delta]} \| z(t+h -z(t)) \|_X \le \varepsilon $$ Then we find easily   $$ \forall t\in[t_0,\infty),\quad \| z(t) \|_X \le \varepsilon + \frac{1}{\delta} \int_t^{t+\delta} \| z(s) \|_X ds . $$ implying $$ \limsup_{t\rightarrow\infty} \| z(t) \|_X \le \varepsilon $$ The conclusion follows immediately
 \end{proof}
\begin{lem} \label{strong}
Let $X$ be a Banach space, $t_0\in\R$   and  $u\in C^1((t_0,\infty) ;X )$ .   
Assume that there exists  $H\in C^1 ((t_0,\infty), \R )$, $\eta\in(0,1)$ and $c>0$ such that 
\begin{equation} \label{energiepositive2}
  H(t)>0 \ \text{ for all } \ t\geq t_0 .
\end{equation}
\begin{equation} \label{strongcond2}
- H'(t) \geq c \,  H(t)^{1-\eta} \, \| u' (t) \|_X \ \text{ for all } \ t\geq t_0 .
\end{equation}
Then  there exists $\varphi\in X$ such that $\displaystyle\lim_{t\to\infty} u(t)=\varphi$ in $X$.
\end{lem}
 \begin{proof}[{\bf Proof.}] 
By using  \eqref{strongcond2},   we get for all $t\geq t_0$
\begin{eqnarray}
\label{e1} - \frac{d}{dt} H(t)^\eta & = & -\eta H'(t) H(t)^{\eta -1}  \\
\nonumber & \geq & {c \eta } \, \|  u' (t)\|_X .
\end{eqnarray}
By integrating this last inequality over $(t_0,T)$, we obtain
\begin{equation}\label{ineqNouvel1}\int_{t_0}^{T} \|  u' (t)\|_X \,dt\leq \frac{H(t_0)^\eta}{c \eta }.
\end{equation}
This implies $ u' \in L^1 ((t_0,\infty) ;X )$. By Cauchy's criterion, $\displaystyle\lim_{t\to\infty} u(t)$ exists in $X$. 
 \end{proof}
\begin{lem}\label{LemmaZelenyakExp} Let $T>0$, let $p$ be a nonnegative square integrable function on $[0, T) $. Assume that there  exists two constants $\gamma>0$ and $a>0$ such that
 $$\forall t\in [0,T], \quad \int_t^{T}p^2(s) ds \leq a e^{-\gamma t}. $$
Then setting $b := {\rm e}^{\gamma /2}/({\rm e}^{\gamma / 2} - 1)$,  for all  $0\leq t\leq\tau\leq T$ we have:
$$J(t,\tau):=\int_t^{\tau}p(s) ds \leq \sqrt{a} be^{-{{\gamma
t}\over2}}.$$
\end{lem}
\begin{proof}[{\bf Proof.}]
Assume first that $\tau -t \leq1.$ Then we have
$$
J(t,\tau) \leq \sqrt{\tau-t} \sqrt{\int_t^{\tau}p^2(s\,ds}\\\leq\sqrt{a}e^{-{{\gamma t}\over2}}.
$$
If $\tau - t\geq1$ we reason as follows. Let $N$ be the integer part of
$\tau-t$, we get
\begin{eqnarray*}J(t,\tau) &\leq& \sum_{i=0}^{N-1} \int_{t+i}^{t+i+1} p(s)\,ds +
      \int_{t+N}^{\tau} p(s)\,ds\\
&\leq&\sum_{i=0}^{N-1}
\sqrt{a}e^{-{{\gamma(t+i)}\over2}}+\sqrt{a}e^{-{{\gamma(t+N)}\over2}}\\
&\leq& \sqrt{a}{{e^{\gamma\over2}}\over{e^{\gamma\over2}-1}}e^{-{{\gamma
t}\over2}}.
\end{eqnarray*}
\end{proof}

\begin{lem}\label{LemmaZelenyakPol} Let $p$ be a nonnegative square integrable function on $[1, \infty) $. Assume that for some  $\alpha>0$ and a
constant $K>0$, we have
$$\forall t\geq 1\quad  \int_t^{2t}p^2(s)ds \leq K t^{-2\alpha -1}$$
Then for all $\tau \geq t\geq 1$ we have:
$$\int_t^{\tau}p(s) \,ds
\leq {\sqrt{K}\over{1-2^{-\alpha}}}\,\, t^{-\alpha } .$$
\end{lem}
\begin{proof}[{\bf Proof.}]  By Cauchy-Schwarz inequality, for all $t\geq 1$ we may write:
$$\int_t^{2t}p(s) ds \leq \sqrt{t}\, (K t^{-2\alpha -1})^{1/2}
= \sqrt{K}\, t^{-\alpha },$$
hence
$$\int_t^{\tau}p(s) \,ds\leq\int_t^{\infty}p(s)\,ds = \sum_{k=0}^{\infty } \int_{2^k
t}^{2^{k+1}t}p(s)\, ds \leq  \sqrt{K}\,
\sum _{k = 0}^{\infty}( 2^kt) ^{-\alpha } =
{\sqrt{K}\over{1-2^{-\alpha}}}\,\, t^{-\alpha }$$
\end{proof}

Finally, in the application of the {\L}ojasiewicz gradient inequality \index{{\L}ojasiewicz gradient inequality} to convergence results, the following topological reduction principle will play an important role. 
\begin{lem}\label{InegLojsABSRAITLemme} Let $W$ and $X$ be two Banach spaces. Let  $U\subset W$ be open and $E:U\longrightarrow \R$ and ${\cal G}: U\longrightarrow X$ be two continuous functions. We assume that for all $a\in U$ such that ${\cG}(a)=0$, there exist 
$\sigma_a>0,$ $\theta(a)\in(0,1)$ and $c(a)>0$ 
\begin{equation}\label{InegLojsABSRAIT}
\Vert {\cG}(u)\Vert_X\geq c(a)\vert E(u)-E(a)\vert^{1-\theta(a)},\ \forall u:\ \Vert u-a\Vert_W<\sigma_a.\end{equation}
Let $\Gamma$ be a compact and connected subset of ${\cG}^{-1}\{0\}$. Then we have
\begin{enumerate}
\item[(1)]   $E$ assumes a constant value on $\Gamma$. We denote by ${\bar E}$ the common value of $E(a)$, $a\in \Gamma$. 
\item[(2)]  There exist $\sigma>0$, $\theta\in(0,1)$ and $c>0$ such that 
$$dist(u,\Gamma)<\sigma\Longrightarrow \Vert {\cG}(u)\Vert_X\geq c\vert E(u)-{\bar E}\vert^{1-\theta} $$  
\end{enumerate}
\end{lem}
\begin{proof}[{\bf Proof.}] By continuity of $E$ we can always assume that $\sigma_a$ is replaced by a possibly smaller number so  that $\vert E(u)-E(a)\vert \leq 1$ for all $u$ such that $\Vert u-a\Vert_W<\sigma_a.$
Let $a\in \Gamma$ and
$$K=\{b\in \Gamma/ \ E(b)=E(a)\}.$$
It follows from \eqref{InegLojsABSRAIT} that $K$ is an open subset of $\Gamma$  which  is obviously closed by continuity and since $\Gamma$ is connected by hypothesis we have $K=\Gamma$.\\
On the other hand, since $\Gamma$ is compact, there exist  $a_1,\cdots,a_p\in \Gamma$ such that
 $$\Gamma\subset\displaystyle\bigcup_{i=1}^{p} B(a_i,\frac{\sigma_{a_i}}{2}).$$
The result  follows with $\sigma=\frac12\inf\sigma_{a_i}$, $c=\inf c(a_i)$ and $\theta=\inf \theta(a_i)$.

\end{proof}

\section{Semi-Fredholm operators \index{semi-Fredholm}} Let $E$, $F$ be two Banach spaces and  $A:E\longrightarrow F$ be a linear operator. We denote by $N(A)$ and $R(A)$ the null space and the range of $A$, repectively.
\begin{defn}  A  bounded linear operator $A\in L(E,F)$   is said to be semi-Fredholm\index{semi-Fredholm} if 
\begin{enumerate}
\item[(1)] $N(A)$ is finite dimensional,
\item[(2)]  $R(A)$ is closed.
\end{enumerate}
We denote by $SF(E,F)$ the set of all semi-Fredholm\index{semi-Fredholm} operators from $E$ to $F$.\end{defn}
\begin{rem}\label{OperatorSemFredholm}{\rm The fact that $N(A)$ is finite dimensional implies that there exists a closed subspace $X$ of $E$ such that $E=N(A)\bigoplus X$ (cf \cite{MR2759829} p. 38). Moreover $R(A)=A(X)$ is a Banach space when equipped with the norm $\Vert\cdot\Vert_F$.  }\end{rem}

\begin{thm}\label{PropInegFredholm} Let $A\in L(E,F)$ and assume that $N(A)$ is finite dimensional. Then  we have $A\in SF(E,F)$ if and only if  
\begin{equation}\label{InegOperFredholm}\exists\rho>0,\ 
\forall u\in X   \quad \Vert A u  \Vert_{F }\ge \rho \Vert u\Vert_E.
\end{equation} 
\end{thm}
\begin{proof}[{\bf Proof.}] \eqref{InegOperFredholm} implies that $R(A)$ is closed. In fact, let $(f_n)=(Au_n)$ be such that $f_n\longrightarrow f$ in $F$. Let $(x_n)$ and $(y_n)$ be such that $u_n=x_n+y_n$ with $(x_n)\subset X$ and $(y_n)\subset N(A)$. So $f_n=Ax_n$. Then the inequality $\Vert x_n-x_m\Vert_E\leq \frac{1}{\rho}\Vert f_n-f_m\Vert_F$ implies that $(x_n)$ is a Cauchy sequence, hence converges. Let $x$ be the limit. We have $Ax_n\longrightarrow Ax$ so $f=Ax$. \\
Conversely, $R(A)$ is a Banach space and $C:=A_{/X}:X\longrightarrow R(A)$ is bijective and continuous, by Banach's theorem we get that $C^{-1}$ is continuous and \eqref{InegOperFredholm} follows.
\end{proof}
\begin{rem}\label{IsomorphismFredholm}{\rm If $A:E\longrightarrow F$ is a topological isomorphism, then $A\in SF(E,F)$ with $N(A)=\{0\}$. Conversely, as a consequence of Banach's theorem, if $A\in SF(E,F)$ with $N(A)=\{0\}$ , then $A:E\longrightarrow R(A)$ is a topological isomorphism. }\end{rem}
\begin{thm}\label{FredholmCptFred} Let $A\in  SF(E,F)$ and $G\in L(E,F)$. Assume that $G$ is \index{compact operator}compact, then $A+G\in SF(E,F)$.\end{thm}
\begin{proof}[{\bf Proof.}] We divide the proof into 3 steps :\\
{\it Step 1 : } If $(u_n)\subset E$ with $\Vert u_n\Vert\leq 1$ and $(A+G)(u_n)\longrightarrow 0$, then $(u_n)$ has a strongly convergent subsequence in $E$. Indeed we can assume $Gu_n\longrightarrow g\in F$. Let $u_n=x_n+y_n$, $x_n\in X$, $y_n\in N(A)$ where $X$ is as in the remark \ref{OperatorSemFredholm}. Since $A u_n=Ax_n\longrightarrow - g$, $(x_n)$ is convergent in $E$. Then $(y_n)$ is bounded in $N(A)$, since $\dim N(A)<\infty$ we can assume that $y_n\longrightarrow y$ in $E$ with $y\in N(A)$. In particular $u_n=x_n+y_n$ is convergent  in $E$.\\
{\it Step 2 : } Let $(u_n)\subset N(A+G)$ with $\Vert u_n\Vert\leq 1$. By step 1, $(u_n)$ is precompact in $E$, hence the unit ball of $N(A+G)$ is precompact and consequently $\dim N(A+G)<\infty$.\\
{\it Step 3 : } Let $Y$ be a Banach space such that  $E=N(A+G)\bigoplus Y$. Assuming $R(A+G)$ not closed, then by Theorem\ref{PropInegFredholm} we can find $y_n\in Y$ with $\Vert y_n\Vert=1$ and $(A+G)y_n\longrightarrow 0$. By step 1, up to a subsequence we can deduce $y_n\longrightarrow y$ in $E$. We immediately find $\Vert y\Vert_E$ and $y\in Y$. Hence since $(A+G)y_n\longrightarrow 0$ we have $y\in N(A+G)$. Since $N(A+G)\cap Y=\{0\}$, we end up with a contradiction since $y\in N(A+G)\cap Y$ and $\Vert y\Vert_E=1$.
\end{proof}
For the next corollary, we consider two real Hilbert spaces $V, H$  where  $V\subset H$ with continuous and dense imbedding and $H'$, the topological dual of $H$ is identified with $H$, therefore $$ V\subset H = H' \subset  V'$$ with continuous and dense imbeddings. 
\begin{cor}\label{SFredplusProj} Let $A\in SF(V,V')$ and assume that $A$ is symmetric. Then $A+P:V\longrightarrow V'$ is an isomorphism where $P:V\longrightarrow N(A)$ is the projection in the sense of $H$.
\end{cor} 
\begin{proof}[{\bf Proof.}] First we have $N(A+P)=\{0\}$. Indeed if $Au+Pu=0$, we have $Au=-Pu\in N(A)$, then $Au\in N(A)\cap R(A)=\{0\}$, so $Au=0$, hence $u=Pu=-Au=0$.\\ On the other hand, since $A\in SF(V,V')$, $\dim N(A)<\infty$ and then $P$ is \index{compact operator} compact. By Theorem \ref{FredholmCptFred} $A+P\in SF(V,V')$, then $R(A+P)$ is closed. Now since $A+P$ is symmetric and $N(A+P)=\{0\}$ then $R(A+P)$ is dense in $V'$, hence $R(A+P)=V'$. By Banach's theorem we get that $(A+P)^{-1}\in L(V',V).$
\end{proof} 
\begin{example}\label{ExempleFredholm}{\rm Let $\Omega$ be a bounded and regular domain of $\R^N$,  $V=H^1_0(\Omega)$
$$A=-\Delta +p(x) I, \quad p\in L^\infty(\Omega)$$
$G:=p(x) I:V\longrightarrow V'$ is \index{compact operator} compact.
$-\Delta\in \hbox{Isom}\, (V,V')$ then by Theorem \ref{FredholmCptFred} $A\in SF(V,V')$. Corollary \ref{SFredplusProj} implies that $A+P\in \hbox{Isom}\,(V,V').$
}
\end{example}

 \section{Analytic \index{analytic} maps} In this section, we introduce a general notion of real analyticity valid in the Banach space framework which will be essential for the proper formulation of  many convergence results applicable to P.D.E. One of the difficulties we encounter here is that the good properties of complex analyticity cannot be used and all the proofs have to be done in the real analytic framework. For example, in this framework the result on composition of analytic maps is not so trivial as in the complex framework and its proof is generally skipped even in the best reference books. Here we shall give a complete argument relying on the majorant series technique of Weirstrass.
 
 \subsection {Definitions and general properties\label{Def-Ana}}
\begin{defn}\label{defFtAnalytique}  Let $X$, $Y$ be two real Banach space and $a\in X$. Let $U$ be an open neighborhood of $a$ in $X$. A map $f:U\longrightarrow Y$ is called analytic \index{analytic}  at $a$ if there exists $r>0$ and a sequence of $n-$linear, continuous, symmetric maps $(M_n)_{n\in \N}$ fulfilling the following conditions
\begin{itemize}
\item[(1)] $\displaystyle \sum_{n\in \N}\Vert M_n\Vert_{{\cal L}_n(X,Y)} r^n<\infty$ where $$\Vert M_n\Vert_{{\cal L}_n(X,Y)}=\sup\{\Vert M_n(x_1,x_2,\cdots, x_n)\Vert_Y,\ \sup_i\Vert x_i\Vert_X\leq 1\}.$$
\item[(2)] ${\bar B}(a,r)\subset U$.
\item[(3)] $\forall h\in {\bar B}(0,r),\ f(a+h)=f(a)+\displaystyle \sum_{n\geq1}  M_n(h^{(n)})$ where $h^{(n)}=\underbrace{(h,\cdots,h)}_{n \hbox{ times}}$.
\end{itemize}
\end{defn}
\begin{rem} {\rm Under the previous definition, it is not difficult to check that
\begin{itemize}
\item[-] $\forall b\in B(a,r)$, $f$ is analytic \index{analytic} at $b$.
\item[-] $f\in C^{\infty}(B(a,r),Y)$ with $D^nf(a)=n!M_n$.
\item[-] A finite linear combination of analytic \index{analytic} maps at $a$ is again analytic \index{analytic} at $a$.
\end{itemize}}
\end{rem}
\begin{defn} $f$ is analytic \index{analytic} on the open set $U$ if $f$ is analytic \index{analytic} at every point of $U$.
\end{defn}
\begin{example} \label{Poly=ana}  {\rm It is clear from the definitions that any bounded linear operator, any continuous quadratic form and more generally any finite linear combination of restrictions to the diagonal of continuous $k$-multilinear maps: $X^{k}\rightarrow Y$ (usually called a  polynomial map) is analytic \index{analytic} on the whole space $X$. } \end{example}
\begin{prop}\label{Df} Let $f\in C^{1}(U,Y)$. The following properties are equivalent
\begin{itemize}
\item[(1)] $f:U\longrightarrow Y$ is analytic \index{analytic};
\item[(2)] $Df:U\longrightarrow {\cal L}(X,Y)$ is analytic \index{analytic}.
\end{itemize}
Moreover if
$$f(a+h)=f(a)+\displaystyle \sum_{n\geq1}  M_n(h^{(n)})$$ is the expansion of $f(a+h)$ for all $h$ in the closed ball ${\bar B}(0,r)\subset U - a $, 
then
$$Df(a+h)=M_1+\displaystyle \sum_{n\geq 2}  n M_n(h^{(n-1)},\cdot) $$ is the expansion of $Df(a+h)$ for all $h$ in the open ball $B(0,r).$
\end{prop}
\begin{proof}[{\bf Proof.}]  First let us explain the meaning of the formula for the derivative. It involves an infinite sum of expressions of the form $$  n M_n(h^{(n-1)},\cdot).$$ Indeed, since $Df(a+h)$ is for all vectors $h$ an element of ${\cal L}(X,Y)$, the formula really means 
$$\forall \xi\in X, \quad Df(a+h)(\xi)=M_1(\xi) +\displaystyle \sum_{n\geq 2}  n M_n(h^{(n-1)},\xi)$$ and for any $n\ge2$ fixed we must identify $  n M_n(h^{(n-1)},\cdot) $ as the trace on the diagonal of $X^{n-1}$  of an $n-1$-linear symmetric continuous map with values in ${\cal L}(X,Y)$. The corresponding map is just 
$$ K_{n-1} (x_1,..., x_{n-1})(\xi) = n M_n(x_1,..., x_{n-1},\xi).$$
Assuming 1), Let us consider $a$ and $r>0$ with ${\bar B}(0,r)\subset U - a $. The expression of the norms of 
$ K_{n-1}$ in the space of $n-1$- linear symmetric continuous map with values in ${\cal L}(X,Y)$ shows that the formal series given by $$\forall \xi\in X, \quad Df(a+h)(\xi)=M_1(\xi) +\displaystyle \sum_{n\geq 2}  K_{n-1}(h^{(n-1)},\xi)$$ satisfies $\displaystyle \sum_{n\in \N}\Vert K_n\Vert_{{\cal L}_n(X,{\cal L}(X,Y))} r'^n<\infty$ for any $r'\in (0, r)$. The summation formula for the derivative is now obvious when the expansion is finite. The general case is more delicate and is in fact related to the formula permitting to recover $f$ from the knowledge of $Df$. This formula:  $$ f(a+h) = f(a) + \int _0^1Df(a+s h)(h) ds $$ is classical and valid for any $C^1$ function $f$. When we substitute the expansion of $Df$ in this formula, the summability of its terms transfers easily to yield the desired expansion for $f$. We skip the details which are classical for this part of the argument.
\end{proof}
\subsection{Composition of analytic \index{analytic} maps}
Let $Z$ be a Banach space, $V$ be an open neighborhood of $f(a)$ and $g:V\longrightarrow Z$   be analytic \index{analytic} at $f(a)$. This means that for some $\rho>0$, we have
$$g(f(a)+k)=g(f(a))+\sum_{m\geq1}  P_m(k^{(m)})$$
whenever $\Vert k\Vert_F\leq \rho$ and $\displaystyle \sum_{m\in \N}\Vert P_m\Vert_{{\cal L}_m(X,Z)} \rho^m<\infty$.
\begin{thm}\label{ComposAnal}  The map $g\circ f$ is analytic \index{analytic} at $a$ with values in $Z$. More precisely, setting
$$R_d(h^{(d)})=\sum_{m\leq d}\displaystyle \sum_{ \sum_{j=1}^m n_j=d} P_m\left(M_{n_1}(h^{(n_1)}),\cdots,M_{n_m}(h^{(n_m)})\right)$$
(the sum is finite for any $d$) we have
\begin{equation}\label{EqCompoAnal}
\sum_{d\geq1}\Vert R_d\Vert_{{\cal L}_d(X,Z)} \sigma^d<\infty
\end{equation}
as soon as  $$\sum\Vert M_n\Vert_{{\cal L}_n(X,Y)}\sigma^n\leq \rho  $$
and 
$$g\circ f(a+h)=g\circ f(a)+\displaystyle \sum_{d\geq1}  R_d(h^{(d)}),\ \forall h, \ \Vert h\Vert_X\leq\sigma.$$
\end{thm}
\begin{proof}[{\bf Proof.}] We have the obvious estimate :
$$\Vert R_d\Vert_{{\cal L}_d(X,Z)}\leq \sum_{m\leq d}\Vert P_m \Vert_{{\cal L}_m(Y,Z)}\sum_{\vert \mu\vert=d}\Vert M_{n_1}\Vert\cdots \Vert M_{n_m}\Vert$$
where $\mu=(n_1,\cdots,n_m)$, $\vert \mu\vert=n_1+\cdots+n_m$ and $\Vert M_{n_i}\Vert=\Vert M_{n_i}\Vert_{{\cal L}_{n_i}(X,Y)}$. Indeed
$$R_d(h_1,\cdots,h_d)=\sum_{m\leq d}\sum_{\vert \mu\vert=d}   P_m ( M_{n_1}(h_1\cdots,h_{n_1}),\cdots,  M_{n_m}(h_{n_1+\cdots n_{m-1}+1},\cdots,h_{d}) ). $$
Therefore
\begin{eqnarray*}
\sum_{d\geq1}\Vert R_d\Vert_{{\cal L}_d(X,Z)} \sigma^d &\leq&\sum_{1\leq m}\sum_{\leq d} \Vert P_m \Vert\sum \Vert M_{n_1}\Vert\cdots\Vert M_{n_m}\Vert\sigma^d   \\
&=&\sum\sum\sum \Vert P_m \Vert \Vert M_{n_1}\Vert \sigma^{n_1}\cdots\Vert M_{n_m}\Vert \sigma^{n_m} \\ 
&=&\sum_m \Vert P_m \Vert  \sum_{d\geq m, \, \vert \mu\vert=d}  \Vert M_{n_1}\Vert \sigma^{n_1}\cdots\Vert M_{n_m}\Vert \sigma^{n_m} \\
&\leq&\sum_m \Vert P_m \Vert \left( \sum  \Vert M_{n}\Vert \sigma^{n} \right)^m  .
\end{eqnarray*}
Then  \eqref{EqCompoAnal} follows. Concerning the convergence of the series to $g\circ f$, we notice that
$$(g\circ f)(a+h)-(g\circ f)(a)=\sum_{m\geq1}  P_m((f(a+h)-f(a))^{(m)})$$
Hence
\begin{eqnarray*}& &\Vert (g\circ f)(a+h)-(g\circ f)(a)-\sum_{m=1}^{M}  P_m((f(a+h)-f(a))^{(m)}) \Vert_Z\\
&\leq&\sum_{m\geq M+1}\Vert P_m\Vert\left( \sum  \Vert M_{n}\Vert \sigma^{n} \right)^m\\
&<&\varepsilon \ \hbox{ for }M\geq M(\varepsilon).
\end{eqnarray*}
Then for $M\geq 1$ fixed
$$\sum_{m=1}^{M}  P_m((f(a+h)-f(a))^{(m)}=\sum_{d\geq 1}\sum_{m=1}^{M} Q_\mu ((h)^{(d)})$$
with $Q_\mu ((h)^d)=P_m(M_{\mu_1}((h)^{(\mu_1)}),\cdots,M_{\mu_m}((h)^{(\mu_m)})$.
\begin{eqnarray*}& &\Vert\sum_{m=1}^{M}  P_m((f(a+h)-f(a))^{(m)}-\sum_{d=1}^{M}\sum_{m=1}^{M} Q_\mu ((h)^{(d)})\Vert\\
&\leq& \sum_{m=1}^{M}\sum_{\vert \mu\vert=d\geq M+1}\Vert Q_\mu ((h)^{(d)})\Vert\to0 \hbox{ as }M\to\infty.
\end{eqnarray*}
Finally
$$\Vert (g\circ f)(a+h)-(g\circ f)(a)-\sum_{d=1}^{M}\sum_{m=1}^{M}\sum_{\mu\vert=d} Q_\mu ((h)^{(d)})\Vert\leq2\varepsilon$$
for $M$ large. But 
$$\sum_{d=1}^{M}\sum_{m=1}^{M}\sum_{\mu\vert=d} Q_\mu ((h)^{(d)})=\sum_{d=1}^{M} R_d ((h)^{(d)})$$ since
$\displaystyle \sum_{m=1}^{M}\sum_{\mu\vert=d} Q_\mu=R_d$ for all $d\leq M$.
\end{proof} 

\subsection{Nemytskii type operators \index{Nemytskii operator}on a Banach algebra } Let $\cal A$ be a real Banach algebra and $f$ be a real analytic \index{analytic} function in a neighborhood of $0$, which means that for some open subset $U$ of $\R$ containing $0$ we have $f \in C^\infty (U, \R)$ and for some positive constants $M, K$  $$ \forall n\in \N, \quad |f^{(n)}(0)| \le M K^n n!$$ It is clear that for any $n\in \N$ the map $u\rightarrow u^n$ is the restriction to the diagonal of ${\cal A}^n$ of the continuous n-linear map $$ U = (u_1, ..u_n) \rightarrow \prod_{i}^n u_j $$ It follow that the map 
$$ {\cal F} (u) = \sum _{n=0} ^\infty \frac{f^{(n)}(0)}{n!} u^n$$ is analytic \index{analytic} in the open ball $B_0 = B(0, \frac{1}{K})$ in the sense of Subsection \ref{Def-Ana}. 
This map will be called the Nemytskii \index{Nemytskii operator} type operator associated to $f$ on the Banach algebra $\cal A$. 
\begin{example}\label{Nemytskiianalytic}{\rm   Let us consider the special case ${\cal A} = L^\infty(S)$ where $S$ is any positively measured space. Then for any $f$ as above the operator defined by  $$ {\N}_{f} (u)(s) = f(u(s))= \sum _{n=0}^\infty \frac{f^{(n)}(0)}{n!} u(s)^n$$ for all $u\in B(0, \frac{1}{K}) \subset L^\infty(S)$ and almost everywhere in $S$ is usually called the Nemytskii \index{Nemytskii operator} operator on $L^\infty(S)$ associated to $f$ and is an analytic map in a ball centered at $0$. The same holds true if we replace $L^\infty(S)$ by the set of continous bounded functions on a topological space $Z$ or more generally any Banach sub-algebra of it. }\end{example}
\begin{rem} {\rm (i) The Nemytskii \index{Nemytskii operator} operator $ {\N}_{f} (u)(s) = f(u(s))$ makes sense in other contexts, for instance from a Lebesgue space into another assuming some growth restrictions of the generating function $f$. 
\par\noindent (ii) We shall use this operator exclusively in the case where $f$ is in fact an entire function, i.e. $ K$ can be taken arbitrarily small.
\par\noindent (iii) Moreover, in the applications we shall usually need some growth restrictions on $f$ or even its first derivative.
\par\noindent (iv) In our applications to convergence, $ {\N}_{f} (u)(s) = f(u(s))$ will usually appear as the derivative of a potential function $ G(u) = \int _S F(u(s)ds $ where $F$ is a primitive of $f$.
}
\end{rem}

\subsection{Inverting  analytic \index{analytic}  maps}  Let  $X$, $Y$ be two real Banach space and $a\in X$. Let $U$ be an open neighborhood of $a$ in $X$ and  $f\in C^1(U, Y).$  The well known inverse map theorem says that if $Df(a)\in \hbox{Isom}\,(X, Y)$, there exists a possibly smaller neigborhood $W$ of $a$ in $X$ such that $f(W)$ is open in $Y$ and $f: W\rightarrow f(W)$ is a $C^1$-diffeomorphism. Moreover we have the formula $$\forall y\in f(W), \quad D (f^{-1} )(y) = [Df(f^{-1}(y))]^{-1}$$ We note that in order for $f$ to be a diffeomorphism, we need the existence of a linear topological isomorphism between $X$ and $Y$, namely $L= Df(a)$, so that diffeomorphisms can be reduced to the case $X= Y$ by replacing the general function $f$ by the "operator`` $g= L^{-1}\circ f $. By combining  \eqref{Df} with the fact that the map $ T\rightarrow T^{-1}$  is analytic \index{analytic} on the open set $ \hbox{Isom}\,(X, X)\subset {\cal L}(X, X)$, it is easy to prove the following 
\begin{thm}\label{InvertAnal} \index{Inverting  analytic maps} Giving a function $f\in C^1(U, Y)$ which is analytic \index{analytic} at $a\in U$, if  $Df(a)\in \hbox{Isom}\,(X, Y)$, the inverse map $f^{-1}$ is analytic \index{analytic} at $f(a)$. 
\end{thm}\begin{proof}[{\bf Proof.}]  By construction, $g:V \rightarrow X$ is analytic with $V$ an open ball of $X$ contained in  $U$ and centered at $a$, so that we may assume $V= U$. As a consequence of 
Proposition  \ref{Df}, $Dg$ is analytic \index{analytic} :  $V \rightarrow {\cal L}(X)$  and we have $Dg(a) = Id_{\cal L}(X).$ Then 
$D g^{-1}(x)  = (D g)^{-1} \circ g^{-1}(x)$ throughout $g(V)$, so that $D g^{-1}$ appears as a composition of 3 analytic \index{analytic} maps by reducing if necessary $V$ to a small ball around $a$ in which $Dg$ is sufficiently close to $ Id_{\cal L}(X)$ in the norm of $ {\cal L}(X)$ to use the formula $ (I- \tau)^{-1}  = \sum \tau^{n} $  where $ \tau(y) =  Id_{\cal L}(X) - Dg(y) $ . Finally by using once more Proposition  \ref{Df}, the gradient $D g^{-1}$ is lifted to $g^{-1}$ which is therefore also analytic \index{analytic}. The details are essentially classical and left to the reader.
\end{proof}
\chapter[Background results on Evolution Equations]{Background results on Evolution Equations}\label{chapter2}

\section[unbounded linear operators]{ Elements of functional analysis. Examples of unbounded
operators }

    Throughout this paragraph, $X$ denotes a real Banach space. The norm of $X$ is denoted by $\Vert\, \Vert$. The results will generally be stated without proof. For the proofs we refer to the classical literature on functional analysis, cf. e.g. \cite{{MR2759829},{MR0239384}}

\subsection{Unbounded Operators on  $X$}

\begin{defn} A linear
operator on  $X$ is a pair  $(D,A)$, where $D$ is a linear subspace of $X,$ and
$A:D\rightarrow X$ is a linear mapping. We say that $A $ is bounded  if $\Vert Au
\Vert$ remains bounded for $u\in \{x\in D, \Vert x\Vert\leq1 \}.$ Otherwise, $A$ is called
unbounded.
\end{defn}

\begin{rem} {\rm  If $A$ is bounded, then $A$ is the restriction to $D$ of some
operator $ \tilde A\in L(Y,X),$ where $Y$ is a closed linear subspace of $X$ containing
$D$. On the other hand if $A$ is unbounded, then there exists no operator $ \tilde A\in
L(Y,X)$ with $Y$ a closed linear subspace of $X$ and $D\subset Y$ such that $\tilde
A\vert D=A.$}\end{rem}

\begin{defn} If $(D, A)$ is a linear operator on  $X$, the graph of $A$ and
the range of $A$ are the linear subspaces $G(A)$ and $R(A)$ of $X$ defined by
    $$G(A) = \{(u, f)\in X\times X, u\in D, f=Au \}  \quad \hbox   { and}\quad  R(A) =
A(D).$$
\end{defn}

As it is usual, we shall frequently call the pair $(D,A)$ as "$A$ with $D(A)=D$ ".
However one must always keep in mind that when we define a linear operator, it is
absolutely crucial to specify the domain.

\begin{defn}  A linear operator $A$ on  $X$ is called dissipative \index{dissipative} if we have
$$\forall u\in D(A), \forall \lambda >0, \Vert u - \lambda Au\Vert \geq \Vert u\Vert.$$
$A$ is called m-dissipative \index{m-dissipative} if $A$ is dissipative \index{dissipative} and for all $\lambda >0$, the operator $I-\lambda A$ is onto, i.e
$$\forall f\in X, \exists u\in D(A), u - \lambda Au = f.$$
\end{defn}

\begin{prop}   Let $A$ be a linear dissipative \index{dissipative} operator on $X$. Then
the following properties are  equivalent.
\begin{itemize}
\item[(i)] $A$ is  m-dissipative \index{m-dissipative} on  $X$.
\item[(ii)] There exists  $\lambda_0>0$ such that for each $f\in X$, there exists $u\in
D(A)$ with : $u-\lambda_0 Au = f.$
\end{itemize}
\end{prop}

\subsection{Case where $X$ is a  Hilbert space}

Let us denote by  $\langle\cdot,\cdot\rangle$ the inner product of $X$. If $A$ is a linear densely defined
operator on $X$, the formula
$$  G(A^*) = \{(v,g)\in X\times X, \, \forall(u,f)\in G(A), \langle g,u\rangle=\langle v,f\rangle \}$$
defines a linear operator $ A^* $ (the adjoint of $A$), with domain
        $$  D(A^*) = \{v\in X,\,  \exists C<\infty, \vert \langle Au,v\rangle \vert \leq C \Vert u\Vert, \forall
u\in D(A)\}$$ and such that:    $\langle A^*v,u\rangle  = \langle v,Au\rangle , \forall u\in D(A), \forall v\in D(A^*).$
Indeed the linear form $u\rightarrow \langle v,Au\rangle $ defined on $D(A$) for each $v\in D(A^*),$
has a unique extension $\varphi \in X'\equiv X,$ and we set: $\varphi=A^*v $.

Obviously, G(A*) is always closed. Moreover, it is immediate to check that if
$B\in L(X),$ then $(A+B)^*=A^*+B^*.$\\
In the Hilbert space setting , m-dissipative \index{m-dissipative} operators can be characterised rather easily . First the following proposition follows from elementary duality properties

\begin{prop} A linear  operator  $A$  on $X$ is dissipative \index{dissipative} in  $X$
if and only if $$\forall u\in D(A),  \langle Au,u\rangle \leq 0. $$
In addition if A is  m-dissipative \index{m-dissipative} on  $X$, then $D(A)$ is everywhere dense in $X$.
\end{prop}
The following result is often useful, especially the two corollaries:
\begin{prop}  Let $A$ be a linear dissipative \index{dissipative} operator on $X$, with dense domain. Then $A$ is m-dissipative \index{m-dissipative} if, and only if  $A^*$ is dissipative \index{dissipative} and $G(A)$ is closed.
\end{prop}

\begin{cor}  If $A$ is self-adjoint \index{self-adjoint} in $X$, in the sense that $D(A)=D(A^*)$ and $A^*u=Au,$ for all $u\in D(A),$ and if  $A\leq0$ (which means $\langle Au,u\rangle \leq0$ for all $u\in D(A),$ Then $A$ is m-dissipative.\index{m-dissipative}
\end{cor}
\begin{cor} If $A $ is skew-adjoint \index{skew-adjoint} in $X$, in the sense that
$D(A)=D(A^*)$ and $A^*u=-Au,$ for all $u\in D(A),$ then $A$ and $-A$ are both
m-dissipative.\index{m-dissipative}
\end{cor}
\subsection{Examples in the theory of PDE}

In this paragraph, we recall some basic facts from the linear theory of partial
differential equations which shall be used throughout the text. The definitions of
Sobolev spaces and the associated norms are the standard ones as can be found in \cite{MR0450957}. In particular, $\Omega$ being an open set in $ {\mathbb{R}}^N$, we shall use the
spaces

$$  H^m(\Omega) = \{ u\in L^2(\Omega), D_j u\in L^2(\Omega), \, \forall j :   \vert
j\vert\leq m
\},
$$endowed with the obvious inner product

    $H^m_0(\Omega)$ = completion of $C^\infty $ functions with compact support in $\Omega$ for
the $   H^m $ norm.\\
We recall the Poincar\'e inequality in $H^1_0(\Omega)$ when $\Omega$ is bounded :
$$  \forall w\in H^1_0(\Omega),   \int_\Omega \vert\nabla w \vert^2 dx \geq
\lambda_ 1
\int_\Omega \vert w \vert^2dx,$$ where $\lambda_ 1= \lambda_ 1(\Omega)$ is the first
eigenvalue of $(-\Delta)$ in
$H^1_0(\Omega)$ . We are now in a position to describe our basic examples. \medskip

\begin{example} \label{HeatHilbert} {\bf: The Laplacian in an open set of ${\mathbb{R}}^N \, :\,\, L^2 $
theory.}\end{example}

\noindent Let $\Omega$ be any open set in  ${\mathbb{R}}^N,$ and $H=L^2(\Omega)$. We define
the linear operator $B$ on $H$ by
$$  D(B) = \{u\in H^1_0(\Omega),\, \Delta u\in L^2(\Omega)\},$$
$$  Bu = \Delta u,\, \forall u\in D(B).$$
Then $B$ is m-dissipative \index{m-dissipative} and densely defined. More precisely $B$ is self-adjoint \index{self-adjoint} and
$B\leq0.$ In addition if the boundary of $\Omega$ is bounded and $C^2$, then
$$D(B)=H^2(\Omega)\cap H^1_0(\Omega),$$ algebraically and topologically.
\begin{example}\label{exempC0}    {\bf : The Laplacian in an open set of ${\mathbb{R}}^N\, :\,\, C^0
$ theory.}\end{example}

\noindent
Let now $\Omega$ be any open set in  ${\mathbb{R}}^N.$ We consider the Banach space
$$X = C^0(\Omega) = \{u\in C(\overline{\Omega}), u\equiv 0 \, \hbox{on} \, \partial\Omega \}$$
endowed with the supremum norm and we define the linear operator A by
$$  D(A) = \{u\in X\cap H^1_0(\Omega), \Delta u\in X\};       Au = \Delta u,  \forall u\in
D(A).$$
Then if the boundary of $\Omega$ is Lipschitz continuous, $A$  is m-dissipative \index{m-dissipative} and densely defined on $X.$\medskip

\begin{example}\label{WaveOperat}{\bf: The wave operator on  $H^1_0(\Omega)\times
L^2(\Omega)$.}
\end{example}
\noindent Let $\Omega$ be any open set in  ${\mathbb{R}}^N$ and $ X = H^1_0(\Omega)\times
L^2(\Omega).$ The space $X$ is a real Hilbert space when equipped with the inner product
$$  \langle (u,v),(w,z)\rangle  = \int_\Omega (\nabla u \nabla w + vz)\,   dx,$$ 
inducing on  $X$ a norm equivalent to the standard product norm on $H^1_0(\Omega)\times L^2(\Omega)$. We define the linear operator $A$  on $X$ by
$$  D(A) = \{(u,v)\in X, \, \Delta u\in L^2(\Omega), \,  v\in H^1_0(\Omega)\}$$
  $$A(u,v) = (v, \Delta u), \forall (u,v)\in D(A).$$
Then $A$ is skew-adjoint \index{skew-adjoint} in $X$, and in particular $A$ and $-A$ are both m-dissipative \index{m-dissipative} with dense domains.
\section[The Hille-Yosida-Phillips theorem]{The semi-group generated by m-dissipative \index{m-dissipative} operators. The Hille-Yosida-Phillips theorem \index{Hille-Yosida-Phillips theorem}}

\subsection {The general case}Let $X$ be a real Banach space and let $A$ be a linear, densely defined, m-dissipative \index{m-dissipative}  operator on $X$. The following fundamental Theorem is proved for instance in \cite{{MR0710486},{MR0239384}}. 
\begin{thm}\label{thm1} There  exists a unique one-parameter family $T(t)\in L(X)$
defined for  $t\geq0$  and such that
\begin{itemize}
\item[(1)]    $T(t)\in L(X)$   and  $ \Vert T(t)\Vert_{L(X)} \leq1, \forall t\geq0.$
\item[(2)]    $ T(0) = I,$
\item[(3)]    $T(t+s) = T(t)T(s), \forall s,t\geq0.$
\item[(4)]  For each $x\in D(A), u(t)=T(t)x$ is the unique solution of the problem
\end{itemize}
\begin{eqnarray*}
\left\{
\begin{array}{rcl}
&&u\in C([0, +\infty);D(A))  \cap C^1([0, +\infty); X)\\
\\
&&u'(t) = Au(t), \, \forall t\geq 0\\
\\
&&u(0) = x
\end{array}
\right.
\end{eqnarray*}
Finally, for each
$x\in D(A)$ and $ t\geq0$, we have:  $ T(t)Ax=AT(t)x.$
\end{thm}

\subsection {Two important special cases}

    In this paragraph, we assume that  $X$ is a (real) Hilbert space. The following two
results can be considered as refinements of Theorem \ref{thm1}.

    \medskip

\begin{thm}\label{thm2} Let $A $ be self-adjoint \index{self-adjoint} and $\leq 0$.  Let $ x\in X,$ and $u(t)=T(t)x.$ Then u is the unique solution of
\begin{eqnarray*}
\left\{\begin{array}{rcl}
&&u\in
C([0, +\infty);X)\cap C((0, +\infty);D(A))\cap C^1((0, +\infty); X)\\
\\
&&u'(t) = Au(t), \, \forall t> 0\\
\\
&&u(0) = x
\end{array}
\right.
\end{eqnarray*}
\end{thm}
\begin{rem}  {\rm Theorem \ref{thm2} means that $T(t)$ has a "smoothing effect" on
initial data. Indeed,  even if $x\in D(A)$, we have $T(t)x\in D(A),$ for all $t>0$. As
a basic example, let us consider the case
 $X=L^2(\Omega), A$ defined  by  $  D(A) = \{u\in H^1_0(\Omega),\, \Delta u\in
L^2(\Omega)\},
Au = \Delta u,\, \forall u\in D(A)$ where $\Omega$ is a bounded open set in  $
{\mathbb{R}}^N$ and the boundary of $\Omega$ is smooth. Theorem \ref{thm2} here says that for
each
$u_0\in L^2(\Omega)\}$, there exists a unique solution
    $$ u\in C([0,+\infty), L^2(\Omega))\cap C(0,+\infty, H^2(\Omega)\cap H^1_0(\Omega) )\cap
C^1(0,+\infty, L^2(\Omega))$$
of : $$ u_t = \Delta u\, ;  \quad   u(0) = u_0.$$
Actually a much stronger smoothing property holds true since by iterating the
procedure we prove easily that $u(t)\in D(A^n)$ for all $n\in {\mathbb{N}}$ and $t>0.$ In
particular $u(t, .)$ is smooth up to the boundary.}\end{rem}

A somewhat  opposite situation is that of isometry groups generated by skew-adjoint \index{skew-adjoint}
operators.
\begin{thm}\label{skewAdj}
Let $A $ be skew-adjoint.\index{skew-adjoint} Then $T(t)$ extends to one-parameter
group of operators $T(t):{\mathbb{R}}\rightarrow L(X)$ such that

\begin{itemize}
\item[(1)]     $\forall x\in X,  \,  T(t)x \in C({\mathbb{R}}, X). $
\item[(2)]    $\forall x\in X, \forall t\in {\mathbb{R}}, \quad \Vert T(t)x\Vert = \Vert x\Vert$.
\item[(3)]  $ \forall s\in {\mathbb{R}}, \forall t\in {\mathbb{R}},T(t+s) = T(t)T(s).$
\item[(4)] For each $ x\in D(A),\,  u(t)=T(t)x$ is a solution of $ u'(t) =
Au(t), \, \forall t\in {\mathbb{R}} $.
\end{itemize}
\end{thm}

\begin{example} {\rm Let  $ X = H^1_0(\Omega)\times L^2(\Omega)$, and let $A$ be
as in Example \ref{WaveOperat}. We obtain that for  any $(u_0,v_0)\in X$ , there is a
solution
$   u\in C({\mathbb{R}}, H^1_0(\Omega) )\cap C^1({\mathbb{R}}, L^2(\Omega))\cap
C^2({\mathbb{R}}, H^{-1}(\Omega) )$
of:     $$ u_{tt} = \Delta u;   \quad  u(0) = u_ 0, u_t(0) = v_0 .$$
It  can be shown that $u$ is unique.}
\end{example}
\section{ Semilinear problems}
     Let $X$ be a real Banach space, let $A$ be a linear, densely defined,  m-dissipative \index{m-dissipative} operator on $X$, and let  $T(t)$ be given by Theorem \ref{thm1}. The following Theorem is quite similar to the construction of the flow associated to an ordinary differential system and is the starting point of the theory of semilinear evolution equations.
\begin{thm}\label{theo3}  Let $ F: X \rightarrow X $ be Lipschitz continuous on each
bounded subset of X. Then for each $x\in X $,
There is $\tau(x)\in(0, +\infty]$  and a
unique maximal  solution $u\in C([0,\tau(x)),X)$
of the equation $$u(t) = T(t)x +\int
_0^t T(t-s)F(u(s)) \,ds$$
The number $\tau(x)$ is the existence time of the solution , and satisfies the following alternative: either $\tau(x)= \infty $  and the solution $u$ with initial datum $x\in X $ is global \index{global solution} (in $X$); or  $\tau(x)<\infty $ and the solution $u$ with initial datum $x\in X$ blows up in finite time (in $X$). In the latter case we have 
 $$\Vert u(t)\Vert
\longrightarrow +\infty \hbox{ as } t\longrightarrow \tau(x). $$
\end{thm}
In the theory of semilinear evolution equations, a basic tool to establish \index{global solution} global existence, uniqueness, boundedness or stability properties of the solution will be the Gronwall Lemma \index{Gronwall Lemma} (cf. Lemma \ref{Gronwall Lemma}).
\section{A  semilinear heat equation}\index{heat equation}\label{SectionHeatEquation}
Let  $\Omega$ be any open set in  ${\mathbb{R}}^N$ with  Lipschitz continuous boundary $\partial\Omega$ , and let us  consider the equation
\begin{equation}\label{heat}
 u_t - \Delta u + f(u) = 0 \quad\hbox{  in} \, \,{\mathbb{R}}^+ \times
\Omega,
\quad     u = 0  \,\,   \hbox{on} \, \, {\mathbb{R}}^+ \times \partial\Omega
\end{equation}
 where  $f$  is a locally Lipschitz continuous function:
${\mathbb{R}}\rightarrow{\mathbb{R}} $ with  $f(0) = 0.$ It is natural to set
$$X = C^0(\Omega) = \{u\in C(\overline{\Omega}), u\equiv 0 \, \hbox{on} \, \partial\Omega
\}$$
and to introduce the semi-group $T(t)$ on $X$ associated to the homogeneous linear
problem

            $$ u_t - \Delta u  = 0 \quad\hbox{  in} \, \,{\mathbb{R}}^+ \times \Omega,
\quad     u = 0  \,\,   \hbox{on} \, \, {\mathbb{R}}^+ \times \partial\Omega
$$ In fact here $T(t)$ is the semi-group generated by the operator $A$ of  Example
\ref{exempC0}. Let $\varphi\in X$: by  Theorem \ref{theo3} we can define
$\tau(\varphi)\leq\infty$ and a unique maximal solution solution $u\in
C([0,\tau(\varphi)),X)$ of the equation$$u(t) = T(t)x +\int _0^t T(t-s)F(u(s)) ds
$$
with  $F : X \rightarrow X$ given by  $(F(u))(x) : = - f(u(x))$ for all $x$ in the closure
of W. Then u can be considered as the local solution of (\ref{heat}) \index{heat equation} with initial condition
$u(0) =\varphi$ in $X$. The following simple result will be useful later on.

    \medskip

\begin{prop} Let  f satisfy the condition
\begin{equation}\label{ineg1}
\forall s\in {\mathbb{R}} \hbox{ with } \vert s\vert \geq C,     f(s) s \geq 0
\end{equation}
Then we have for any $\varphi\in X $
\begin{equation}\label{ineg2} \tau(\varphi) = \infty  \quad\hbox{and} \quad  \sup_{t \geq
0}\Vert u(t)\Vert_{L^\infty} \leq Max \{C, \Vert \varphi\Vert_{L^\infty} \} < \infty
\end{equation}
 where u is the solution of  \eqref{heat} \index{heat equation} with initial condition $u(0) =\varphi$ .
\end{prop}
\begin{proof}[{\bf Proof.}]  Let $M = Max \{C, \Vert \varphi\Vert_{L^\infty} \}$ and let us
show for instance that
$ u(t, x) \leq M$ on $(0, \tau(\varphi))\times \Omega $. Introducing $z = u - M,$ we have
    $$ z_t - \Delta z = f(M) - f(u) - f(M) \leq f(M) - f(u) $$
since $f(M) \geq 0.$ In addition it can be shown that $$u\in C(0,
\tau(\varphi);H^2(\Omega)\cap H^1_0(\Omega))\cap C^1(0, \tau(\varphi);L^2(\Omega)) $$
and then
    $$(d/dt) \int_{\Omega} \vert z^+\vert^2 dx = 2 \int_{\Omega} z^+z_{t}\,  dx =2
\int_{\Omega} z^+( \Delta z + f(M) - f(u) - f(M)) dx$$
 $$     \leq - 2 \int_{\Omega}
\nabla z^+. \nabla z \, \, dx + 2 \int_{\Omega}
z^+\vert f(M) - f(u) \vert dx $$
Because $f$ is locally Lipschitz and $u$ is bounded on $(0, t)\times \Omega $ for each
$t <\tau(\varphi)$, we have     $$ \vert f(M) - f(u) \vert (t, x) \leq K(t) \vert z(t, x)\vert
\quad {on}
\quad (0, t)\times \Omega$$
Then by using the identities $z = z^+ - z^-$  and  $z^+. z^- = 0, \nabla z^+.\nabla z^-
= 0$ almost everywhere, we obtain:

$$(d/dt)\int_{\Omega} \vert z^+\vert^2 dx \leq - 2 \int_{\Omega}
\Vert\nabla z^+\Vert^2 \, \, dx  + 2  K(t) \int_{\Omega}
\vert z^+\vert^2 \, \, dx$$
The inequality $u(t, x) \leq M$ on $(0, \tau(\varphi))\times \Omega$  now  follows
easily  by an application of Lemma \ref{Gronwall Lemma} \index{Gronwall Lemma} since $z^+(0, x)\equiv 0 $. Similarly
we show $u(t, x) \geq - M $ on $(0, \tau(\varphi))\times \Omega$.
\end{proof}
\section[A semilinear wave equation with  linear dissipation]{A semilinear wave equation with a linear dissipative term\index{wave equation}}\label{sectionWaveEqua}
Let  $\Omega$ be any open set in  ${\mathbb{R}}^N$ with  Lipschitz
continuous boundary
$\partial\Omega$ , and
let us  consider the equation \index{wave equation}
\begin{equation}\label{wave} u_{tt} - \Delta u +\gamma u_t + f(u) = 0 \quad\hbox{  in} \,
\,{\mathbb{R}}^+
\times
\Omega,
\quad     u = 0  \,\,   \hbox{on} \, \, {\mathbb{R}}^+ \times \partial\Omega
\end{equation}
 where  $f$  is a locally Lipschitz continuous function:
${\mathbb{R}}\rightarrow{\mathbb{R}} $ with  $f(0) = 0 $ satisfying the growth condition \index{growth condition}

\begin{equation}\label{GrowthCondi}\vert f'(u)\vert \leq C(1+\vert u\vert^r),    \,\,
\hbox{   a.e. on }\,\,{\mathbb{R}}
\end{equation}
 with  $r \geq 0$ arbitrary if $N = 1$ or  $2$  and  $\displaystyle 0 \leq r \leq
\frac{2}{N-2}$ if
$N\geq 3.$  It is natural to set $$ X = H^1_0(\Omega)\times L^2(\Omega)$$
Let us denote by  $f^*$ the mapping  defined by
$$f^*((u,v)) = (0, - f(u)),    \forall(u,v)\in X. $$
The growth condition \index{growth condition} (\ref{GrowthCondi}) together with Sobolev embedding theorems imply that
$$f^*(X)\subset X   ;\    f^*: X\longrightarrow X  \    \hbox{is Lipschitz continuous
on bounded subsets}.$$
We also define the operator $\Gamma \in L (X)$ given by
$$\Gamma((u,v))=(0,\gamma v),\quad    \forall(u,v)\in X. $$
Finally let $T(t)$ (cf. Theorem \ref{skewAdj} with $A$ as in
example \ref{WaveOperat}  in $X = H^1_0(\Omega)\times L^2(\Omega)$)  be the isometry group on X generated by the linear  wave equation\index{wave equation}
$$ u_{tt} - \Delta u  = 0 \quad\hbox{  in} \, \,{\mathbb{R}}^+ \times
\Omega,
\quad     u = 0  \,\,   \hbox{on} \, \, {\mathbb{R}}^+ \times \partial\Omega
$$For each
$(\varphi,\psi)\in X$, by  Theorem \ref{theo3} we can define a unique maximal solution solution $ U = (u,u_t) \in
C([0,\tau(\varphi,\psi));X)$ of the equation
$$ U(t) = T(t)(\varphi,\psi) + \int_0^t T(t-s)\{f^*((U(s)-\Gamma(U(s))\} ds$$ The following simple result will be useful later on.

\begin{prop}\label{propEqOnde}  Assume $\gamma \geq 0$ , and let $f$ satisfy the
condition
\begin{equation}\label{ineg4}
\forall s\in {\mathbb{R}}, \quad  F(s) \geq ( - \frac{\lambda_1}{2} +
\varepsilon) s^2 - C   \quad with \, \, \varepsilon> 0, C \geq 0
\end{equation}
 where $F$ is the primitive of $f$ such that $F(0) = 0$ and $\lambda_1$ is
the first eigenvalue of  $- \Delta $ in $H^1_0(\Omega)$. Then we have for any
$(\varphi,\psi)\in X:  \tau(\varphi,\psi) = \infty $ and the solution  $U = (u,u_t)$  of
\eqref{wave} \index{wave equation} such that $ U(0) = (\varphi,\psi)$ satisfies :
$$\sup_{t \geq 0}\Vert (u(t),u_t(t))\Vert_X  < \infty.$$
\end{prop}
\begin{proof}[{\bf Proof.}] The solutions  of  \eqref{wave} \index{wave equation} satisfy the energy equality
$$\gamma \int_0^t\int_{\Omega}u_t^2(t,x) dx dt  + E(u(t),u_t(t)) = E(\varphi,\psi) $$
with $$ E(\varphi,\psi): = \frac{1}{2}\int_{\Omega}\Vert \nabla \varphi(x)\Vert^2
dx + \frac{1}{2}\int_{\Omega}\vert  \psi(x)\vert^2
dx + \int_{\Omega}F(\varphi(x))
dx $$ In particular since $\gamma \geq 0$, we find $E(u(t),u_t(t))\leq E(\varphi,\psi) $
and the result follows quite easily from (\ref{ineg4}). Indeed, from Poincar\'e
inequality we deduce    $$\forall w\in H^1_0(\Omega), (1-\eta) \int_{\Omega} \vert\nabla w
\vert^2dx
\geq (
\lambda_1-2\varepsilon)\int_{\Omega} w^2 dx,$$
whenever $ \eta \leq 2\varepsilon/\lambda_1 . $ Then
$$ E(\varphi,\psi) \geq (\eta/2) \int_{\Omega} \vert\nabla\varphi \vert^2dx +
\frac{1}{2}\int_{\Omega}\vert  \psi(x)\vert^2 dx- C\vert\Omega\vert  , \quad
\forall(\varphi,\psi)\in X,$$  and a bound on $E$ implies a bound in $X$.
\end{proof}
\chapter[Uniformly damped linear semi-groups]{Uniformly damped linear semi-groups}
\section{A general property of linear contraction semi-groups}
Let $X$ be a real Banach space and $L$ any m-dissipative \index{m-dissipative} operator on $X$ with dense
domain. We consider the evolution equation
\begin{equation}\label{FirstOrderEqu}
u' =  Lu(t) ,  \quad  t \geq 0
\end{equation} For any $u_0\in X$, the formula $u(t) = S(t) u_0$
where $S(t)$ is the contraction semi-group generated by $L$ defines the unique
generalized solution of \eqref{FirstOrderEqu} such that $u(0) = u_0 .$
 We recall  the
following simple property :\medskip

\begin{prop}\label{prop2.1.1}
For all $t \geq 0,$ let us denote by $\Vert S(t)\Vert$ the norm of the contractive operator S(t)
in $L(X).$ Then  $\Vert S(t)\Vert$  satisfies  either of the two following properties
\begin{itemize}
\item[(1)] For all $t \geq 0,  \Vert S(t)\Vert = 1.$
\item[(2)]  $\exists \varepsilon > 0, \exists M > 0,$  for all $t \geq 0 ,
 \Vert S(t)\Vert  \leq M e^{-\varepsilon t}.$
 \end{itemize}
\end{prop}
\begin{proof}[{\bf Proof.}]  The function $\Vert S(t)\Vert$ is nonincreasing.  If
for some $T >0$ we have
$\Vert S(t)\Vert=1$ for $t\in[0, T)$ and $\Vert S(T)\Vert=0,$ then $\forall \varepsilon> 0, \forall t \geq 0 ,  \Vert S(t)\Vert  \leq M (\varepsilon) e^{ - \varepsilon t}$  with $M (\varepsilon) = e^{
\varepsilon T}$.  Assuming, on the contrary, that
for some $\tau > 0$  we have $0 < \Vert S(\tau)\Vert < 1,$ for each $ t \geq 0 $ we can write $t
= n\tau + s$,  with $n\in {\mathbb{N}}, \,  0\leq s \leq \tau.$ Then  $\Vert S(t)\Vert \leq \Vert S(\tau)\Vert^n$
and we obtain (2) with  $\displaystyle\varepsilon = -  {Log \Vert S(\tau)\Vert\over \tau}$
and
$M =e^{ \varepsilon \tau } = 1/ \Vert S(\tau)\Vert$.
\end{proof}
\section{The case of the heat equation\index{heat equation}}
 The linear heat equation \index{heat equation} can be studied in many  interesting spaces. Its
treatment  is especially simple in the Hilbert space setting of example \ref{HeatHilbert}.
However, in view of the applications to semilinear perturbations the $C_0 $- theory
is more flexible. Let us start with the Hilbert space setting : following the notation
of example \ref{HeatHilbert}, we denote by $S(t)$ the semi-group generated by $B$ in $H =
L^2(\Omega).$ We have the following simple result.\medskip
\begin{prop}\label{prop2.2.1}   Let $ \lambda_1 = \lambda_1(\Omega)$ be the first
eigenvalue of $(-\Delta)$ in $H^1_0(\Omega).$
Then \begin{equation}\label{equ2.10} \Vert S(t) \Vert _ {{\cal L}(H)}
\leq e^{-\lambda_1 t}, \quad \forall t\geq 0.\end{equation}
\end{prop}
\begin{proof}[{\bf Proof.}] Let
$\varphi
\in D(B)$, and consider
$$f(t)=(e^{\lambda_1 t}\ \Vert S(t)\varphi\Vert_H)2 ,\quad  \forall
t\geq0.$$ We have
\begin{eqnarray*}  e^{- 2\lambda_1 t}f '(t) &=& 2\lambda_1 \int_\Omega u(t, x)^2 dx+ 2\int_\Omega u(t, x)u'(t,x)dx \\
&=&2\lambda_1 \int_\Omega u(t, x)^2 dx+ 2\int_\Omega u(t, x)\Delta u(t,x)dx \\
&=&2\Bigl(\lambda_1 \int_\Omega u(t, x)^2 dx-\int_\Omega \vert\nabla  u(t, x)\vert ^2dx\Bigr) \leq 0.
\end{eqnarray*}
Hence  $$  \Vert S(t)\varphi\Vert_H\leq e^{-\lambda_1 t} \Vert\varphi\Vert_H,\quad  \forall t\geq0,
\forall
\varphi\in D(B).$$  The result follows by  density.
\end{proof}

    We now assume that  $\Omega $ is bounded  with  a Lipschitz continous boundary and we
use the  notation of Example \ref{exempC0}. Let $T(t)$ denote the semi-group generated by $A$
in $X.$ Since $X\subset H$ with continous imbedding and $G(A)\subset G(B),$ it is
classical, using the Hille-Yosida theory, to prove
\begin{equation}\label{equ2.11}\forall\varphi\in X ,
 \forall t \geq0, \quad  T(t)\varphi=S(t)\varphi
 \end{equation}
In particular we have:  $ \Vert S(t)\varphi\Vert_H \leq e^{-\lambda_1 t}\Vert \varphi\Vert_H,$  for
each $t\geq0$  and
$\varphi \in X.$  The following property of uniform damping in X will be more
interesting for semilinear perturbations\medskip
\begin{thm}\label{thm2.2.2}     Let $ \lambda_1 = \lambda_1(\Omega)$ be the first
eigenvalue of $(-\Delta)$ in $H^1_0(\Omega).$ Then
\begin{equation}\label{equ2.12}\Vert S(t)\Vert_{{\cal L}(X)} \leq M e^{-\lambda_1 t} ,
\quad \forall t\geq0, \end{equation}
with
\begin{equation}\label{equ2.13} M = \exp\bigl({\lambda_1\vert\Omega\vert^{2/N} \over{4\pi}}\bigr).
\end{equation}
\end{thm}
In the proof of Theorem \ref{thm2.2.2} we shall use a rather well-known smoothing property of
$S(t)$ in $L^p$ spaces. Denoting by $\Vert\cdot \Vert_p $ the norm in $L_p(\Omega)$ ,
we recall
\begin{prop}\label{Prop2.2.3} Let $ 1\leq p\leq q\leq\infty.$  Then
$$  \Vert S(t)\varphi\Vert_q \leq ({1\over{4\pi t}})^{{N\over 2}({1\over p}-{1\over q})}
\Vert\varphi\Vert_p , \forall t>0, \forall \varphi\in X.
$$
\end{prop}

A possible proof, omitted here, relies on the explicit form of the heat kernel in ${\mathbb{R}}^N $ together with a comparison principle.\\
\begin{proof}[{\bf Proof of Theorem
\ref{thm2.2.2}}]
 Let $\varphi\in X$ and $T>0.$ First for  $0\leq t\leq T,$ we have trivially
$$  \Vert S(t)\varphi\Vert_\infty \leq \Vert\varphi\Vert_\infty \leq e^{-\lambda_1 t} e^{\lambda_1
T}\Vert\varphi\Vert_\infty .$$
Then if $t \geq T,$ we find successively, applying first Proposition \ref{Prop2.2.3} with $p = 2$ and
$q =\infty$
\begin{eqnarray*}\Vert S(t)\varphi\Vert_\infty  &\leq& \bigl({1\over4\pi T}\bigr)^{N\over4} \Vert S(t-T)\varphi\Vert_2\\
&\leq &
\bigl({1\over4\pi T}\bigr)^{N\over4}e^{-\lambda_1 t} e^{\lambda_1 T}\Vert\varphi\Vert_2\quad\hbox{ (by Proposition \ref{prop2.1.1})}\\
&\leq& \vert\Omega\vert^{1\over2 }\bigl({1\over4\pi T}\bigr)^{N\over4} e^{\lambda_1
T}e^{-\lambda_1 t} \Vert\varphi\Vert_\infty .
\end{eqnarray*}
Then the estimate follows by letting $\displaystyle T = {{\vert\Omega\vert^{2\over N
}}\over{4\pi}}$.
\end{proof}
\begin{rem}\label{Rem2.2.4}{\rm  Actually  \eqref{equ2.12} is not valid with $M=1.$
More precisely, if  $\Vert S(t)\Vert_{{\cal L}(X)} \leq M'e^{-mt} $ with $m>0,$ we must have $M'>1$. Indeed, let $\varphi\in {\cal D}(\Omega)$ be such that $\varphi \equiv 1$ near $x_0\in \Omega$ and $\Vert\varphi\Vert_X=1$, and let $u(t) =S(t)\varphi.$ It is then
easily verified that $u\in C^\infty([0,\infty)\times \Omega).$ Consequently
$u_t(0, x)\equiv 0$ near $x_0.$ Hence, for any $ \varepsilon >0 $
and any $x$ close enough to $x_0$, we find $$ u(t,x) \geq 1-
\varepsilon t,$$ for all t sufficiently small : in particular
$$\Vert u(t)\Vert_X \geq1- \varepsilon t $$ for $t$ small. This estimate
with $ \varepsilon >0 $ arbitrary small is not compatible with
\par\noindent $\Vert S(t)\Vert_{{\cal L}(X)} \leq e^{-\mu t} $,  for whatever
value  $\mu >0$.}
\end{rem}
\section{The case of linearly damped wave equations\index{wave equation}}
We have the following result
\begin{prop}\label{Proposition2.3.3.}      Let $\Omega$  be a bounded domain in ${\mathbb{R}}^N $. Consider the equation \index{wave equation}
\begin{equation}\label{Equ(2.18)}   u_{tt} -  \Delta u + \lambda u_t = 0 \ \hbox {in}\
{\mathbb{R}}^+\times\Omega,\quad     u = 0    \ \hbox {on}\
{\mathbb{R}}^+\times \partial\Omega \end{equation}
Then, denoting by  $\Vert\cdot\Vert$ the norm in  $H^1_0(\Omega)$ and by $\vert \cdot\vert$ the norm in $L^2(\Omega)$,  for any solution u of \eqref{Equ(2.18)} we have
\begin{equation}\label{Equ(2.19)}\Vert u(t)\Vert + \vert u_t(t) \vert \leq C( \Vert u(0)\Vert + \vert u_t(0)\vert ) e^{ -\delta  t}
\end{equation} for some $C, \delta>0$ .

\end{prop} This result is a special case of the following more general statement. Let $A$ be a positive self-adjoint \index{self-adjoint} operator with dense domain on a real Hilbert space $H$  with norm denoted by $\vert.\vert$ and inner product denoted by $(.,.) $. A is assumed  coercive on $H$ in the sense that 
$$ \exists\alpha >0, \forall u \in D(A), \quad (Au, u) \ge \alpha \vert u\vert^2 .$$
We introduce $ V : = D(A^{1/2})$, the closure in $H$  of $D(A)$ under the norm 
$$ p(u) : = (Au, u) ^{1\over 2}.$$
 The norm $p$ extends  on $V$  and we equip $V$ with the extension of $p$, denoted by $\Vert \Vert$ so that 
$$\forall  u\in V, \quad  \Vert u\Vert = \vert A^{1/2}u\vert$$ where $A^{1/2}\in L (V; H)\cap L (D(A); V) $ is the unique nonnegative square root of $A$. The duality product between $V$ and its topological dual $V'$ extends the inner product on $H$ in the following way: 
$$ \forall (f, v) \in H\times V, \quad  \langle  f , v\rangle_{V', V} = (f, v).$$
  In particular we have $$ \forall (u, v) \in D(A)\times V, \quad  \langle Au, v\rangle_{V', V} = (Au, v) = (A^{1/2}u, A^{1/2}v).$$ 
In particular by the definition of the standard norm on $V'$ we have $$ \forall u \in D(A), \quad  \Vert Au\Vert_{V'} \le \vert A^{1/2}u\vert  = \Vert u\Vert .$$ By selecting $v= u$ we even obtain $$ \forall u \in D(A), \quad  \Vert Au\Vert_{V'}= \Vert u\Vert .$$ By Lax-Milgram's theorem the extension $\Lambda $ of $A$ by continuity on $V$ is bijective from $V$ to $V'$ and in addition,  $\Lambda $ satisfies $$ \forall (u, v )\in V\times V,\quad  \langle \Lambda u , v\rangle_{V', V} = (A^{1/2}u, A^{1/2}v)$$ so that $\Lambda $ becomes by definition the duality map from $V$ to $ V'$. Finally, denoting by $\Vert\cdot \Vert_{*}$ the standard norm on of $V'$  we remark that
$$ \forall f\in V',\quad \Vert f\Vert_* = \Vert \Lambda^{-1} f\Vert.$$ 
Let now  $B\in L (V; V')$ be such that  $$\forall  v\in V, \quad (Bv, v) \geq 0. $$
We consider the second order equation 
 \begin{equation*}
u''+\Lambda u+Bu'=0.
\end{equation*}
and the energy space $E =  V
\times H$ is equipped with the Hilbert product space norm.  
\begin{prop}\label{Prop2.4.1.}     The unbounded operator on $E$ defined by
\begin{equation}
 D(L) = \{ (u, v) \in V\times V;  \quad \Lambda u+ Bv \in H\}\end{equation}
 \begin{equation} L(u, v) =(v, - \Lambda u - Bv)\quad \forall (u, v)\in D(L)
\end{equation}
is m - dissipative \index{m-dissipative} on $E$.
\end{prop}
\begin{proof}[{\bf Proof.}] We denote by $\langle , \rangle$ the inner product in $E.$ First $L$ is dissipative \index{dissipative} on $E$ . Indeed for any $U = (u, v) \in D(L)$ we have 
\begin{eqnarray*} \langle LU, U \rangle &=& (v, u)_V + ( - \Lambda u - Bv, v)_H\\
& =& (A^{1/2}v, A^{1/2}u) + \langle - \Lambda u - Bv, v\rangle_{V', V} \\
&=& \langle  - Bv, v\rangle_{V', V} \le 0.
\end{eqnarray*}
In order to prove that $L$ is m-dissipative \index{m-dissipative} on $E$ we consider, for any $ (f, g) \in E $ the equation 
$$ (u, v) \in D(L);\quad - L(u, v) + (u, v) = (f, g) $$ which is equivalent to 
$$(u, v) \in V\times V;  \quad -v+u = f;\quad   \Lambda u + Bv +v = g $$ or in other terms 
$$(u, v) \in V\times V;  \quad u = f +v;\quad   \Lambda v + Bv +v = g - \Lambda  f $$ Assuming we know that the operator $ C =  \Lambda  + B +I $ is such that $C(V) = V'$ we conclude immediately that $$( I-L)D(L) = E $$ and therefore $L$ is m-dissipative \index{m-dissipative} as claimed. The property $C(V) = V'$ is an immediate consequence of the following elementary lemma \end{proof}

\begin{lem}\label{Lemma2.4.2.}   Let $V$ be a real Hilbert space and $C\in L(V, V')$. Assume that for some $\eta>0$ we have 
$$ \forall v \in V, \quad \langle Cv, v\rangle_{V', V}\ge \eta \Vert v\Vert^2. $$ Then $C(V) = V.'$ \end{lem} 
\begin{proof}[{\bf Proof.}] First $C(V)$ is a closed linear subspace of $V'$. Indeed if  $ f_n= Cv_n\in C(V)$ and $f_n$ converges to $f \in V'$ we have for each $(m, n)$ the inequality $$\Vert v_n-v_m\Vert^2 \le \frac{1}{\eta}\langle f_n-f_m, v_n-v_m\rangle_{V', V} \Longrightarrow \Vert v_n-v_m\Vert \le \frac{1}{\eta}\Vert f_n-f_m\Vert_{*}$$ Hence $ v_n$ is a Cauchy sequence in $V$ and its limit $v$ satisfies $ Cv = f$. Now if $C(V) \not = V'$ there exists a non-zero vector $w\in V$ such that  $$ \forall v \in V, \quad \langle Cv, w\rangle_{V', V} = 0$$ By letting $v = w$ we conclude that $w = 0$, a contradiction.
\end{proof}

\begin{prop}\label{Prop2.4.3}  Let
$A, V$ and $H$ be as above.  Let $ B\in {\cal L}(V, V')$  satisfy the
following conditions
\begin{eqnarray*}& &\exists\alpha >0, \quad \forall v\in V, \quad \langle B v, v\rangle_{V', V}\,\,\geq\, \alpha\vert v\vert^2 \\
& &\exists C >0,\quad \forall v\in V,\quad \Vert
B(v)\Vert_{V'}^2\leq C(\langle B v, v\rangle_{V', V}+ \vert v\vert^2).
\end{eqnarray*}
 Let $u\in
C^1(0,+\infty, V) \cap C^2(0,+\infty, V')$ be a solution of $$ u'' + Au + Bu' = 0 .$$
There exists some constants $C\ge1$  and  $ \gamma > 0 $ independent of $u$ such that $$
\forall \ge0, \,\, \Vert u(t), 
u'(t)\Vert_E \leq C e^{-\gamma t}\Vert
u(0), u'(0)\Vert_E. $$ \end{prop}\begin{proof}[{\bf Proof.}] 
We consider for all $t>0$ and $ \varepsilon > 0$ small enough
$$H_\varepsilon(t)= \Vert
u'(t)\Vert^2\,+ \Vert
A^{1\over2}u(t)\Vert^2 +\varepsilon\ (u(t),\, u'(t))$$ 
and we compute
\begin{eqnarray*}H_\varepsilon'(t)&=&-\langle B(u'(t)),u'(t)\rangle\,+\varepsilon \Vert u'(t)\Vert^2  +\varepsilon \langle u''(t),\, u(t)\rangle \\
&=&-\langle B(u'(t)),u'(t)\rangle\,+\varepsilon\Vert u'(t)\Vert^2 -\varepsilon\Vert A^{1\over2}u(t)\Vert^2 
-\varepsilon \langle Bu'(t),\, u(t)\rangle \\
&\leq & -\langle B(u'(t)),u'(t)\rangle\,+\varepsilon \Vert u'(t)\Vert^2 -\varepsilon \Vert A^{1\over2}u(t)\Vert^2 + \eta \varepsilon\Vert u(t)\Vert_V^2 +{\varepsilon\over \eta} \Vert Bu'(t)\Vert_{V'}^2 \\
& \leq& (-1 + {C \varepsilon\over \eta})\langle B(u'(t)),u'(t)\rangle\,+\varepsilon(1 + {C\over \eta})
\Vert u'(t)\Vert^2 -\varepsilon (1-\eta)
\Vert A^{1\over2}u(t)\Vert^2.
\end{eqnarray*}
 Choosing for instance $\eta = \sqrt \varepsilon$
and letting $ \varepsilon $ small enough we obtain first 
$$  H_\varepsilon'(t) \leq  -{\varepsilon\over 2}[\Vert
u'(t)\Vert^2\,+ \Vert
A^{1\over2}u(t)\Vert^2 ]. $$ On the other hand it is not difficult to check for $ \varepsilon $ small enough the inequality: $$
(1-M\varepsilon) \Vert
u'(t)\Vert^2\,+ \Vert
A^{1\over2}u(t)\Vert^2\le  H_\varepsilon(t)\le (1+M\varepsilon) \Vert
u'(t)\Vert^2\,+ \Vert
A^{1\over2}u(t)\Vert^2. $$ where $M$ is independent of the solution $u$ as well as $t$ and $\varepsilon. $ This concludes the proof. \end{proof}
\begin{rem}\label{wave diss}{\rm  If $(u(0), u'(0)) \in D(L)$, then clearly $u\in
C^1(0,+\infty, V) \cap C^2(0,+\infty, V')$. By density, Proposition \ref{Prop2.4.3} means that the semi-group generated by $L$ is exponentially  damped in $E$. In particular Proposition \ref{Proposition2.3.3.} follows as a special case.}
\end{rem}

\chapter[Generalities on dynamical systems]{Generalities on dynamical systems\index{dynamical system}}

\section{General framework}

Throughout this paragraph, $(Z, d)$ denotes a complete metric space.

\begin{defn} A dynamical system \index{dynamical system}  on $(Z, d)$ is a one parameter family
$\displaystyle \{S(t)\}_{t\geq0}$ of maps   $Z\rightarrow Z$ such
that
\begin{itemize}
\item[(i)] $\forall t \geq 0, S(t)\in C(Z,Z)$;
\item[(ii)] $S(0) = $ Identity;
\item[(iii)] $\forall s, t \geq 0,S(t+s) = S(t)\circ S(s)$;
\item[(iv)] $\forall z\in Z,  S(t)z\in C([0,+\infty), Z).$
\end{itemize}
\end{defn}

\begin{rem}\label{Remark4.1.2.}{\rm In the sequel we shall often denote $S(t)S(s)$ instead of
$S(t)\circ S(s)$.}
\end{rem}

\begin{rem}\label{Remark4.1.3.} {\rm If  $F$ is a closed subset of $Z$ such that $S(t)F\subset
F $ for all $t\geq 0,$ then $\{ S(t)_{/F} \}_{t\geq0}$ is a dynamical
system on $(F,d)$.}
\end{rem}

 \begin{defn}\label{Definition4.1.4.} For each $z\in Z,$ the continuous curve
$t\rightarrow S(t)z$ is called the trajectory of  $z$  (under $S(t)$).
\end{defn}

\begin{defn}\label{Definition4.1.5.} Let  $z\in Z$. The set $$ \omega(z)= \{y\in Z,
\exists t_n \rightarrow +\infty, S(t_n)z\rightarrow y \, \hbox{ as }\, n \rightarrow
+\infty \}$$is called the  $\omega$-limit  set \index{$\omega$-limit set}  of  $z$ (under $S(t)$).
\end{defn}

\begin{prop}\label{prop4.1.6} We also have
$$ \omega(z) = \bigcap_{s>0}\overline{\bigcup_{t\geq s}\{S(t)z\}}. $$
\end{prop}
 \begin{proof}[{\bf Proof.}] Immediate according to Definition \ref{Definition4.1.5.}.
 \end{proof}
\begin{prop}\label{prop4.1.7} For each $z\in Z$ and any $t \geq0 $, we have
 \begin{equation}\label{Equ4.1}  \omega(S(t)z) = \omega(z);  \end{equation}
\begin{equation}\label{Equ4.2}  S(t)(\omega(z))\subset \omega(z).  \end{equation}
In addition, if $\ {\displaystyle \bigcup_{t\geq0}\{ S(t)z\}}$ is relatively compact in  $Z$, then
\begin{equation}\label{Equ4.3}  S(t) (\omega(z)) = \omega(z)\not= ÃÂÃÂÃÂÃÂ­ \emptyset .   \end{equation}
\end{prop}
 \begin{proof}[{\bf Proof.}]  a) \eqref{Equ4.1} is an immediate consequence of Proposition \ref{prop4.1.6}.\\
b) Let $y \in \omega(z)$. There is an infinite sequence $t_n
\rightarrow +\infty$ such that as $ n \rightarrow +\infty$,
$S(t_n)z\rightarrow y $. For each $t \geq0$, setting $
\tau_n=t_n+t,$ we find $S(\tau_n)z\rightarrow S(t)y $, therefore
$S(t)y\in\omega(z)$; hence \eqref{Equ4.2}.\\
c) Finally, assume ${\displaystyle \bigcup_{t\geq0}\{ S(t)z\}}$ to be precompact in $Z$. There is an infinite sequence $t_n \rightarrow +\infty$ and $y \in Z$ such that as $ n \rightarrow +\infty$,  $S(t_n)z\rightarrow y $. Thus $y\in\omega(z)$ and $\omega(z) ÃÂÃÂÃÂÃÂ­\not=\emptyset$. To establish the inclusion $\omega(z)\subset S(t) (\omega(z))$, let us consider $y\in\omega(z)$ and $t_n \rightarrow +\infty$ such that $S(t_n)z\rightarrow y $. let  $ \tau_n=t_n-t.$
By possibly replacing $ \tau_n $ by a subsequence, we may assume
$S(\tau_n)z\rightarrow w\in\omega(z) $. Hence by continuity of
$S(t)$
$$  S(t)w = S(t) \lim_{n \rightarrow +\infty}{S(\tau_n) z} = \lim_{n \rightarrow
+\infty}S(t_n)z = y,  $$
and  \eqref{Equ4.3} is completely proved.
\end{proof}
    
 In the sequel, a subset $B$ of $Z$ being given , we shall denote by 
 
 $$ d(z, B): = \inf_{y\in B} d(z, y) $$ the usual distance in the sense of $(Z, d)$ from a point $z\in Z$ to the set $B$. Using this notation we can state

\begin{thm}\label{thm 4.1.8.} Assume that ${\displaystyle \bigcup_{t\geq0}\{ S(t)z\}}$ is
relatively compact in $Z$. Then
\begin{itemize}
\item[(i)] $ S(t) (\omega(z)) = \omega(z)\not=\emptyset $, for each $t\geq0$;
\item[(ii)]$\omega(z)$ is a compact connected subset of $Z$;
\item[(iii)] $ d(S(t)z,\omega(z))\rightarrow 0 $ as $t\rightarrow +\infty$.
\end{itemize}
\end{thm}
 \begin{proof}[{\bf Proof.}] (i) is just \eqref{Equ4.3}. Moreover, for all $\displaystyle  s>0,
\overline{\bigcup_{t\geq s}\{S(t)z\}}$ is a nonempty compact
connected subset of $Z$ . Proposition \ref{prop4.1.6} therefore implies that $\omega(z)$
is a compact connected subset of $Z$ as a nonincreasing
intersection of  such sets: this is (ii). To check (iii), let us
asssume that there exist $t_n \rightarrow +\infty$ and
$\varepsilon >0$ such that for all $n, \,
d(S(t_n)z,\omega(z))\geq\varepsilon.$ By compactness and by the
definition of $ \omega(z)$, there is a point $y\in\omega(z)$ and a
subsequence $t_{n'}\rightarrow +\infty$ for which
$S(t_{n'})z\rightarrow y.$ Hence
$d(S(t_{n'})z,\omega(z))\rightarrow 0$, a contradiction which
proves the claim.
 \end{proof}
We now introduce the basic example of dynamical systems \index{dynamical system} to be studied in this book. Let $X$ be a real Banach space, let $A $ be a linear, densely defined,  m-dissipative \index{m-dissipative} operator on $X$, and let  $F: X\longrightarrow X$ be Lipschitz continuous on each bounded subset of $X$. As recalled in Theorem \ref{theo3}, for each $x\in X$, there is $\tau(x)\in (0, +\infty]$  and a unique maximal  solution $u\in C([0,\tau(x)),X)$ of the equation
\begin{equation}\label{Equ4.4}  
u(t)=T(t)x+\int_0^{t} T(t-s)F(u(s))\,ds \quad \forall t\in [0,\tau(x))
\end{equation}
where $T(t)$ is the semigroup generated by $A$ (cf. Theorem \ref{thm1}) and the number
$\tau(x)$ is the existence time of the solution. For $x\in X$ and $t\in
[0,\tau(x))$,we set
$$  S(t)x = u(t).$$

Let  $Y\subset X$ be such that for some $ M<+\infty$ we have
\begin{equation}\label{Equ4.5} \tau(y)=+\infty, \forall y\in Y; \end{equation}
\begin{equation}\label{Equ4.6} \Vert S(t)y \Vert \leq M, \forall y\in Y, \forall t\geq0.   \end{equation}
    We set $\displaystyle
Z =\overline{\bigcup_{y\in Y}\bigcup_{t\geq s}\{S(t)z\}}$ and we denote by $d$ the
distance induced on $Z$ by the norm of $X$.

\begin{lem}\label{Lemma 4.1.9.} We have the following properties
\begin{itemize}
\item[(i)] $ \quad\tau(z)=+\infty,\ \forall z\in Z;$
\item[(ii)] $  \Vert S(t)z\Vert \leq M,\ \forall z\in Z,\ \forall t\geq0;$
\item[(iii)]   $ S(t)z\in Z,\ \forall z\in Z,\ \forall t\geq0$.
\end{itemize}
\end{lem}
 \begin{proof}[{\bf Proof.}]  Let $y\in Y.$ Then if $u(t)= S(t)y$ is the solution of \eqref{Equ4.4}
with $x = y$  a straightforward calculation shows that for any
$s\geq0, v(t)=u(t+s)$ is the solution of \eqref{Equ4.4} with $x = u(s).$
Therefore,
$$S(t)S(s)y =S(t)(u(s))=u(t+s),\quad \forall s, t \geq0.$$
Consequently $\tau(S(s)y)=+\infty $ for all $y\in Y$ and each $s ,t\geq0 $ and $\Vert S(t)S(s)y \Vert \leq M $ for all $y\in Y$ and each $s ,t\geq0 $.  Now let $z\in Z.$ There exists a sequence $(t_n)$ in $[0,+\infty)$ and  a sequence $(y_n)$ in $Y$ such that $S(t_n)y_n\rightarrow z$ as $n\rightarrow +\infty.$ Pick $T< \tau(z).$ Of course we have by Gronwall's Lemma (lemma \ref{Gronwall Lemma}\index{Gronwall Lemma}) :
\begin{equation}\label{Equ4.7} S(t)S(t_n)y_n\rightarrow S(t)z \ \hbox{ as } \ n\rightarrow +\infty, \hbox{uniformly on
}[0,T]. \end{equation}

In particular  $\Vert S(t) z \Vert\leq M, \forall t\in[0,T].$ Since
$T<\tau(z)$ is arbitrary, we deduce first (i), then (ii). Finally
(iii) follows as a consequence of \eqref{Equ4.7}.
 \end{proof}
\begin{thm}\label{Theorem 4.1.10.} $\{S(t)\}_{t\geq0} $ is a dynamical system \index{dynamical system} on $(Z,
d)$.
\end{thm}
 \begin{proof}[{\bf Proof.}] First $S(0) = $ Identity. Moreover for each  $z\in Z$, if $z_n\in Z$ and $z_n\rightarrow z$ as $n\rightarrow+\infty$, as a consequence of the Gronwall \index{Gronwall Lemma} Lemma (lemma \ref{Gronwall Lemma}) we obtain classically :
$$S(t)z_n\rightarrow S(t)z \ \hbox{ as } \ n\rightarrow +\infty, \hbox{uniformly on
}[0,T] $$ for each finite T . In particular $S(t)\in C(Z,Z)$ for
all $t \geq0.$ Moreover for each $y\in Z$, the calculation
performed in the proof of Lemma \ref{Lemma 4.1.9.} shows that
$$S(t)S(s)y =S(t+s)y $$ for all  $s, t  \geq 0$. Finally by construction we have
$S(t)z\in C([0,+\infty),Z)$ for each $z\in Z.$ Hence the
result.
 \end{proof}

    As a particular case of Theorem \ref{Theorem 4.1.10.}, we can choose $X = {\mathbb R}^N, N \geq 1.$ For each
vector field
 $F \in W^{1, +\infty}_{loc}({\mathbb R}^N, {\mathbb R}^N)$ we consider the \index{autonomous} (autonomous)
differential system

\begin{equation}\label{Equ4.8} u'(t) = F(u(t))\end{equation}
and its integral curves $u(t) = : S(t)x$  defined for $t\in
[0,\tau(x))$. Theorem \ref{Theorem 4.1.10.} says that if  $\tau(y)  = +\infty $
and the corresponding local solution $u(t)$ remains bounded for $t
\geq0,$ then  $\tau(z) = +\infty$ for each $z\in Z:=
\overline{u({\mathbb R}^+)}$ and the restriction of $S(t)$ to $Z$
(endowed with the distance associated to the norm) is a dynamical
system. To see this we apply Theorem \ref{Theorem 4.1.10.} with $A = 0$ and $Y =
\{y\}$.\\

Other important examples of dynamical systems \index{dynamical system} will be associated to the partial differential equations studied in Chapter \ref{chapter2}. Their properties will be studied precisely in the next chapter. 
 \section{Some easy examples}
    In the first section (Theorem \ref{thm 4.1.8.}), we showed that the  $\omega$-limit \index{$\omega$-limit set}  set  of
a  precompact  trajectory $u(t) = S(t)z $ is a continuum invariant
under $S(t)$ and which (by construction!) attracts the  trajectory
as $t\rightarrow +\infty.$ In some cases this gives directly a convergence result. As a first easy case we have 
\begin{prop}\label{prop4.2.1} If $ \omega(z)$ is discrete, there exists $a\in Z$ such that 
$ d(S(t)z,a)\rightarrow 0 $ as $t\rightarrow +\infty$
\end{prop}
 \begin{proof}[{\bf Proof.}] This is an immediate consequence of Theorem \ref{thm 4.1.8.}. Indeed, $ \omega(z)$ , being compact and discrete is finite. But a connected finite set is reduced to a point.
 \end{proof}
 As an example let us  consider  the second order ODE  $$ u'' + u' + u^3 - u  =  0.  $$ 

All  solutions are  global\index{global solution}  and an immediate calculation gives:
        $$  (d/dt) [(1/2)u'^2+ (1/4)u^4 - (1/2)u^2] = - u'^2 \leq 0 .$$
    Hence we can define the dynamical system \index{dynamical system} generated on the whole of $ \R ^2 $ by setting $U(t) = (u(t), u'(t)) $  and writing the equation as a first order system. The function $t \mapsto [(1/2)u'^2+ (1/4)u^4 - (1/2)u^2](t)$  is nonincreasing along trajectories. Consequently it has a limit as $t$ tends to infinity and, as a consequence,  each trajectory (v, v') contained in the $\omega$-limit \index{$\omega$-limit set} set  of  a  given
trajectory satisfies automatically
        $$  0 = (d/dt) [(1/2)v'^2+ (1/4)v^4 - (1/2)v^2] = - v'^2 .$$
It follows, since this implies  $ v' \equiv 0$, that the $\omega$-limit \index{$\omega$-limit set} set of  any trajectory consists of  stationary points and is therefore contained in $\{ 0, 1,
-1\}\times\{0\}$. By connectedness, the $\omega$-limit \index{$\omega$-limit set} set
reduces to a singleton $\{(z, 0)\}$ with  $z  = 0, 1$ or (-1). Therefore every
solution has a limit at  $+\infty$. \\

Actually the argument which we gave above in this special case is general for systems having what will be called a  "strict Liapunov \index{strict Liapunov function} function". On the other hand already in $ \R ^2 $ there are many examples of systems with non-convergent bounded trajectories. For instance the basic second order equation $$ u'' + \omega^2 u  = 0 $$ has no convergent trajectory except $u = 0$. Here instead of a Liapunov \index{Liapunov function} function we have an invariant energy, and the $\omega$-limit \index{$\omega$-limit set}  set  of  any solution other than the single \index{equilibrium point} equilibrium point $(0, 0)$ does not intersect the set of equilibria.\\

\section{Convergence and equilibrium points}
\index{equilibrium point}
In this section we introduce some general concepts which will be used throughout the text. \begin{defn}\label{Definition4.3.1.}
 Let  $z\in Z$.  The trajectory $t\rightarrow S(t)z$ is called convergent if there is $a\in Z$ such that  $$ \lim_{t\rightarrow +\infty} d(S(t)z, a) = 0.$$
\end{defn}

\begin{defn}\label{Definition4.3.2.}
 A point  $z\in Z$ is called an \index{equilibrium point} equilibrium point (or equivalently a stationary point) of the dynamical system \index{dynamical system} $S(t)$ if $\{z\}$ is invariant under $S(t)$, i.e. $$\forall t\ge 0,\ S(t)z = z .$$\end{defn}
 
 The following property is now obvious \begin{prop}\label{prop4.3.3}If a trajectory of the dynamical system \index{dynamical system} $S(t)$ is convergent, the limit is always a stationary point.
\end{prop}
\begin{proof}[{\bf Proof.}] This is an immediate consequence of Proposition \ref{prop4.1.7}. Indeed if a trajectory converges, it is precompact and the omega-limit set is an invariant singleton.
 \end{proof}
 
 \begin{rem}\label{Remark4.3.4.} {\rm As a trivial consequence of Proposition \ref{prop4.3.3}, a necessary condition for a precompact  trajectory to be convergent is that its  $\omega$-limit \index{$\omega$-limit set}  set  be made of equilibria. In  chapter 6 we shall study an important class of systems for which the $\omega$-limit \index{$\omega$-limit set}  set  of all precompact  trajectories is reduced  to equilibria. Then if the set of equilibria is finite,  convergence follows from Proposition \ref{prop4.2.1}. On the other hand an important part of the book will be devoted to the harder case of a continuously infinite set of equilibria.}\end{rem}
 
\section{Stability  of equilibrium points}\index{equilibrium point} Another important concept concerning equilibria (and more generally trajectories) of a dynamical system \index{dynamical system} is the concept of stability as defined by \index{Liapunov} Liapunov.
 
 \begin{defn}\label{Definition 4.3.5.}
 An equilibrium \index{equilibrium point} point  $a$  of the dynamical system \index{dynamical system} $S(t)$ is called stable \index{stable} (under $S(t)$) if 
 $$\forall \varepsilon >0, \quad \exists \delta>0, \quad  \forall z\in Z, d(z, a) <\delta  \Longrightarrow  \forall t >0, 
 \quad   d(S(t) z, a) <\varepsilon.$$
Otherwise we say that $a$ is \index{unstable}unstable.
 \end{defn}
 
 The following result, relted to the concept of Liapunov \index{Liapunov function} function,  provides a general stability criterion applicable even to infinite dimensional systems.
 
  \begin{thm}\label{stab-gen} Let  $a\in Z$ be an equilibrium \index{equilibrium point} point of the dynamical system \index{dynamical system} $S(t)$ and $U$ be an open subset of $Z$ with  $a\in U$ such that  for some $V \in C(Z)$ we have
 
   \begin{equation} \label{>} \forall r\in (0, r_0) , \min _{d(u, a) = r} V(u) > V(a)  \end{equation}
   
   $$ \forall u \in U, \quad  \forall t\ge 0, \quad V(S(t)u \le V(u)$$  
 
Then $a$ is a stable \index{stable} equilibrium \index{equilibrium point} point of the dynamical system \index{dynamical system} $S(t)$.
\end{thm}
\begin{proof}[{\bf Proof.}]  Let $r>0$ be such that $\overline{B}(a, r)\subset U$ and let
 $$ c: = \min _{d(u, a) = r } V(u) > V(a) $$
 Let $$W = \{ u\in B(a, r), V(u) < c\} $$ 
 It is clear that $W$ is open with $a\in W$. In addition if $u_0\in W$, $u(t) = S(t)u_0$ satisfies  
 $$\forall t\ge 0, \quad u(t) \in W  $$ Indeed if this property fails for some $u_0\in W$, we can consider $$ t_0 = \inf \{ t \ge 0, \quad u(t) \not\in W \}.$$ We have $V(u(t_0)) \le V(u_0)<c$ and since $W$ is open the only possibility is $ d(u(t_0) , a) = r $, a contradiction with the definition of $c$. The result is now  immediate since $r>0$ can be chosen arbitrarily small. \end{proof}
Under the hypothesis that balls with finite radius are compact subsets, we obtain the following result applicable in finite dimensions. 
\begin{cor}\label{stab-comp} Assuming that closed balls with finite radius are compact subsets of Z ,  let  $a\in Z$ be an equilibrium \index{equilibrium point} point of the dynamical system \index{dynamical system} $S(t)$and $U$ be an open subset of $Z$ with  $a\in U$ such that  for some $V \in C(Z)$ we have
$$ \forall u \in U, u\not = a \Rightarrow V(u)> V(a)   $$  
$$ \forall u \in U, \quad  \forall t\ge 0, \quad V(S(t)u \le V(u)$$  
Then $a$ is a stable \index{stable} equilibrium \index{equilibrium point} point of the dynamical system \index{dynamical system} $S(t)$ .
\end{cor}
\begin{proof}[{\bf Proof.}]  Let $r>0$ be such that $\overline{B}(a, r)\subset U$ : as a consequence of the compactness of closed balls we have \eqref{>}. The result is now  an immediate consequence of Theorem \ref{stab-gen}.\end{proof}
 
\begin{defn}\label{Definition4.3.6.}
 An equilibrium \index{equilibrium point} point  $a$  of the dynamical system \index{dynamical system} $S(t)$ is called asymptotically stable \index{asymptotically stable} (under $S(t)$) if it is stable \index{stable} and in addition 
  $$\exists \delta_0>0, \quad  \forall z\in Z, d(z, a) <\delta_0  \Longrightarrow  
 \quad   \lim _{t\rightarrow +\infty}d(S(t) z, a) = 0 $$ \end{defn}

  \begin{rem}\label{Remark4.3.7.}{\rm The first order ODE $$  u' + u^3 - u  =  0  $$ generates a dynamical system \index{dynamical system} on $Z= \R$  which has a set of 3 equilibria $\{ -1, 0, +1\}$. It is easy to verify that all trajectories of this system are convergent, positive initial data lead to a trajectory converging exponentially fast to +1, negative initial data to a trajectory converging exponentially fast to -1. Therefore +1 and -1 are asymptotically stable, \index{asymptotically stable} whereas 0 is  \index{unstable} unstable.  It is not too difficult to check that the equilibria $(1, 0)$ and $(-1, 0)$ are also asymptotically stable \index{asymptotically stable} for the system generated in $Z= \R^2$ by the second order ODE   $$ u'' + u' + u^3 - u  =  0  $$ considered in the previous section, whereas in this case the set of initial data leading to a trajectory tending to $(0, 0)$ is a 1D curve separating the attraction basins of the 2 stable \index{stable} equilibria. Hence $(0, 0)$ is also unstable \index{unstable} in this case.\\
In the case of the basic oscillator governed by $$ u'' + \omega^2 u  = 0 $$
the only equilibrium \index{equilibrium point} is $ 0 $ which is stable \index{stable} (with   $\delta = \varepsilon $  since we have an isometry group on $Z= \R^2$) but not \index{asymptotically stable} asymptotically stable. This result can also be viewed as a special case of theorem \ref{stab-gen} with $V(u, u') = \frac{1}{2} (u'^2+ \omega u^2).$ The same argument holds true for the wave \index{wave equation} equation with $V$ the usual energy functional. We remark that except for the initial data (0, 0), the  omega-limit set does not cross the set of equilibria. In fact if the omega-limit set of a trajectory contains a stable \index{stable} equilibrium \index{equilibrium point} point, the trajectory must converge to this point. This makes the study of convergence somewhat easier when the dynamics is \index{unconditionally stable} unconditionally stable, a typical case being contraction (or more generally uniformly equicontinous) semi-groups which will be studied in Chapter \ref{Somebasicexamples}.}\end{rem}
\chapter[The linearization method ]{The linearization method in stability analysis }

 When looking for stability of an equilibrium \index{equilibrium point} point $a$ for an evolution equation $U'+ {\cal A}U = 0$, a natural idea is to examine the nature (convergent or divergent)  of the linear semi-group generated by the linearized operator $D {\cal A}(a)$.  It is intuitively clear that this will work only when the spectrum of $D {\cal A}(a)$ does not intersction the imaginary axis. In this chapter, we first describe an extension of the Liapunov \index{Liapunov} linearization method to establish the asymptotic stability of equilibria. The perturbation argument developed here is
applicable, in conjonction with the linear results of Chapter 2, to various semi - linear evolution problems on infinite dimensional Banach spaces. At  the opposite, an argument essentially coming back to R. Bellman \cite{MR0061235}  allows to deduce instability from the existence of an eigenvalue with the "wrong" sign. We shall also provide an infinite dimensional version of the linearized instability principle. 

\section{A simple general result}

    Let $X$ be a real Banach space, $T(t)$ a strongly continuous linear semi-group on  $X$,
and $F: X\longrightarrow X$ locally Lipschitz continuous on bounded subsets. For any
$x\in X,$ we consider  the unique maximal solution $u\in C([0,\tau(x)),X)$ of the
equation
\begin{equation}\label{Equ3.1}
u(t) = T(t)x +\int_0^t T(t-s)F(u(s))ds, \quad\forall t\in[0,\tau(x))\end{equation}
By a stationary solution of \eqref{Equ3.1} we mean a constant vector $a\in X$ such that
\begin{equation*}a=T(t)a +\int_0^t T(t-s)F(a) ds,\quad\forall t\geq 0\end{equation*}
The following result is an easy consequence of the general theory of strongly continuous linear semi-groups. Let $ L $ denote the generator of $T(t)$. Then we have
\begin{lem}
A vector $a\in X$ is a stationary solution of \eqref{Equ3.1} if and only if we have
\begin{equation*}
a\in D(L) \quad \hbox{and}\quad    La + F(a) = 0.
\end{equation*}
\end{lem}
We are now in a position to state the main result of this section
\begin{thm}\label{thm3.1.2.} Assume that for some constants $\delta> 0, M \geq 1$ we have
\begin{equation}\label{equ3.4} \forall t \geq 0,\  \Vert T(t)\Vert \leq M e^{- \delta t}. \end{equation}
Let $a\in X$ be a stationary solution of \eqref{Equ3.1} such that
\begin{equation}\label{equ3.5}\exists R_0 > 0 ,  \exists \eta> 0 :    \Vert F(u) - F(a)\Vert \leq \eta \Vert u - a\Vert\
\hbox{ for }\,  \Vert u - a \Vert \leq R_0 \end{equation}
with
$$  \eta <{\delta\over M}. $$
Then for all $x\in X$ such that
$$ \Vert x - a\Vert \leq R_1 =  {R_0 \over M } $$
the solution u of \eqref{Equ3.1} is global \index{global solution} and satisfies
\begin{equation}\label{equ3.8}\forall t \geq 0,  \quad \Vert u(t) - a \Vert \leq M \Vert x - a\Vert  e^{- \gamma t},
\end{equation}
with :  $   \gamma = \delta - \eta M > 0.$
\end{thm}
 \begin{proof}[{\bf Proof.}] On replacing $u$ by $u - a$  and $F$ by $F - F(a)$, we may assume $a = 0$
and $F(a) = 0$ with  $\Vert F(u)\Vert \leq\eta\Vert u \Vert$  whenever  $\Vert u \Vert  \leq R_0 .$ In particular,
setting
$$T = \hbox {Sup} \{t \geq 0, \Vert u(t)\Vert \leq R_0 \} \leq +\infty, $$
we find$$\forall t \in[0, T),  \quad \Vert u(t)\Vert \leq M \Vert x \Vert e^{- \delta t}+\eta M
\int_0^t    e^{- \delta (t-s)}\Vert u(s)\Vert\, ds    .$$
Letting  $\varphi(t) = e^{ \delta t} \Vert u(t)\Vert$, we obtain $$\varphi(t)\leq C_1 +
C_2\int_0^t\varphi(s)ds\quad \forall t\in [0, T)$$
with:   $C_1 =M\Vert x\Vert,       C_2 =\eta M.$ By applying \index{Gronwall Lemma} Lemma \ref{Gronwall Lemma} with
$\lambda(t)\equiv C_2$  we deduce
\begin{equation}\label{Equ3.9}
\forall t\in [0, T), \quad e^{\delta t} \Vert u(t)\Vert \leq M \Vert x \Vert e^{\eta
M t}.
 \end{equation}
Since $\delta > \eta M$ , we conclude that if $M\Vert x\Vert \leq  R_0 $, then $T = +\infty$ and \eqref{Equ3.9}
holds true on $[0, +\infty )$. This completes the proof of \eqref{equ3.8}.
 \end{proof}

\section{The classical Liapunov \index{Liapunov stability} stability theorem}

\subsection{A simple proof of the classical Liapunov stability \index{Liapunov stability} theorem}

The object of this
paragraph is to give a simple proof of the following well known result:
\begin{thm}\label{Liapunov}   (Liapunov)\index{Liapunov}  Let X be a finite dimensional
normed space, and $f\in C^1(X, X)$ a vector field on X. Let $a\in X$ be such that f(a)
= 0  and assume
$$  \hbox{All eigenvalues $\{s_j,\ 1\leq j\leq k\}$ of $Df(a)$ have negative real parts.}    $$
Then a
is an asymptotically Liapunov \index{asymptotically stable} stable equilibrium \index{equilibrium point} solution of the equation
\begin{equation}\label{ODEFirstOrder} u'= f(u(t)),\quad   t \geq 0.     \end{equation}
More precisely : for each  $ \displaystyle \delta < \eta = \min_{1\leq j \leq k}\{ -Re(
s_j )\}$, there exists $\rho =\rho(\delta) > 0$ and
$ M(\delta) \geq 1 $ such that if  $\Vert x - a\Vert \leq \rho(\delta)),$  the solution u of \eqref{ODEFirstOrder}
such that $u(0)=x$ is \index{global solution} global with
 $$\forall t \geq 0,  \quad \Vert u(t) - a\Vert \leq M(\delta) \Vert x
- a\Vert e^{- \delta t} .$$
\end{thm}
 \noindent {\bf Proof. } We consider first the case where $a = 0$ and $f$ coincides with a  linear
operator $A$. In this case, the question reduces to the following:

\begin{lem} Let $X$ be a finite dimensional complex vector space, $A\in
L(X)$ and $u\in C^1({\mathbb R}, X)$ a solution of $u'(t) = Au(t)$. Then we have
\begin{equation}\label{Equ3.12}
u(t) =\sum_{j= 1}^kP_j(t) e^{s_j t}
\end{equation}
where $\{sj\}_{1\leq j\leq k} $ is the sequence of eigenvalues of $A$ and $P_j$  a polynomial with coefficients in  $X$ for all $j$.
\end{lem}
 \begin{proof}[{\bf Proof.}] By induction on $\dim _{{\mathbb C}} (X) = p.$\\
- If  $\dim _{{\mathbb C}} (X) = 1,$ then  $j = 1$ and $A = s_1I,$ hence  $u(t) = u_0 e^{s_1 t}.$\\
- If  $\dim _{{\mathbb C}} (X) = p>1,$ assuming that the result is true for all complex vector
spaces with complex  dimensions $\leq p - 1$, we set
$$v(t) = u(t) e^{-s_1 t}, $$
therefore $v$ is a solution of
$$v' = (A - s_1I)v.$$ 
Then setting  $Y = R(A - s_1I),  B = (A - s_1I) \vert  _{Y}$   and $w = v',$ it is clear that $w$
is a solution of
    $$  w\in C^1({\mathbb R}, Y) ; \quad w'(t) = Bw(t). $$
Since by construction $\ker (A - s_1I) \not= \{0\}$, we have  $R(A - s_1I) \not=X $ and in
particular
$$ \dim _{{\mathbb C}} (Y) \leq \dim _{{\mathbb C}} (X) - 1 = p - 1.$$
By the induction hypothesis we have$$ w(t) =\sum_{j= 1}^kQ_j(t) e^{(s_j-s_1) t}$$
because the eigenvalues of  $B$ are of the form  $s_j-s_1.$ By  integrating we obtain
$$ w(t) =a_1 + \sum_{j= 1}^kR_j(t) e^{(s_j-s_1) t}$$then on multiplying by $e^{s_1
t}$, we obtain \eqref{Equ3.12}, completing the proof by induction.\medskip
 \end{proof}
  \begin{proof}[{\bf Completion of the proof of Theorem \ref{Liapunov}}]
Since all eigenvalues of $ Df(a)=:A$ have negative real parts, it follows obviously
from \eqref{Equ3.12} that $\Vert  e^{tA }\Vert \leq C(\delta) e^{- \delta t} $ for all
$\delta < \eta = \min_{1\leq j \leq k}\{ -Re(
sj )\}$. Then we apply Theorem \ref{thm3.1.2.} with $T(t) = e^{tA }$, and  $F$ defined by the
formula
$$  F(u) = f(u) - Df(a)(u-a).$$
The result follows at once.
\end{proof}
\subsection{Implementing Liapunov's first method \index{Liapunov}}

Theorem \ref{Liapunov} gives an apparently simple and almost optimal way of checking the asymptotic stability of a given equilibrium \index{equilibrium point} point of a differential system : check whether all (complex) eigenvalues of the linearization at this point have negative real parts. However in practice we have to check this property on the characteristic polynomial, but as soon as $N\ge 3$ in general the roots cannot be computed .\medskip

\begin{defn} We say that a polynomial $P$ with real coefficients $$P(X) = \sum _{j = 0}^{N} p_jX^j$$
is a Hurwitz polynomial if all its zeroes have negative real parts. \end{defn}

\begin{prop}\label{prop7.1.2} If $P$ is a Hurwitz polynomial, then $p_0\not = 0$ and for each $j \in \{0, ...N\}$, we have $p_jp_0>0$.
\end{prop}
 \begin{proof}[{\bf Proof.}] We have $$ P(X) = p_N \prod _{k}(X+ \lambda_k)\prod _{j }(X+ \mu_j+ i\nu_j)(X+ \mu_j- i\nu_j)$$
  where  all numbers $ \lambda_k,  \mu_j$ are positive . But  $$(X+ \mu_j+ i\nu_j)(X+ \mu_j- i\nu_j) = 
  X^2 + 2 \mu_jX + \mu_j^2 + \nu_j ^2 $$ The result follows immediately by expanding $P$. \end{proof}
  \begin{rem}\label{Remark 7.1.3.} {\rm The converse of Proposition \ref{prop7.1.2} is false if $N>2$ . If all coefficients of $P$ have the same sign, of course $P$ cannot have a positive real root but on the other hand the polynomial 
  $$ P{\varepsilon} (X) = (X+1) (X^2 -\varepsilon X + 1) = X^3 + (1-\varepsilon ) X^2 + (1-\varepsilon ) X + 1$$ has all its coefficients positive for $ 0<\varepsilon< 1$ , although the two conjugate imaginary roots have imaginary parts equal to $ {\varepsilon\over 2}$.} \end{rem}
It is sometimes useful to remember the following criterion which we give without proof :

\begin{prop}\label{prop7.1.4}  For $N\le 4$ a polynomial  $P$ of degree $N$ with $p_0>0$ is a Hurwith polynomial if and only if the following inequalities hold true

\medskip \noindent
- If  $N=2$: $p_1>0, p_2>0. $\\
- If   $N=3$: $p_1>0, p_3>0, p_2p_1> p_3p_0 $\\
- If   $N=4$: $p_1>0, p_3>0, p_4>0, p_3(p_2p_1-p_3p_0) >p_4 p_1^2  $
\end{prop}

 \begin{rem}\label{Remark 7.1.5.} {\rm The general conditions for $N\ge 5$ become complicated and  are known as the Routh-Hurwitz criterion. The criterion  consists in N inequalities which can be computed  either using the diagonal (N-1) dimensional minors of some $N\times N $ matrix (cf. \cite{MR0140742} ) or through a step by step inductive procedure involving only some determinants of order 2.}
\end{rem}
 \subsection{Remarks on  Liapunov's \index{Liapunov stability} original proof of the stability theorem}\label{RemarkLiapunov}
 The original method of Liapunov \index{Liapunov} consisted in introducing the quadratic form
$$ \Phi(u) = \int_0^{+\infty}\Vert T(t)u\Vert^2 dt$$
where $T(t) = \exp (t A).$ For a solution of the equation
$$  u' = Au + F(u)$$
we have 
\begin{eqnarray*} 
{d\over{dt}}\Phi(u(t))&=& 2\int_0^{+\infty}(T(s)u(t), T(s)u'(t))\,ds \\
&=&2\int_0^{+\infty}(T(s)u(t), T(s)Au(t)+T(s)F(u(t)))\,ds.
\end{eqnarray*} 
But $$\int_0^{+\infty}(T(s)u(t), T(s)Au(t))ds = \int_0^{+\infty}(T(s)u(t),
{d\over{ds}}T(s)u(t))ds = -{1\over2}\Vert u(t)\Vert^2$$ and
$$\Bigl\vert2\int_0^{+\infty}(T(s)u(t),T(s)F(u(t)))ds\Bigr\vert\leq2C\Vert
u(t)\Vert\Vert F(u(t))\Vert.$$
The result then follows for  $\Vert F\Vert_{Lip} $  small enough. On this proof we want to
make two observations that will justify our choice of a perturbation argument in integral form :

 1) Even when $F = 0$, the decay rate obtained by Liapunov's \index{Liapunov} method is not optimal. For
instance if $X = {\mathbb R}^N$ and we apply the above estimates to the equation $$u'' + u +
2u' = 0,$$  we  obtain      $$  \Vert T(t) \Vert \leq C e^{- (1- \sqrt 2/2) t }$$

which is not optimal since in fact   $$ \Vert T(t) \Vert \leq C (1+t) \exp(- t ).$$

2) When $F=0$, the quadratic form $\Phi$ does not provide the decay in the correct space if $X$ is an infinite-dimensional Hilbert space. If, for instance, we consider the heat equation \index{heat equation}
$$u_t -  \Delta u  = 0  \quad\hbox{in}\quad   {\mathbb R}^+ \times \Omega, \quad   u =
0  \quad\hbox{ on} \quad {\mathbb R}^+ \times \partial\Omega$$ in a bounded open domain of $\R^N$ 
which generates a contraction semigroup $T(t)$ on $X = L^2(\Omega)$, the quadratic form $\Phi$ does not control the norm in $X$. Indeed , if $\varphi_n$ is an eigenfunction of the operator $-\Delta$ , i.e
$$-  \Delta \varphi_n  = \lambda_n  \varphi_n \quad\hbox{in}\quad   \Omega, \quad   u =
0  \quad\hbox{ on} \quad  \partial\Omega$$ it is immediate that
$$ \Phi(\varphi_n) = \int_0^{+\infty}\Vert T(t)\varphi_n\Vert^2 dt = ||\varphi_n ||^2 \int_0^{+\infty}e^{-2\lambda_n} dt  = \frac{1}{2\lambda_n}||\varphi_n ||^2 . $$\\

  3) The introduction of $\Phi$  is only possible when $X$ is a Hilbert space. If, for
instance, we work with the semilinear equation

         $$u_t -  \Delta u + f(u) = 0  \quad\hbox{in}\quad   {\mathbb R}^+ \times \Omega, \quad   u =
0  \quad\hbox{ on} \quad {\mathbb R}^+ \times \partial\Omega$$
and we try to apply Liapunov's \index{Liapunov} result with $X = L^2(\Omega)$ , we shall be very
limited in our range of application. Indeed in order for the operator $F$ defined by
    $$  (F(u))(x) = f(u(x)), \quad  \hbox{a.e.  in } \quad\Omega$$
to satisfy the condition $$\Vert F(u)\Vert_X \leq \varepsilon \Vert u\Vert_X  \quad \hbox{for}
\quad\Vert u\Vert_X\quad  \hbox{small}$$ it is necessary (and sufficient , of course) that $f$ satisfy the global condition $$\vert f(s)\vert  \leq \varepsilon\vert s\vert , \quad  \forall s\in{\mathbb R}.$$
As a consequence, $F$ cannot be tangent to $0$ at the origin, except if $F = 0.$ The situation is very different if $X = C_0(\Omega)$: in this case, in order for the operator $F$ to satisfy the condition $$\Vert F(u)\Vert_X \leq \varepsilon\Vert u\Vert_X  \quad \hbox{for}
\quad\Vert u\Vert_X\quad\hbox{small}$$ it is sufficient that f satisfy the local condition
$$\vert f(s)\vert  \leq \varepsilon\vert s\vert , \quad  \hbox{for all } s \hbox{ small enough.}$$
In particular, if $f$  is a function of class $C^1$ and $f'(0) = 0$, $ F$ is tangent to $0$ at the origin. Considering for instance the equation
$$u_t -  \Delta u  = \vert u\vert ^{p-1}u
\quad{in}\quad   {\mathbb R}^+ \times \Omega, \quad   u = 0  \quad \hbox{on} \quad {\mathbb R}^+
\times \partial\Omega.$$
The original Liapunov \index{Liapunov} technique does not give the stability of the 0 solution when
working in $L^2(\Omega)$.
The method will work if we replace  $L^2(\Omega)$ by some Sobolev space of type $H^m(\Omega)$, but
then we need some growth conditions \index{growth condition} on the nonlinearity, imposing extraneous
limitations on p. If $X = C_0(\Omega)$, we obtain easily the stability of the 0
solution for any $p>1$, cf. Proposition \ref{Prop3.3.1.}.\bigskip

\section{Exponentially  damped systems governed by
PDE}
\subsection{Simple applications}

    In this paragraph, we show how the stability theorem \ref{thm3.1.2.} can be applied to partial differential equations.

   a) We first consider the semilinear heat \index{heat equation} equation  \eqref{heat} :
$$ u_t - \Delta u + f(u) = 0 \quad\hbox{  in} \, \,{\mathbb R}^+ \times \Omega,
\quad     u = 0  \,\,   \hbox{on} \, \, {\mathbb R}^+ \times \partial\Omega$$ where  $\Omega$ be any open set in  ${\mathbb R}^N$ with  Lipschitz continuous boundary
$\partial\Omega$ , and $f:{\mathbb R}\longrightarrow {\mathbb R} $  is a function of class $C^1$   with
$$f(0) = 0\quad \hbox{and}\quad f'(0)>-\lambda_1(\Omega).$$
 We
have the following simple result :
\begin{prop}\label{Prop3.3.1.} Under the above hypotheses, the stationary solution
$u\equiv 0 $ of \eqref{heat} \index{heat equation} is exponentially  stable in $X = C^0(\Omega)$. More precisely  : for each  $\gamma \in (0, \lambda_1(\Omega) + f'(0)),$ there exists  $R = R(\gamma)$ such that for all $x\in X$ with  $\Vert x\Vert \leq R,$ the solution $u$ of \eqref{heat} such that $u(0) = x$  is global \index{global solution} and satisfies
$$\forall t \geq 0,  \Vert u(t)\Vert \leq M \Vert x\Vert e^{-\gamma t},$$
with M independent of $\gamma$  and x.
\end{prop}
 \begin{proof}[{\bf Proof.}] We have shown in  Theorem \ref{thm2.2.2}  that the contraction semi-group $T_0(t)$
generated in $ C^0(\Omega)$ by the equation
     $$ u_t - \Delta u  = 0 \quad\hbox{  in} \, \,{\mathbb R}^+ \times \Omega,
\quad     u = 0  \,\,   \hbox{on} \, \, {\mathbb R}^+ \times \partial\Omega       $$
satisfies \eqref{equ3.4} with  $\delta =\lambda_1(\Omega)$  and some $M > 1.$ It is therefore
sufficient to apply Theorem \ref{thm3.1.2.}  with $T(t) = e^{- f'(0)t}\, T_0(t),$ since for
$f\in C^1({\mathbb R}),  F(u) = f(u) - f'(0)u $ satisfies  \eqref{equ3.5} with $a = 0$ and $\eta$
arbitrarily small.
 \end{proof}

    b) Similarly we can consider the semilinear wave \index{wave equation} equation \eqref{wave}
\begin{equation*}\label{Equ3.16} u_{tt} - \Delta u +\gamma u_t + f(u) = 0 \quad\hbox{  in} \, \,{\mathbb R}^+ \times
\Omega,
\quad     u = 0  \,\,   \hbox{on} \, \, {\mathbb R}^+ \times\partial\Omega       \end{equation*}
where $\Omega$ is a bounded  open set in  ${\mathbb R}^N$ with  Lipschitz continuous
boundary
$\partial\Omega$, and $f$  is a function of class $C^1$: ${\mathbb R}\rightarrow {\mathbb R} $ with  $$f(0) = 0\quad \hbox{and}\quad
f'(0)>-\lambda_1(\Omega).$$ satisfying the growth condition \index{growth condition} \eqref{GrowthCondi}. We obtain the following  result :
\begin{prop}  Under the above hypotheses, the stationary solution $(u ,
v) \equiv (0, 0)$ of \eqref{wave} \index{wave equation} is exponentially  stable in $ X
= H^1_0(\Omega)\times L^2(\Omega)$ in the following
sense: for each $\delta>0$ small enough, there exists $R = R(\delta)$ such that for all $x\in X
$ with  $\Vert x\Vert \leq R,$ the solution $u$ of \eqref{wave} \index{wave equation} such that $(u(0), u_t(0)) = x$  is
global \index{global solution} and satisfies
\begin{equation}\label{Equ3.19}\forall t \geq 0,  \Vert u(t)\Vert \leq M(\delta) \Vert x\Vert e^{-\delta t}.
  \end{equation}
  \end{prop}
 \begin{proof}[{\bf Proof.}] It follows from  Proposition \ref{Proposition2.3.3.}  that the contraction semi-group $T_0(t)$
generated in  $ X = H^1_0(\Omega)\times L^2(\Omega)$  by the equation
\begin{equation}\label{Equ3.20} u_{tt} - \Delta u + f'(0)u +\gamma u_t = 0 \quad\hbox{  in} \, \,{\mathbb R}^+ \times
\Omega,
\quad     u = 0  \,\,   \hbox{on} \, \, {\mathbb R}^+ \times \partial\Omega
\end{equation}

  satisfies
\eqref{equ3.4} with   some $M > 1$ for any $\delta>0$ small enough.  In order to apply Theorem \ref{thm3.1.2.} with $T(t)$ the semi-group generated by \eqref{Equ3.20}, all we need to check is that the function
$F(u, v)) = - (0, f(u) - f'(0)u )$ satisfies \eqref{equ3.5} with  $ a = 0$ and $\eta$
arbitrarily small. But this is immediate since the function $\varphi(s) = f(s) -
f'(0)s$  is $o(\vert s\vert )$ near the origin and, by \eqref{GrowthCondi} we have $\vert \varphi(s)\vert  \leq C(\vert s\vert ^r)$
for s large. Therefore for each $d> 0$ arbitrarily small, we have $\vert \varphi(s)\vert  \leq
d\vert s\vert  + C(d)\vert s\vert ^r,$ globally on ${\mathbb R}$. The result then follows immediately from
Sobolev imbedding theorems.
 \end{proof}

\subsection[Positive solutions of a heat equation\index{heat equation}]{Exponentially  stable positive solutions of a heat
equation}
In this paragraph, we  consider the semilinear heat equation \index{heat equation}
\begin{equation*} 
 u_t - \Delta u + f(u) = 0 \quad\hbox{  in} \, \,{\mathbb{R}}^+ \times
\Omega,
\quad     u = 0  \,\,   \hbox{on} \, \, {\mathbb{R}}^+ \times \partial\Omega
\end{equation*}
 where  $\Omega$ be any open set in  $\R^N$ with
 Lipschitz continuous boundary
$\partial\Omega$ , and $f$  is a function of class $C^1$: $\R\rightarrow\R $ with $f$ convex on  $\R^+, f(0) = 0 $ and  $$
f'(0)<-\lambda_1(\Omega).$$
 We have the following simple result :
\begin{prop}   Under the above conditions, assuming that $f(s)>0$ for
some $s>0$, there exists a unique solution $\varphi > 0$ of
\begin{equation}\label{Equ5.7}
 - \Delta \varphi +
f(\varphi) = 0 \quad\hbox{  in} \, \, \Omega,
\quad    \varphi = 0  \,\,   \hbox{on} \, \,  \partial\Omega.
\end{equation}
 In addition, $\varphi $ is asymptotically (even exponentially) \index{asymptotically stable} stable in
$ C({\overline\Omega})\cap H_0^1(\Omega). $
\end{prop}
 \begin{proof}[{\bf Proof.}] If $a\in l^{\infty}(\Omega)$ we
denote by $\lambda_1(  -
\Delta + aI ) $  the first eigenvalue of $ - \Delta + aI$
in the sense of $H^1_0 (\Omega)$. First of all if \eqref{Equ5.7} has a positive solution
$\varphi$ and we set
$$ p(x) = {f(\varphi(x))\over{\varphi(x)}}$$ we have obviously $$\lambda_1(  -
\Delta + pI ) = 0 $$ with eigenfunction equal to $\varphi$. Now if
$\psi$ is another positive solution, we introduce $$ q(x) =
{{f(\varphi(x))-f(\psi(x)) }\over{\varphi(x)-\psi(x)}} \quad\hbox
{if}\quad \varphi(x)\not=\psi(x)$$
$$q(x) =
f'(\varphi(x))\quad\hbox
{if}\quad \varphi(x)=\psi(x). $$  By strict convexity we have $$ q(x)>p(x)
 $$ everywhere in $\Omega$. In particular  $$\lambda_1(  -
\Delta + qI ) > 0 $$ On the other hand if $\varphi\not
\equiv\psi$, then $\varphi-\psi$ is an eigenfunction of $ ( -
\Delta + qI ) $ with eigenvalue 0. This contradiction means that
$\varphi\equiv\psi$ and therefore $\varphi$ is unique. In addition
since by strict convexity we have $$ f'(\varphi(x))>p(x)
 $$ everywhere in $\Omega$ we have in particular  $$\lambda_1(  -
\Delta + f'(\varphi(x))I ) > 0 $$ as soon as a positive solution
$\varphi$ exists. Therefore we have uniqueness and exponential stability of $\varphi$
as soon as it exists.\\
To prove the existence of $\varphi$, first we deduce from the hypotheses on $f$
that $$ \exists s_0>0, \quad f'(s)\geq f'(s_0)>0,\quad \forall s\geq s_0.   $$ In
particular
\begin{equation}\label{Equ5.8} \lim_{s\rightarrow +\infty} f(s) =\lim_{s\rightarrow +\infty} F(s) =
+\infty \end{equation}
 where $$ F(s)= \int_{0}^s f(\sigma) d\sigma.$$ Therefore $$
\inf_{s\geq 0 }F(s) = C > -\infty.$$ For the proof of existence, first we modify (if necessary) $f$ on ${\R^{-} }$ by setting $$ \forall s<0, \quad
f(-s) = -f(s)$$ And then $F$ is extended as the primitive of $f$. This means $$ \forall
s<0,
\quad F(-s) = F(s)$$ We introduce $$ m = \inf \{ \int_{\Omega}[{1\over 2}\vert\nabla
z\vert^2 + F(z)] dx, \quad z\in H_0^1(\Omega)\,\}\geq C
\vert\Omega\vert
\geq -\infty.$$ Since as $s\rightarrow 0$ we have $$ F(s) \sim -
f'(0){s^2\over 2}
$$ and $f'(0)<-\lambda_1(\Omega)$, by taking $ z = \varepsilon \varphi_1$ and letting
$\varepsilon \rightarrow 0$ we find $$ m<0$$ Since any minimizing sequence is bounded
in $ H_0^1(\Omega)$ and
$ F $ is convex up to a quadratic term, there exists, as a consequence of compactness
in
$ L^2(\Omega)$ and Fatou's Lemma, a function $\varphi \in  H_0^1(\Omega)$ such that $$
 \int_{\Omega}[{1\over 2}\vert\nabla\varphi\vert^2 + F(\varphi)] dx = m $$ Setting $
\psi = \vert\varphi\vert $ we also have, since $F$ is even:$$
 \int_{\Omega}[{1\over 2}\vert\nabla\psi\vert^2 + F(\psi)] dx = m $$ Because $m<0$, of
course $\psi\not = 0$. It is then classical to conclude that $\psi$ is a positive
solution of \eqref{Equ5.7}.
 \end{proof}

 \section{Linear instability and Bellman's approach} \label{Linear instability and Bellman's approach} In any finite dimensional real Hilbert space X , the hypothesis of Theorem \ref{Liapunov} is sharp. Actually if $ f = L$ is linear and has an eigenvalue $ s: = s_1 + i s_2$  with $s_1, s_2 $ real and $s_1\ge 0 $, let 
 $$ L(\varphi_1 + i \varphi_2) = s (\varphi_1 + i \varphi_2)$$ with 
 $\varphi_1 , \varphi_2$ real vectors and $(\varphi_1 , \varphi_2) \not= (0, 0)$.  Then the real vector-valued function $$ u(t) = e^{s_1 t} [ cos (s_2t) \varphi_1 - sint (s_2t) \varphi_2] $$ is a solution of \eqref{ODEFirstOrder} because the function $$z(t) = e^{s t}\varphi $$ is a solution of the extended equation of \eqref{ODEFirstOrder} on the complexification of $X$ and 
$L$ being a real endomorphism on $X$, $\overline {z(t)} =e^{\overline{s }t}\overline{\varphi } $ and $ u = {1\over
2}(\overline{z(t)}- z(t)) $ are solutions of the same equation. But now we observe that $$ u({k\pi\over {s_2}}) = (-1) ^k exp({ {{k\pi}  {s_1}}\over {s_2}})u(0) $$ and therefore $u$ cannot converge to anything at all as $t$ goes to infinity.  \medskip 

In the next paragraph we collect some instability results proved in \cite{MR1842097} in the Hilbert space framework. 

\subsection{The finite dimensional
case}
 Let $X$ be a finite dimensional normed
space, and $f\in C^1(X, X)$ a vector field on $X$. Let $a\in X$ be
such that $f(a) = 0$. By Liapunov's \index{Liapunov} theorem (Theorem \ref{Liapunov}), if all
eigenvalues of $Df(a)$ have negative real parts,  $a$  is an
asymptotically Liapunov stable \index{asymptotically stable} equilibrium \index{equilibrium point} solution of the
equation \begin{equation}\label{FirstOrderODE}  u'= f(u(t)), \quad     t \geq 0.\end{equation}
This result is sharp since in the opposite direction we
have
\begin{thm}\label{Theorem 5.1.1:} (R. Bellman, \cite{MR0061235}) Let $a\in X$ be such that $f(a) =
0$ and assume that at least one eigenvalue of $Df(a)$ has a
positive real part. Then  a is an unstable \index{unstable} equilibrium \index{equilibrium point} solution of
\eqref{FirstOrderODE}.
\end{thm}
 \begin{proof}[{\bf Proof.}] Let $ \eta>0$ be the minimum of real
parts of the eigenvalues of $Df(a)$ having a positive real part
and choose an integer $K$ to be fixed later. The Jordan reduction
theorem implies in particular the existence of an upper triangular
matrix $T$ with zero diagonal terms and coefficients all equal to
0 or 1 such that
$$\frac{K}{\eta}M = D + T
$$ where $M$ is the matrix of $Df(a)$ in a certain basis of $X$ and $ D$ is a complex
diagonal matrix. Then $$  M = L  + R $$ where $L = \frac{\eta}{K}D $ is a diagonal
matrix and $ R = \frac{\eta}{K}T $ is a matrix with all coefficients having
moduli smaller than $\frac{\eta}{K}$. Let us identify $X$ with $H ={\C}^N $ with
the usual Hilbert norm and the associated real inner product. It is clear that the
coefficients of $L$, in other terms the diagonal coefficients of $M$, are in fact the
eigenvalues of in $Df(a)$. In addition under this identification we have $\displaystyle
\Vert R\Vert\leq \frac{\eta\dim X}{K}$. Let $P: H \longrightarrow H$ denote the projection
operator on
\begin{equation}\label{Equ5.1} Y: =\bigoplus_{ Re(\lambda)> 0} Ker(L-\lambda I). \end{equation}
 If $u$ is
any bounded solution of \eqref{FirstOrderODE}, for $t>0$ setting $u= a +v$ we have
$$\frac{d}{dt} (\vert Pv \vert^2-\vert (I-P)v \vert^2) = 2[( Pv, v') -
((I-P)v , v')]$$ $$= 2[( Pv, Lv + Rv +g(v)) -
((I-P)v , Lv + Rv +g(v))]$$  where $ g(v) = f(a+v)-f(a)- Df(a)v $ satisfies $$ g(v) =
o(v)$$Since
$L\leq 0$ on
$[Y]^\perp = (I-P)H$ we have:
$$ -((I-P)v ,Lv) = -(L(I-P)v
,(I-P)v)\geq 0 .$$
 On the other hand by definition of $\eta$ we have $$
\forall w\in Y, (Lw, w)\geq \eta\vert w \vert^2$$ In particular we find $$ 2(Pv
,Lv) = 2(LPv
,Pv)\geq 2\eta\vert Pv\vert^2.$$
 And therefore $$2[( Pv, Lv ) -
((I-P)v , Lv )]\geq 2\eta\vert Pv\vert^2.$$ On the other hand we have the easy
inequality $$ 2[( Pv,  Rv ) -
((I-P)v ,  Rv )]\geq -4\Vert R\Vert\vert v\vert^2\geq -4\frac{\eta\dim X}{ K}\vert
v\vert^2$$ and since $ g(v) =
o(v)$, there exists
$\varepsilon>0$ such that if $\vert v\vert \leq \varepsilon $, we have $$2( Pv,g(v)
) - 2 ((I-P)v, g(v))\geq -{\eta\over2}\vert v\vert^2 = -{\eta\over2} (\vert Pv\vert^2
+\vert (I-P)v\vert^2)$$ Choosing $K = 8\dim X$ and combining the above inequalities we
find
$${d\over{dt}} (\vert Pv \vert^2-\vert (I-P)v \vert^2) \geq \eta (\vert Pv\vert^2
-\vert (I-P)v\vert^2) $$ whenever $\vert v\vert \leq \varepsilon $. Now
assuming that a is \index{Liapunov stability} Liapunov-stable in $X$, let us select $ v(0) = v_0\in X $ such that
\begin{equation}\label{Equ5.3} \vert Pv_0\vert
>\vert (I-P)v_0\vert\end{equation} and $\vert v_0\vert $ small enough so that
\begin{equation}\label{Equ5.4}\forall
t\geq 0,
\vert v(t)\vert \leq \varepsilon. \end{equation} For instance, $ v_0 $ might be
any "small" vector of $Y$. As a consequence of the above computation  it follows that
\begin{equation}\label{Equ5.5}\forall t\geq 0,\quad
(\vert Pv(t) \vert^2-\vert (I-P)v \vert^2) \geq e^{\eta t}(\vert Pv_0\vert^2
-\vert (I-P)v_0\vert^2)). \end{equation}  This is clearly absurd since \eqref{Equ5.3}, \eqref{Equ5.4} and
\eqref{Equ5.5} are incompatible. The proof of Theorem \ref{Theorem 5.1.1:} is complete.
 \end{proof}
\subsection{An abstract instability result}\label{SectionInstability}
 The main
result of this Section is a natural infinite-dimensional extension
of Theorem \ref{Theorem 5.1.1:} to the special case of sel-adjoint linearized operator. 
\begin{thm}\label{Theorem 5.3.1:} Let $H $ be a real Hilbert space with
inner product and norm respectively denoted by $(\cdot,\cdot )$ and $\vert\cdot\vert $, $L$ a (possibly unbounded) self-adjoint \index{self-adjoint} operator such
that $$\exists c>0, \quad L+cI\geq 0.$$
\begin{equation}\label{Equ5.12}
(L+(c+1)I)^{-1}\quad\hbox {is compact}\end{equation}
$$\lambda_1(L): = \inf_{u\in H, u\not=0} {(Lu, u)\over\vert u\vert} < 0.$$
Assume that there exists a Banach space $X\subset H$ with continuous imbedding with
norm denoted by $\Vert . \Vert $ for which $ f: X\longrightarrow H$ is a locally
Lipschitz map with $f(0) = 0$ and such that
\begin{equation}\label{Equ5.14}\lim_{u\in X\setminus \{0\}, \,\Vert u \Vert \rightarrow 0} {\vert f( u)
\vert\over\vert u\vert} = 0.\end{equation}
 Then if $X$ contains all eigenvectors of
L, the stationary solution 0 of
\begin{equation}\label{Equ5.15} u'+ L(u)= f(u)\end{equation} is unstable \index{unstable} in
$X.$\end{thm}
 \begin{proof}[{\bf Proof.}] Let $P: H \longrightarrow H$ denote the projection operator
on $$ H^-: =\bigoplus_{\lambda< 0} Ker(L-\lambda I).
$$ As a consequence of \eqref{Equ5.12} we know that $dim(H^-)<\infty. $ If $u$ is any bounded
solution of \eqref{Equ5.15}, for $t>0$ $u$ is differentiable with values in H and we have
\begin{eqnarray}
\nonumber & &{d\over{dt}} (\vert Pu \vert^2-\vert (I-P)u \vert^2)\\
\nonumber &= & 2[( Pu, u') - ((I-P)u , u')] \\
 \label{Equ5.16} & = &  2[( Pu,f(u) -Lu) + 2((I-P)u ,Lu-f(u))].
\end{eqnarray}
Since $L\geq 0$ on $[H^-]^\perp = (I-P)H$ we have:
\begin{equation}\label{Equ5.17}((I-P)u ,Lu) = (L(I-P)u
,(I-P)u)\geq 0.\end{equation}
 On the other hand by \eqref{Equ5.12} we know that $$\exists \eta>0,
\forall w\in H^-, (-Lw, w)\geq \eta\vert w \vert^2$$ In particular we find
\begin{equation}\label{Equ5.18} 2(Pu,-Lu) = 2(-LPu,Pu)\geq 2\eta\vert Pu\vert^2.\end{equation}
 As a consequence of \eqref{Equ5.14}, there exists
$\varepsilon>0$ such that if $\Vert u\Vert \leq \varepsilon $, we have
\begin{equation}\label{Equ5.19}
2(Pu,f(u)) + 2 ((I-P)u, f(u))\geq -\eta\vert u\vert^2 = -\eta (\vert Pu\vert^2 +\vert
(I-P)u\vert^2).\end{equation}
Combining \eqref{Equ5.16}, \eqref{Equ5.17}, \eqref{Equ5.18} and \eqref{Equ5.19} we find
\begin{equation}\label{Equ5.20}{d\over{dt}} (\vert Pu \vert^2-\vert (I-P)u \vert^2) \geq \eta (\vert Pu\vert^2
-\vert (I-P)u\vert^2)\end{equation} whenever $\Vert u\Vert \leq \varepsilon $. Now
assuming that 0 is \index{Liapunov stability} Liapunov-stable in $X$, let us select $ u(0) = u_0\in X $ such that
$$ \vert Pu_0\vert
>\vert (I-P)u_0\vert$$ and $\Vert u_0\Vert $ small enough so that
\begin{equation}\label{Equ5.21}\forall t\geq 0,
\Vert u(t)\Vert \leq \varepsilon.\end{equation}
As a consequence of \eqref{Equ5.20} we obtain as previously \eqref {Equ5.5}, clearly incompatible with \eqref{Equ5.21}. Consequently if $X$ contains all eigenvectors of $L$, the choice $$u_0
= \eta \varphi;\quad  L \varphi= \lambda \varphi, \quad  \lambda < 0\quad
\vert \eta \vert \Vert \varphi \Vert \rightarrow  0
$$ shows by contradiction that 0 is not Liapunov-stable in $V$. The proof of Theorem \ref{Theorem 5.3.1:} is complete.
 \end{proof}
\begin{rem} {\rm  One might wander why the condition $\displaystyle {\vert f( u)
\vert\over\vert u\vert} \rightarrow 0 $ is required as  $u \rightarrow 0 $ in the
sense of $X$ instead of $H.$ Let us consider the example $H = L^2(\Omega)$ where
$\Omega$ is a bounded open subset of $\R^N $ and $$ \exists g\in C^1\cap W^{1,
\infty}(\R): \forall u\in L^2(\Omega),\,   (f(u))(x) = g(u(x))  \hbox{ a.e. in }
\Omega.$$ In this case, the condition $${\vert f( u)
\vert\over\vert u\vert} \rightarrow 0 \hbox{ as } \vert u\vert\rightarrow 0 $$implies
$f\equiv 0 .$ Indeed if $ f(0)  = 0 $ and $ f(c) \not = 0 $ we can consider $
u_{\omega} = c
\chi_\omega $ with $\omega$ an arbitrary open subset of $\Omega$, so that $$
\vert u_{\omega}\vert = \vert c\vert\vert {\omega}\vert^{1\over 2};\quad \vert
f(u_{\omega})\vert = \vert f(c)\vert\vert {\omega}\vert^{1\over 2}$$If $\vert
\omega\vert\rightarrow 0 $ we have by construction $\vert u_{\omega}\vert\rightarrow
0 $ and therefore $$ \liminf_{\vert u\vert\rightarrow 0}{\vert f( u)
\vert\over\vert u\vert} \geq {\vert f(c)
\vert\over\vert c\vert} > 0.$$ On the other hand if $X\subset L^{\infty} $ with
continuous imbedding, the condition $$ \liminf_{\Vert u\Vert\rightarrow 0}{\vert f( u)
\vert\over\vert u\vert} =  0 $$ is equivalent to the natural assumption
$\displaystyle\lim_{s\rightarrow 0}{\vert g(s)\vert\over\vert s\vert} =  0 $.
}\end{rem}
\begin{rem} {\rm The instability result in $X$ is only of interest when the
existence of at least local (and preferably global)\index{global solution} solutions for small initial data
in $X$ is fulfilled. Otherwise Theorem \ref{Theorem 5.3.1:} might just mean failure of existence in
$X$.
}\end{rem}
\begin{rem} {\rm The proof of Theorem \ref{Theorem 5.3.1:} actually implies a
stronger instability property, namely $$ \exists \varphi \hbox { eigenvector of } L
,\quad
\exists
\varepsilon _n \rightarrow 0: \sup_{t\geq 0}\Vert u_n(t) \Vert \geq \alpha >0$$ where
$u_n$ is the sequence of solutions of \eqref{Equ5.15} such that $u_n(0) = \varepsilon_n
\varphi .$ This appears much stronger since $\varepsilon_n\varphi $ tends to zero in
any reasonable norm while the norm of $X$ just needs to fulfill \eqref{Equ5.14}.
}\end{rem}
\subsection{Application to the one-dimensional heat equation}
Consider the  one - dimensional semilinear heat equation \index{heat equation}
\begin{equation}\label{Equ5.23}  u_t -
u_{xx} + f(u ) = 0 \ \hbox{ in }\  {\R^+}\times (0, L) ;  \  u(t, 0) = u(t, L) = 0 \hbox{ on }
\  {\R^+}\end{equation}
 where $f$ is  a $C^1$ function: $ \R\rightarrow\R .$ Any solution u of this problem which is global \index{global solution} and uniformly bounded on
${\R^+}\times (0, L)$ converges as $ t\rightarrow +\infty$ to a  solution $\varphi$ of
the elliptic problem
\begin{equation}\label{Equ5.24}\varphi \in H^1_0 (0, L), \ \ - \varphi_{xx} + f(\varphi) = 0. \end{equation}
\begin{prop}\label{Proposition 5.4.1} If $\varphi$ is a  solution of \eqref{Equ5.24}
which is stable \index{stable} as a
solution of \eqref{Equ5.23}, then
$\varphi$ has a constant sign on $(0, L)=: \Omega.$\end{prop}
 \begin{proof}[{\bf Proof.}] Indeed, if $\varphi$ is not identically  0  and vanishes somewhere in
$(0, L)$, the function  $w := \varphi'$ has two zeroes in $(0, L)$ and satisfies
$$w \in C^2 \cap H^1_0 (0, L), \ \ - w_{xx} + f'(\varphi)w = 0  \ \hbox{ in }
\ (0, L).$$
In particular if $ 0 < \alpha <\beta < L$  are such that $w(\alpha) = w(\beta) = 0 ,
\ w
\not= 0$ on $(\alpha,\beta )$ and if we set $\omega = (\alpha,\beta)$, we clearly
have $\lambda_1(\omega;  - \Delta + f'(\varphi)I ) = 0$ where $\lambda_1(\omega;  -
\Delta + f'(\varphi)I ) $ denotes the first eigenvalue of $ - \Delta + f'(\varphi)I$
in the sense of $H^1_0 (\omega)$. We introduce the quadratic form J and the real
number $\eta$ defined by        $$\forall z\in H_0^1(\Omega),  J(z) : =  \int_{\Omega}
\{\vert  z_x \vert ^2 + f'(\varphi)z^2 \} dx $$
$$ \eta  =  Inf\  \{ J(z), z\in
H_0^1(\Omega), \int_\Omega z^2 dx = 1\}$$ Let us also denote by  $\zeta$ the extension
of w by 0 outside $\omega$. Because $$J(\zeta) = \int_{\Omega}
\{\vert  \zeta_x \vert ^2 + f'(\varphi)\zeta^2 \} dx = \int_{\omega}
\{\vert  \zeta_x \vert ^2 + f'(\varphi)\zeta^2 \} dx = \int_{\omega}
\{\vert  z_x \vert ^2 + f'(\varphi)z^2 \} dx = 0,$$ we clearly have
$$ \eta =\lambda_1(\Omega;  -\Delta + f'(\varphi)I )  \leq 0.$$
Assuming $\eta  =  0$ means that a multiple $\lambda\zeta = \psi $ of $\zeta$ realizes
the  minimum of $J$ and therefore is a solution of $$\psi\in C^2([0, L])\cap H_0^1 (0,
L), -
\psi_{xx} + f'(\varphi)\psi = 0.$$ This is impossible since $\psi $ is not identically
0 and however vanishes throughout $(0, \alpha)$ for instance. Therefore $ \eta <  0$, and $\varphi$ is \index{unstable} unstable.
 \end{proof}

\section{Other infinite-dimensional systems} 
The following generalization of  theorems \ref{Theorem 5.1.1:} and \ref{Theorem 5.3.1:} is not difficult.

\begin{thm}\label{Theorem5.5.1:} Let
$H
$ be a real Hilbert space with inner product and norm respectively denoted by $(.,. )$
and
$\vert .\vert
$,
$L$ a (possibly unbounded)linear operator such that
$$\exists c>0, \quad L+cI\geq
0$$
 $$ R(L+(c+1)I)= H.$$ Assume in addition that we have a
decomposition $H = X
\oplus Y $   with  $dim (X) < \infty$ and $$   X\subset D(L),\, LX\subset X; \quad Y =
X^{\perp}, \, L(Y\cap D(L))\subset Y,\,   L\geq 0
\hbox { on } Y. $$
Let $ f: H\longrightarrow H$ be a locally Lipschitz map such that $f(0) = 0$ and such
that there exists a Banach space $V\subset H$ with continuous imbedding with norm
denoted by $\Vert . \Vert $ for which
$$\lim_{u\in V\setminus \{0\}, \,\Vert u \Vert \rightarrow 0} {\vert f( u)
\vert\over\vert u\vert} = 0.$$
 Then if $V$ contains all eigenvectors of
$L$, the stationary solution 0 of
$$u'+ Lu = f(u)$$
 is \index{unstable} unstable in $V$ as soon as $ L $ has at least one eigenvalue with negative real part and eigenvector in  $ X$.\end{thm}
 As a typical application of Theorem \ref{Theorem5.5.1:} we can consider the abstract second order evolution equation
\begin{equation}\label{Equ5.30} u''+ u'+ Au =f(u)\end{equation}
 where $A$ is a  self-adjoint \index{self-adjoint} operator with compact resolvant on a
real Hilbert space $H$ such that
$A+mI\geq 0
$ for some $m\geq 0$ . Introducing
$ V= D((A+(m+1)I)^{1\over 2})$ we can set $${\cal H}= V\times H,\quad D(L) = D(A)\times
V
$$ and $$\forall U = (u, v)\in D(L),\quad L(u, v) = (-v, Au+v)$$
Then \eqref{Equ5.30} takes the
form
$$ U'+ LU = F(u) = (0, f(u))$$ By considering $\{\lambda_n\}_{n\in {\bf N}}$  the
nondecreasing sequence of eigenvalues of $A$ eventually repeated according to their
multiplicity and  observing that $${\cal H}= V\times H = {\overline {\bigoplus_{n\in {\bf N}}
[(A-
\lambda_n)^{-1}(0)]^2 }} $$ it is not difficult to check the hypotheses of Theorem \ref{Theorem5.5.1:}  Hence we can state
\begin{cor}\label{Corollary 5.5.2} Under the above conditions, if $A$ has a negative
eigenvalue, and if $f, W $ are such that $$\lim_{u\in W\setminus \{0\}, \,\Vert u
\Vert_W \rightarrow 0} {\vert f( u)
\vert\over\vert u\vert} = 0$$ the solution (0, 0) is unstable \index{unstable} in the sense of ${\cal
V}:= W\times H$ as a solution of \eqref{Equ5.30}.
\end{cor} \noindent By the same method as in section \ref{SectionInstability}, we deduce easily the following
consequences of Corollary \ref{Corollary 5.5.2} :
\begin{cor}\label{Corollary 5.5.3} Let $\Omega $ be as in the introduction, $f\in C^1(\R)$ and $\varphi \in C(\overline \Omega)\cap
H^1_0(\Omega)$  a solution of the elliptic problem   $$ -\Delta
\varphi +f(\varphi) = 0  \ in \
\Omega; \
\ \varphi = 0\
\ \ on \ \ \partial \Omega$$  such that $$ \lambda_1(-\Delta +
f'(\varphi)I)< 0 $$ then $(\varphi, 0)$ is unstable \index{unstable} in $ [C(\overline
\Omega)\cap H^1_0(\Omega)]\times L^2(\Omega)$ as a  solution of the hyperbolic problem
$$ u_{tt}+u_t -
\Delta u  + f(u) = 0 \ \hbox{  in}\ \ \R^+\times\Omega; \ \ \ u = 0\ \ \hbox{ on } \ \R^+\times
\partial \Omega.$$
\end{cor}
\begin{cor}\label{Corollary 5.5.4} Consider the  one - dimensional semilinear wave \index{wave equation} equation
\begin{equation}\label{Equ5.32} u_{tt}+  u_t - u_{xx} + f(u ) = 0 \ \   in \  {\R^+}\times (0, L) ;  \  u(t, 0) = u(t, L)
= 0 \ on\  {\R^+}\end{equation}
where $f$ is  a $C^1$ function: $ \R\rightarrow\R .$ If $\varphi$ is a  solution of the elliptic problem
$$\varphi \in H^1_0 (0, L), \ \ - \varphi_{xx} + f(\varphi) = 0 $$
such that $(\varphi, 0)$ is stable \index{stable} in $ H^1_0(0, L)\times L^2(0, L)$ as a solution of \eqref{Equ5.32}, then
$\varphi$ has a constant sign on $(0, L).$
\end{cor}

\chapter[Gradient-like systems]{Gradient-like systems}\index{gradient-like}
\section{A simple general property}
Let $S(t)$ be a dynamical system \index{dynamical system} on $(Z, d)$. We denote by $\cF$ the set of equilibrium \index{equilibrium point} point of $S(t)$ i.e. 
\begin{equation}\label{equilibriumset} \cF =  \{x\in Z,\quad \forall t\ge 0, \quad S(t) x =x\}.
\end{equation}
\begin{thm}\label{Thmgradientlike}
Let $u_0\in Z$ be such that the trajectory $S(t)u_0$ has precompact range in $Z$.
The following properties are equivalent
\begin{eqnarray}
\label{gradientlike1} & &{\omega(u_0) \subset \cF },\\
\label{gradientlike2} & &{\forall h>0, \quad d(S(t+h)u_0, S(t)u_0)\longrightarrow
0 \,\,\hbox{ as }\,\,t\rightarrow +\infty },\\
\label{gradientlike3}& &{\exists \alpha>0,  \forall h\in [0,
\alpha], \quad d(S(t+h)u_0, S(t)u_0)\longrightarrow
0 \,\,\hbox{ as }\,\,t\rightarrow +\infty }.
\end{eqnarray}
\end{thm}
 \begin{proof}[{\bf Proof.}]
  i) \eqref{gradientlike3} implies
\eqref{gradientlike1}. Indeed assume
\eqref{gradientlike3} and let $x\in
\omega(u_0)$. There exists $t_n$ tending to $+ \infty$ for which
$$ \lim_{n\rightarrow \infty}S(t_n) u_0 = x.$$
 Therefore by continuity of $S(h)$ $$
\forall h>0, \quad\lim_{n\rightarrow \infty}S(t_n + h ) u_0 = \lim_{n\rightarrow \infty}S(h+ t_n ) u_0 = S(h) x .$$
As a consequence of
\eqref{gradientlike3} we have on the other hand $$\forall h\in [0,
\alpha], \quad \lim_{n\rightarrow \infty}S(t_n+h)u_0 = x.$$
By comparing the two previous formulas we find
$$\forall h\in [0,\alpha], \quad S(h)x = x.$$
Then a trivial induction argument gives
$$\forall h\in [0,\alpha], \quad \forall n\in \N, \quad S(n\alpha + h)x = x.$$
This obviously implies \eqref{gradientlike1}.\\
 ii) \eqref{gradientlike1} implies \eqref{gradientlike2}. Indeed assume that \eqref{gradientlike2} is {\it false}. Then for
some $h>0$ there is an $\varepsilon >0$ and a sequence $t_n$ tending to $+ \infty $ for
which $$\forall n\in \N, \quad d(S(t_n+h)u_0, S(t_n)u_0)\ge \varepsilon. $$ We can replace
the sequence $t_n$ by a subsequence, still denoted $t_n$, for which  $S(t_n)u_0$ converges
to a limit $x\in X$. As a consequence of \eqref{gradientlike1} we have $ x\in \cF.$ By
letting $n$ tend to infinity in the above inequality, since $S(t_n+h) = S(h) S( t_n)$ and
$S(h)$ is continuous we obtain
$$d(S(h)x, x)\ge \varepsilon. $$ This contradicts \eqref{gradientlike1}. Hence \eqref{gradientlike1} implies
\eqref{gradientlike2} and this concludes the proof.
\end{proof}
\section{A minimal differential criterion}
In this section we assume that $Z$ is a closed subset of some Banach space $X$.
\begin{cor}\label{corgradientlike}
Let $u_0\in Z$ be such that the trajectory $S(t)u_0$ has precompact range in $Z$. Assume in
addition that $$ S(t)u_0 =: u(t)  \in W^{1, 1} _{loc} (\R^+, X).$$ Then if
\begin{equation}\label{null_stepanov} {\exists \alpha>0,   \quad \int_t^{t+\alpha} \Vert
u'(s)  \Vert ds\rightarrow 0 \,\,\hbox{ as }\,\,t\rightarrow +\infty
}\end{equation} we have \eqref{gradientlike1}.
\end{cor}\begin{proof}[{\bf Proof.}] It is sufficient to observe that
$$\forall h\in [0,
\alpha], \quad d(S(t+h)u_0, S(t)u_0) =\Vert  \int_t^{t+h}
u'(s) ds \Vert \le \int_t^{t+\alpha} \Vert
u'(s)  \Vert ds\rightarrow
0 \,\,\hbox{ as }\,\,t\rightarrow +\infty.$$
 Hence \eqref{gradientlike3} is fulfilled, and by Theorem \ref{Thmgradientlike} this
implies \eqref{gradientlike1}. \end{proof}
\begin{cor}\label{corgradientLp}
Let $u_0\in X$ be such that the trajectory $S(t)u_0$ has precompact range in $Z$. Assume in
addition that $$ S(t)u_0 =: u(t)  \in W^{1, 1} _{loc} (\R^+, X).$$ Then if for some $p\ge 1$
\begin{equation}\label{Lpderivative} u' \in L^p(\R^+,X)
\end{equation} we have
\eqref{gradientlike1}.
\end{cor}\begin{proof}[{\bf Proof.}] Indeed in this case we have $$
\int_t^{t+1} \Vert
u'(s) \Vert ds \le \left(\int_t^{t+1} \Vert
u'(s) \Vert^p  ds \right)^{\frac{1}{p}}\rightarrow
0 \,\,\hbox{ as }\,\,t\rightarrow +\infty. $$
\end{proof}
\section {The case of gradient systems}\label{SystemGradient}\index{gradient system}
 Let $ N \geq 1$ and  $F \in C^{2}({\R}^N)$. We
consider the equation
\begin{equation}\label{gradientsystem} u'(t) +  \nabla F(u(t)) = 0\end{equation} and we define
$$ \cE  = \{ z\in {\mathbb R}^N, \nabla F(z) = 0\}.$$
\begin{cor}\label{gradientsyst1} Any solution $u(t)$ of \eqref{gradientsystem} defined and
bounded on ${\R}^+$ satisfies
$$\lim_{ t \rightarrow +\infty} \,\text{ dist} \{u(t), \cE  \} = 0. $$ In other terms we have $\omega
(u(0))\subset \cE. $ In addition, if for each $c$, the set $\cE_c =\{u\in \cE, F(u) = c\}$ is discrete, then there exists $u^*\in \cE$   such that
$$\lim_{ t \rightarrow +\infty} u(t) = u^*. $$
\end{cor}
 \begin{proof}[{\bf Proof.}] We consider the dynamical system \index{dynamical system} generated by \eqref{gradientsystem} on
the closure of the range of $u$. It is obvious here that the set $\cF$ of fixed points of $S(t)$ is precisely
equal to $\cE$ defined above.  Multiplying by
$u'$ in the sense of the inner product of
${\R}^N$ and integrating we find
$$\int_0^T \Vert u'(t)\Vert^2 dt = F(u(0) - F(u(t). $$ Hence since $u$ is bounded we obtain
$ u' \in L^2(\R^+, X)$ with $X = {\R}^N$. By Corollary \ref{corgradientLp}, we have $\omega
(u(0))\subset \cE. $ Moreover $F(u(t))$ is non-increasing along the trajectory since
$$ \frac {d}{dt} F(u(t))= -\Vert u'(t)\Vert^2 $$
Hence $F(u(t))$ tends to a limit $c$ as $t$ becomes infinite and therefore $\omega(u(0))\subset \cE_c. $ The rest is a consequence of Proposition \ref{prop4.2.1}.
\end{proof}

\begin{rem}\label{Curry-PDM}{\rm By using Lemma \ref{L1-conv} applied to the function $\Vert u'(t)\Vert^2 $ it is easy to prove that $u'(t)$ tends to $0$ at infinity. One might wonder whether $u(t)$ is always convergent. In 2 dimensions, it was conjectured by H.B. Curry \cite {MR0010667} and proven by  J. Palis and W. De Melo \cite {MR0669541} that convergence may fail even for a $C^{\infty} $ potential $F$. }
\end{rem}

\section {A class of second order systems}
 Let $F, \cE$ be as in Section \ref{SystemGradient}. We consider the equation
\begin{equation}\label{2gradientsystem} u''(t) + u'(t) +  \nabla F(u(t)) = 0.\end{equation} 
\begin{cor}\label{gradientsyst2} Any solution $u(t)$ of \eqref{2gradientsystem} defined and
bounded on ${\R}^+$ together with $ u'$ satisfies
$$\lim_{ t \rightarrow +\infty}\Vert u'(t)\Vert=\lim_{ t \rightarrow +\infty} \,\text{ dist} \{u(t), \cE  \} = 0 $$ 
In other terms we have $\omega
(u(0), u'(0))\subset \cE\times \{0\}. $
In addition, if for each $c$, the set $\cE_c =\{u\in \cE, F(u) = c\} $ is
discrete, then there exists $u^*\in \cE$   such that
$$\lim_{ t \rightarrow +\infty} u(t) = u^*.$$
\end{cor}
 \begin{proof}[{\bf Proof.}] We consider the dynamical system \index{dynamical system} generated by \eqref{2gradientsystem} on
the closure of the range of $U = (u, u') $. Here the set $\cF$ of fixed points of $S(t)$ is made of
points $ (y, z)\in{\mathbb R}^N\times {\mathbb R}^N$ for which the solution $u$ of
\eqref{2gradientsystem} of initial data $ (y, z)$ is independent of $t$. Consequently
$\cF=\cE\times \{0\}$.  Multiplying by
$u'$ in the sense of the inner product of
${\R}^N$ and integrating we find $$ \frac {d}{dt} (\frac{1}{2}\Vert u'(t)\Vert^2 + F(u(t)))= -\Vert
u'(t)\Vert^2 $$ hence in particular
$$\int_0^T \Vert u'(t)\Vert^2 dt = F(u(0) - F(u(t) +\frac{1}{2}(\Vert u'(0)\Vert^2 -\Vert
u'(t)\Vert^2).$$ Hence since $u$ is bounded we obtain
$ u' \in L^2(\R^+, X)$ with $X= {\R}^N$. Moreover differentiating the equation we have
$$ u'''+ u'' + \nabla^2 F(u(t))u' = 0.$$
 By multiplying by
$u''$ in the sense of the inner product of ${\R}^N$ and integrating we find
$$\int_0^T \Vert u''(t)\Vert^2 dt = \int_0^T (\nabla^2 F(u(t))u', u''(t))dt
+\frac{1}{2}(\Vert u''(0)\Vert^2 -\Vert u''(t)\Vert^2).$$
Since $u''$ is bounded by the equation, it follows immediately that $ u'' \in L^2(\R^+, X)$,
therefore $U' = (u', u'') \in L^2(\R^+, X\times X).$
By Corollary \ref{corgradientLp}, we have
$\omega (u(0), u'(0) )\subset \cE\times \{0\}.$ In particular $u'(t)$ tends to $0$ as $t$ becomes
infinite. Moreover
$\frac{1}{2}\Vert u'(t)\Vert^2+ F(u(t))$ is non-increasing along the trajectory and therefore tends
to a limit $c$ as $t$ becomes infinite. Finally
$\omega (u(0), u'(0) )\subset \cE_c\times \{0\}. $ The rest is a consequence of Proposition \ref{prop4.2.1}.
\end{proof}
\section{Application to the semi-linear heat equation\index{heat equation}}
 Let $\Omega$ and $f$ be as in Section \ref{SectionHeatEquation} and let $X=C_0(\Omega)$. Throughout this section we assume that $\Omega$ is bounded and we define
$$ \cE  =\{u\in X\cap H^1_0(\Omega), -\Delta u+f(u)=0 \},$$
$$\forall\varphi\in  X\cap H^1_0,\  E(\varphi) = {1\over 2}\int_{\Omega}\vert\nabla\varphi\vert^2 dx +\int_{\Omega}F(\varphi)\, dx$$
 with
$$ F(u) = \int_0^{u}f(s)\,ds , \quad \forall u\in {\mathbb R}. $$
Moreover let $\cE_c =\{u\in \cE,  E(u)=c\}$, for $c\in {\mathbb R}.$ We shall prove

\begin{thm}\label{Theorem 4.4.1.} let $u$ be a global \index{global solution} solution of \eqref{heat}\index{heat equation} which is bounded in $X$ for $t\geq0$.   Then  we have the following properties
 \begin{itemize}
\item[{\rm(i)}]    $E(u(t))$ tends to a limit $c$ as $t\rightarrow +\infty$;
\item[{\rm(ii)}]    $\cE_{c}\not = \emptyset$;
\item[{\rm(iii)}]  $ \text{  dist}(u(t),\cE_{c})\rightarrow 0$  as $t\rightarrow +\infty$, where   dist
denotes the distance in $X\cap H^1_0(\Omega)$.
\end{itemize}
\end{thm}
 \begin{proof}[{\bf Proof.}] The smoothing effect of the heat \index{heat equation} equation implies (cf. e.g. \cite{MR0747194}  for a proof based on the theory of holomorphic semi-groups) that for
each $\varepsilon > 0$ and $\alpha\in[0, 1)$,$$\bigcup_{t\geq
\varepsilon}\{u(t)\}\quad\hbox{is bounded in }\quad C^{1+\alpha}({\overline\Omega}).$$
In particular,  $\bigcup_{t\geq 0}\{u(t)\}$ is precompact in $X$ and
$\bigcup_{t\geq 1}\{u(t)\}$ is precompact in $H^1_0(\Omega)$. Let us denote by $Z$  the
closure in  $X\cap H^1_0(\Omega)$ of $u({\mathbb R}^+)$.  $E$ is continuous on  $X\cap H^1_0(\Omega)$,
hence on $(Z, d)$ where $d$ is the distance in $X\cap H^1_0$. In addition by precompactness the topologies of $X\cap H^1_0(\Omega)$ and $L^2(\Omega)$ coincide on $Z$. An easy
calculation shows that for $t \ge 1$, we have
$$\int_1^t\int_{\Omega}u_t^2(\tau, x)dxd\tau +E(u(t))=
E(u(1))$$ Hence   By Corollary \ref{corgradientLp}, we have
\index{$\omega$-limit set} $\omega (u(0) )\subset \cE$.   Since $E(u(t)$ is nonincreasing the result follows as in the previous examples.
 \end{proof}
 \section[Application  to a semilinear wave equation]{Application  to a semilinear wave equation with a linear damping\index{wave equation}}

\medskip\noindent Let $\Omega$ and $f$ be as in Section \ref{sectionWaveEqua} and consider the wave equation \index{wave equation} \eqref{wave}. Throughout this section we assume that $\Omega$ is bounded.
 Keeping the notation and the hypotheses of Section \ref{sectionWaveEqua}. Introducing
 $$\cE  =\{u\in H^1_0,
-\Delta u+f(u)=0 \},$$ and  $\cE_c =\{u\in \cE,  E(u, 0)=c\}$, for
$c\in {\mathbb R}$. We can state
\begin{thm}\label{Theorem 4.5.1.} Assume $\gamma>0$ and  that the growth condition \index{growth condition} \eqref{GrowthCondi} is
satisfied with the strict  inequality: $r <2/(N -2)$ if $N\geq 3$.
Let $(\varphi,\psi)\in X:=H^1_0 \times  L^2,$ and let $u$ be the corresponding
maximal solution of \eqref{wave} \index{wave equation} with $u(0) = \varphi$ and $u_t(0) =
\psi$. Assume that $T(\varphi,\psi)=+\infty$ and
$$\sup\{\Vert(u(t),u_t(t))\Vert_X , t\geq0\}<+\infty.$$
Then we have the
following properties :
\begin{itemize}
\item[{\rm(i)}]   $\, E(u(t),u_t(t))$ tends to a limit $c$ as $t\rightarrow +\infty$;
\item[{\rm(ii)}]  $\, \cE_{c}\not=\emptyset$;
\item[{\rm(iii)}]  $\, \Vert u_t(t)\Vert_{L^2} \rightarrow 0$, as $t\rightarrow +\infty$;
\item[{\rm(iv)}]   $\, dist(u(t),\cE_{c})\rightarrow 0$  as
$t\rightarrow +\infty$, where $dist$ denotes the distance in $H^1_0$.
\end{itemize}
\end{thm}
The proof of Theorem \ref{Theorem 4.5.1.} relies on a general compactness lemma due to G.F. Webb \cite{MR0549869} :
\begin{lem}\label{Lemma Webb4.5.2.}  Let $X$ be a real Banach space
and $T(t)$ a contraction semi-group on  $X$ satisfying
\begin{equation}\label{Equ4.14}
\Vert T(t)\Vert_{L(X)} \leq M e^{-\sigma t} , \quad     \forall t\geq0.
\end{equation}
 where $M, \sigma$ are some positive constants. Let
$H\in L^{+\infty}({\mathbb R}^+, X)$ and let $K$ be a compact set in $X$
such that  $H(t)\in K$ , a.e. on $ {\mathbb R}^+$. Then the function
$V:{\mathbb R}^+\rightarrow X $ defined by $$ V(t) =
T(t)(\varphi,\psi) + \int_0^t T(s)H(t-s) ds$$ satisfies: $V({\mathbb R}^+)$ is precompact in $X$.
\end{lem}
\begin{proof}[{\bf Proof.}] We have $V(t) = T(t)(\varphi,\psi) + W(t),$  where
$$ W(t) = \int_0^t T(s)H(t-s) ds.$$
Since $T(t)(\varphi,\psi)\longrightarrow 0$ in $X$ as $t\rightarrow +\infty$, there is a compact subset $K_1$ of $X$ such that $\bigcup_{t\geq 0}\{T(t)(\varphi,\psi)\}\subset K_1$. It is  therefore sufficient to prove that there is a compact subset $K_2$ of $X$ for which
$$\bigcup_{t\geq 0}\{W(t)\}\subset K_2.$$
Let $\varepsilon > 0$, and according to \eqref{Equ4.14}, let $\tau$ be such that 
$$\Vert H \Vert_{L^{+\infty}({0,\infty, X)}}\int_\tau^\infty\Vert T(s)\Vert_{L(X)}ds< \varepsilon.$$
For $t \geq \tau $, we have $$\Vert W(t)- \int_0^\tau T(s)H(t-s) ds \Vert_X < \varepsilon$$ consequently,
\begin{equation}\label{Equ4.15}\bigcup_{t\geq \tau}\{W(t)\}\subset K_3+B(0,\varepsilon)\end{equation}
with 
$$K_3=\bigcup_{t\geq \tau}\{\int_0^\tau T(s)H(t-s) ds\}. $$ Observe that the map $(s,x)\mapsto T(s)x$ is continuous: $[0,+\infty)\times X\rightarrow X.$ As a consequence, $\displaystyle U=\bigcup_{0\leq t\leq \tau}{T(t)K}$ is compact in $X$. Hence, $F=\tau\cdot {\rm conv}(U)$ is precompact in $X$. Since $K_3\subset F$, $K_3$ is precompact in $X$. By \eqref{Equ4.15}, we can cover $\displaystyle \bigcup_{t\geq \tau}\{W(t)\}$ by a finite union of balls of radius $2\varepsilon$. On the other hand, $W\in C([0,+\infty),X)$, hence $\displaystyle \bigcup_{0\leq t\leq \tau}\{W(t)\}$ is compact and can also be covered by a finite union of balls of radius $2\varepsilon$. Finally we can cover $\displaystyle \bigcup_{t\geq 0}\{W(t)\}$ by a finite union of balls of radius $2\varepsilon$. Since $\varepsilon>0$ is arbitrary, $\displaystyle \bigcup_{t\geq 0}\{W(t)\}$ is precompact, and the conclusion follows.
\end{proof}

 \begin{proof}[{\bf Proof of Theorem \ref{Theorem 4.5.1.}}] We define an unbounded operator $A^{\gamma}$ on
$X$ by $$   D(A^{\gamma}) = \{(u,v)\in X, \Delta u\in L^2 \hbox{ and } v\in H^1_0\};    $$
$$A^{\gamma}(u,v) = (v, \Delta u - \gamma v), \forall (u,v)\in D(A^{\gamma}).$$
It is easily seen that $A^{\gamma}$ is m-dissipative \index{m-dissipative} on $X$. As a consequence of Proposition \ref{Proposition2.3.3.}, the contraction semi-group $T(t)$ generated by  $A^{\gamma}$  on $X$ satisfies \eqref{Equ4.14}.\\
Now set $ U(t)=(u(t),u_t(t))$ and $H(t)=(0, -f(u(t))$, for $t \geq 0$. Clearly $u$ is a solution of \eqref{wave} \index{wave equation} on $[0,\tau]$ if, and only if $ U \in C([0,\tau];X)$ and $U$ is a solution of the equation
$$ U(t) = T(t)(\varphi,\psi) + \int_0^t T(t-s)H(s) ds,\quad \forall t\in [0, \tau].$$
Now we recall the energy identity
$$\gamma \int_0^t\int_{\Omega}u_t^2(t,x) dx dt  + E(u(t),u_t(t)) = E(\varphi,\psi) $$ with $$ E(\varphi,\psi): = {1\over 2}\int_{\Omega}\Vert \nabla \varphi(x)\Vert^2 dx + {1\over 2}\int_{\Omega}\vert  \psi(x)\vert^2 dx + \int_{\Omega}F(\varphi(x)) dx.$$ 
$E$ is continuous on $X$, hence on $(Z, d)$ where $Z$ is the closure of $u({\mathbb R}^+)$ in $X$. The energy identity shows that $E(u(t), u_t(t))$ is non-increasing. The set of stationary points of the dynamical system \index{dynamical system} is easily identified as $ \cE\times \{0\}$.
 On the other hand the function $H: {\mathbb R}^+ \rightarrow X$ defined by $H(t)=(0, -f(u(t))$ for $t \geq 0$ is such that $H({\mathbb R}^+)$ is precompact in $X$. (This comes from the strict  condition: $r <2/(N -2)$ if $ N\geq 3.$) Applying Lemma \ref{Lemma Webb4.5.2.}, we obtain compactness of bounded trajectories in $X$. Then the topologies of $X = H^1_0 \times  L^2 $ and $Y = L^2 \times H^{-1} $ coincide on $Z$ and an easy calculation using the equation now shows that
$$U' = (u_t, u_{tt})\in L^2(\R^+, Y).$$
 Indeed the energy identity gives $ u_t \in L^2(\R^+, L^2) .$ On the other hand the growth assumption on $f$ is
easily seen to imply
$$ \forall (u, v)\in X,\quad f'(u) v \in H^{-1} $$  with
$$ \Vert f'(u) v \Vert _{ H^{-1}} \leq K(1+ \Vert u\Vert_{H^1_0}^r)\Vert
v\Vert_{L^2}  .$$  By multiplying
the equation by $u_{tt}$ in the sense of $ H^{-1}$ and integrating in $t$ we find
$$ \int_0^t \Vert u_{tt}\Vert_{ H^{-1}} ^2 ds + \frac {\gamma}{2}[\Vert u_t\Vert_{ H^{-1}} ^2(0) -
\Vert u_t\Vert_{ H^{-1}} ^2(t)] + \left[\int _\Omega u u_t dx \right]_0^t$$
$$ + [\langle f(u),  u_t \rangle_{ H^{-1}} ]_0^t + \int_0^t \langle-\Delta u, u_{tt}\rangle_{ H^{-1}} ds =   \int_0^t \langle u_t, f'(u)u_t\rangle_{ H^{-1}} ds. $$ Hence, using the identity 
$$ \int_0^t \langle\Delta u, u_{tt}\rangle_{ H^{-1}} ds =  \int_0^t \Vert \nabla u_{t}\Vert ^2_{ H^{-1}} ds +  [\langle\Delta u, u_{t}\rangle_{ H^{-1}} ]_0^t  $$ $$= \int_0^t \Vert u_t\Vert_2^2 ds +  [\langle\Delta u, u_{t}\rangle_{ H^{-1}} ]_0^t $$ we derive easily
$$\int_0^t \Vert u_{tt}\Vert_{ H^{-1}} ^2 ds \le C_1 + C_2 \int_0^t \Vert u_t\Vert_2^2 ds .$$
Then the conclusion follows as in the previous example.
 \end{proof}
\begin{rem}\label{Remark 4.5.3.}{\rm Under the conditions of Proposition \ref{propEqOnde}, the
conclusions of Theorem \ref{Theorem 4.5.1.} are valid for any solution $u$ of \eqref{wave}\index{wave equation}.}
\end{rem}


\chapter[Liapunov's second method \index{Liapunov} - invariance principle]{Liapunov's second method and the  invariance principle}

\section{Liapunov's \index{Liapunov} second method}
    As explained in Section \ref{RemarkLiapunov}, the Liapunov \index{Liapunov stability} stability theorem for equation \eqref{ODEFirstOrder} can be proved by exhibiting a positive definite quadratic form which decreases exponentially  along the trajectory if the initial data are close enough to the equilibrium \index{equilibrium point} under consideration: such a function is called a {\it Liapunov function}. Sometimes it is possible to find directly such a function without calculating the fundamental matrix of the linearized equation, and this is the principle of the so-called `direct" or second Liapunov \index{Liapunov} method. This method can often be reduced to the following simple criterion 

\begin{prop}\label{prop7.2.1} Let $V \in C^1(\R^N)$ be such that  $V(u)$ tends to $+ \infty $ as $\Vert u\Vert \rightarrow + \infty $ and let $a\in\R^N$ be such that
\begin{equation}\label{VftLiapunov} \forall u \not = a, \quad \langle V'(u), f(u)\rangle < 0 .
\end{equation} Then we have  $f(a) = 0$ and in addition
\begin{itemize}
\item[-] $ \forall u \in \R^N,\ V(u)\ge V(a) $ 
\item[-] $a$ is an asymptotically stable \index{asymptotically stable} equilibrium \index{equilibrium point} point of the equation $ u' = f(u)$.
\end{itemize}
\end{prop}

 \begin{proof}[{\bf Proof.}] Since $V$ is continuous and $V(u)$ tends to $+ \infty $ as $\Vert u\Vert \rightarrow + \infty $, then there exists $b\in \R^N$ such that $V(u)\geq V(b)$ for all $u\in \R^N$. Clearly we have $V'(b)=0$ and now \eqref{VftLiapunov} imply that $b=a$.

 Once again by \eqref{VftLiapunov} $V$ is non-increasing along the trajectories, therefore all trajectories are bounded. Given such a trajectory $u$ , let $\varphi\in \omega (u(0))$ \index{$\omega$-limit set}  and let $z$ be the solution of
$$ z' = f(z)\quad z(0) = \varphi $$  Since $V(u(t)) $ tends to a limit $l$ as $t \rightarrow + \infty $, we have by \eqref{Equ4.2}
$$ \forall t, \quad V(z(t)) = l$$ 
  and then 
  $$ \forall t, \quad \langle V'(z(t)), f(z(t))\rangle= {d\over {dt}} V(z(t)) = 0. $$ In particular $ \forall t, \quad z(t) = a$, hence $ \varphi= a$ and since $z$ is constant we have $f(a) = f(z(0)) = z'(0) = 0$.  The stability of $a$ follows easily from Corollary \ref{stab-comp}. Indeed for any trajectory $u$ we have $$\frac{d}{dt} V(u(t) =  \langle V'(u(t)), f(u(t))$$  therefore either $u(t) = a$ or $\frac{d}{dt} V(u(t) < 0. $ Whenever $u(0) \not = a$ we deduce that $V(a) = \lim V(u(t)) < V(u(0))$ . \end{proof}
  
\begin{example}{\rm Let us consider the system 
 $$ u' = -u + \frac{cv}{1 + \alpha\vert v\vert}; \quad v' = -v + \frac{du}{1 + \beta\vert u\vert} $$ where $\alpha\ge 0, \, \beta\ge 0$ and $\sup\{\vert c \vert, \vert d \vert\}< 1$ which has the form \eqref{ODEFirstOrder} with $f$ Lipschitz but not differentiable at the origin except when $\alpha=\beta=0$ . Setting $$ V(u,v) = u^2 + v^2 $$ we find easily that 
 $$ \forall (u, v)  \not = (0, 0), \quad \langle V'(u, v)), f(u, v)\rangle  = -2 (u^2 + v^2) + uv ( \frac{c}{1 + \alpha\vert v\vert} +\frac{d }{1 + \beta\vert u\vert})$$ $$ \le -2 (1 - \sup\{\vert c \vert, \vert d \vert\} ) (u^2 + v^2) < 0 $$ Therefore $(0,0)$ is the unique equilibrium \index{equilibrium point} point and is globally asymptotically \index{asymptotically stable} (here exponentially) stable. In the special case $c= d>0$ and $\alpha=\beta=1$ , assuming $u_0\ge 0, v_0\ge 0$ the solutions of the above system with initial data $(u_0 , v_0)$ remain non-negative for all times and coincide with the solutions of 
 $$ u' = -u + \frac{cv}{1 + v}; \quad v' = -v + \frac{cu}{1 +  u} $$ which is known as the Naka-Rushton model for neuron dynamics in the short term memory framework. 
}\end{example}

\section{Asymptotic stability obtained by Liapunov \index{Liapunov} functions}

 Consider the nonlinear
 wave \index{wave equation} equation
\begin{equation}\label{Equ3.21} u_{tt} - \Delta u +g( u_t) = 0 \quad\hbox{  in} \, \,{\mathbb R}^+ \times
\Omega,
\quad     u = 0  \   \hbox{on} \, \, {\mathbb R}^+ \times \partial\Omega \end{equation}
where $\Omega$ is a bounded domain of $\R^N$ and $ g$ satisfies the conditions
\begin{equation}\label{Equ3.22} \exists \alpha >0, \    g(v)v \geq  \alpha\vert v\vert ^2, \ \forall v \in{\mathbb R}
 \end{equation}
  \begin{equation}\label{Equ3.23} \exists C \geq 0, \  \vert g(v)\vert  \leq C(\vert v\vert  + \vert v\vert ^{\gamma}),\  \forall v \in{\mathbb R}
 \end{equation}
with :
\begin{equation}\label{Equ3.24}
\gamma > 1 \quad\hbox{and if}\quad   N\geq3, \quad\gamma  \leq (N+2)/(N-2).
\end{equation}

For the sake of simplicity we consider classical solutions of \eqref{Equ3.21} for which
differentiations are plainly justified. We obtain the following result of global
asymptotic stability :
\begin{thm}\label{thm3.3.3} Let $$u\in L^\infty_{loc}({\mathbb R}^+, H^2\cap H^1_0(\Omega
))\cap W^{1,\infty}_{loc} ({\mathbb R}^+,  H^1_0(\Omega
))\cap W^{2,\infty}_{loc} ({\mathbb R}^+,  L^2(\Omega
))$$  be a solution of \eqref{Equ3.21}. Then
we have
\begin{equation}\label{Equ3.25} \int_\Omega\{\vert \nabla u \vert ^2+ u_t^2\}(t, x) dx \leq M \Bigl(\int_\Omega\vert \nabla u(0,
x) \vert ^2 dx, \int_\Omega\vert  u_t(0,
x) \vert ^2 dx \Big) e^{-\delta t}
\end{equation}
 where  $\delta>0$ depends only on $\alpha,
C$ and $\gamma$ and  $M$ is bounded on bounded sets.
\end{thm}
 \begin{proof}[{\bf Proof.}]  We denote by $(\cdot,\cdot)$ the inner product in $L^2(\Omega)$, by $\vert\cdot\vert$ the corresponding norm  and by  $\Vert\cdot\Vert$ the norm in $H^1_0(\Omega)$.
In addition the duality pairing on $H^{-1}(\Omega)\times H^1_0(\Omega)$ will be denoted by $\langle\cdot,\cdot\rangle$.
Now we define
$$ \Phi_{\varepsilon}(t) = \Vert u(t)\Vert^2 + \vert u_t(t)\vert ^2 + \varepsilon(u(t), u_t(t))    $$
where $\varepsilon> 0$ is at our disposal. For $\varepsilon$ small enough, $ \Phi_{\varepsilon}$ is comparable to the usual energy \index{energy} and we obtain :
$$\frac{d}{dt}\{ \Vert u(t)\Vert^2 + \vert u_t(t)\vert ^2\} = \langle u_{tt} + Lu , u_t \rangle = - 2\int_\Omega g(u_t)u_t dx$$
$$\frac{d}{dt}(u(t), u_t(t)) = \vert u_t(t)\vert ^2 + \langle u_{tt}(t) , u(t)\rangle = \vert u_{t}(t)\vert ^2 - \Vert u(t)\Vert^2  -
 \int_\Omega g(u')u\,dx.$$
Therefore :
\begin{equation}\label{Equ3.26}
\frac{d\Phi_{\varepsilon}}{ dt} = - 2 \int_\Omega g(u_t)u_t dx
+\varepsilon\vert u_t(t)\vert ^2 -\varepsilon\Vert u(t)\Vert^2  - \varepsilon
\int_\Omega g(u_t)u dx.
\end{equation}
It follows from \eqref{Equ3.22} and  \eqref{Equ3.23} that
$$ \vert g(v)\vert  \leq 2 C \vert  v \vert        \quad\hbox{for }\quad \vert v\vert  \leq 1$$
$$ \vert g(v)\vert ^{\gamma+1} \leq 2 C (v g(v))^{\gamma}   \quad\hbox{for }\quad \vert v\vert  > 1.$$
In particular for each $v\in L^{\gamma+1}(\Omega)$ we have by setting  $\beta
=\frac{\gamma+1}{\gamma}$ and denoting as  $\Vert \cdot \Vert_\beta $ the norm in $L^{\beta}(\Omega)$
$$\Vert g(v)\Vert_\beta \leq 2C \Vert v\Vert_\beta + (2 C)^{\frac{1}{\gamma+1}} \left(\int_\Omega v g(v)
dx\right)^{\frac1\beta} \leq C_1 \Vert v\Vert_\beta + C_2\Bigl( \int_\Omega v g(v) dx\Bigr)^\frac{1}{\beta}.$$
Since the condition $\gamma  \leq (N+2)/(N-2) $  yields $ \beta=\frac{\gamma+1}{\gamma} \geq\frac{
2N}{N+2} = (2^*)'$, \eqref{Equ3.26} implies
$$\frac{d\Phi_{\varepsilon}}{ dt} = (- \alpha +\varepsilon) \vert u_t(t)\vert ^2 -
\varepsilon\Vert u(t)\Vert^2 + K_1\varepsilon \Vert u(t)\Vert \vert u_t(t)\vert $$
\begin{equation}\label{Equ3.27} - \int_\Omega g(u_t)u_t dx   + C_2 \varepsilon \Bigl(\int_\Omega u_t g(u_t)
dx)\Bigr)^{\frac1\beta} \Vert u(t)\Vert.
\end{equation}
By reordering the terms and using Young's
inequality with exponents $\gamma+1$ and $\beta$ we deduce from \eqref{Equ3.27} :
$$\frac{d\Phi_{\varepsilon}}{ dt}\leq (- \frac{\alpha}{2} +\varepsilon) \vert u_t(t)\vert ^2  +(K\varepsilon ^2-
\varepsilon)\Vert u(t)\Vert^2 + (C_2\varepsilon)^{\gamma+1} \Vert u(t)\Vert^{\gamma+1}.$$ Since
$2 E(t) = \Vert u(t)\Vert^2 + \vert u_t(t)\vert ^2$  is a nonincreasing function ot $t \geq 0,$ we can first
choose $\varepsilon>0$ small, depending on $E(0)$, such that
\begin{equation}\label{Equ3.28} \forall t \geq 0, \quad  \frac{d\Phi_{\varepsilon}}{ dt} \leq  - \frac{\varepsilon}{2} \{\Vert u(t)\Vert^2 + \vert u_t(t)\vert ^2\}.
\end{equation}
This shows that $E(t)\rightarrow 0$ exponentially, uniformly on
bounded subsets of $  H^1_0(\Omega)\times L^2(\Omega)$. Then for each initial
condition , we can  find $T_0>0,$ depending on $E(0)$, such that $E(t) \leq 1$ whenever
$t \geq T_0.$ Now for $t \geq T_0$, we have \eqref{Equ3.28} with $\varepsilon$ independent of $E(0)$.
Hence \eqref{Equ3.25} follows with $\delta$ independent of $E(0)$.\bigskip
 \end{proof}
 In section \ref {Linear instability and Bellman's approach}, we saw that even in the nonlinear case , the existence of an eigenvalue $s$ of $ Df(a)$ with $Re(s)> 0$ implies the instability of $a$. On the other hand, in the marginal case $Re(s)=  0$ (for instance when $s= 0$) , a can still be \index{asymptotically stable} asymptotically stable, as shown by the next examples.

1) A typical example of such a situation is the first order ODE
\begin{equation}\label{Equ3.29}
 u' = - \vert u\vert ^{p-1} u  ,  \quad      t \geq 0
\end{equation}
with $p>1.$ The solutions $u ÃÂÃÂÃÂÃÂ­ 0$ of \eqref{Equ3.29} are given by the formula
\begin{equation}\label{Equ3.30}
u(t) ={sgn(u_0)\over{\{(p-1)t + \vert u_0 \vert ^{1-p}}\}^{1\over{p-1}}}.
\end{equation}

It is clear from \eqref{Equ3.30} that $$ \vert u(t) \vert \sim
\Bigr\{{1\over{(p-1)t}}\Bigl\}^{1\over{(p-1)}}$$ as $t\rightarrow +\infty $ for every
$u_0 \not= 0.$ Analogous, but somewhat artificial parabolic example can be given. Let us
consider now some second order examples.\bigskip

2) First we consider the equation ( with  $c>0, p>1.$)
\begin{equation}\label{Equ3.31}
u'' + u + c \vert u'\vert ^{p-1} u' = 0  ,  \quad       t \geq 0.
\end{equation}
We set
$\varphi(t) = (u^2+ u'^2)(t)$: then
$$\varphi'(t) = -2c \vert u'\vert ^{p+1} \geq -2c (u^2+ u'^2)^{ (p+1)/2}= -2c\varphi(t)^{
(p+1)/2}$$and as in the previous example we deduce
$$ \varphi(t)\geq
\Biggr\{{1\over{[\varphi(0)]^{{1-p}\over2}+c(p-1)t}}\Biggl\}^{2\over{p-1}}.$$
Hence the energy \index{energy} tends to  $0$ at most like $t^{ -2/(p-1)}$  as  $t\rightarrow +\infty.$ In fact we have
\begin{prop}\label{prop3.4.1} For each solution u of \eqref{Equ3.31} we have
\begin{equation}\label{Equ3.33}\forall
t>0,\quad \{u^2(t) +u'^2(t)\}^{1\over 2}\leq C(u(0), u'(0))
t^{-{1\over{p-1}}}.\end{equation}
\end{prop}
 \begin{proof}[{\bf Proof.}]   We set:
    $$  \Phi_{\varepsilon}(t) = u^2(t) + u'^2(t) + \varepsilon\vert u(t)\vert ^{p-1} u(t)
u'(t)$$ Then:   $$\Phi'_{\varepsilon} = -2c \vert u'\vert ^{p+1} + \varepsilon\vert u\vert ^{p-1}
(pu'^2+uu'') = -2c \vert u'\vert ^{p+1} + \varepsilon[p\vert u\vert ^{p-1}u'^2- \vert u\vert ^{p+1}$$
    $$ -c \vert u'\vert ^{p-1}u'\vert u\vert ^{p-1}u)    \leq     -2c \vert u'\vert ^{p+1} + \varepsilon\{-(1/2)\vert u\vert ^{p+1}
+ C \vert u'\vert ^{p+1}\},$$
where $ C $ depends only on $u(0), u'(0).$ For $\varepsilon >0$ small enough, we
therefore obtain
\begin{equation}\label{Equ3.34}  \Phi'_{\varepsilon} \leq  - (\varepsilon /2) \{ \vert u\vert ^{p+1} + \vert u'\vert ^{p+1}\} \leq - \delta
(\Phi_{\varepsilon})^{(p+1)/2}.
\end{equation}
Clearly, \eqref{Equ3.34} implies \eqref{Equ3.33} for $\varepsilon$ small enough.
 \end{proof}

3) Finally , by using the method of proof of Theorem \ref{thm3.3.3}, one can
prove\medskip

\begin{thm}\label{thm3.4.2} Assume that $g\in C^1(R)$ satisfies the conditions
$$ \exists\alpha >0,\     g(v)v \geq
\alpha\vert v\vert ^{p+1},\  \forall v \in{\mathbb R},$$
$$ \exists C \geq 0,  \vert g(v)\vert  \leq C(\vert v\vert  + \vert v\vert ^{\gamma}),  \forall v \in{\mathbb R},$$
with : $1<p\leq \gamma, \gamma$ satisfying \eqref{Equ3.24}. Then for each solution
$$u\in L^\infty_{loc}({\mathbb R}^+, H^2\cap H^1_0(\Omega ))\cap W^{1,\infty}_{loc} ({\mathbb R}^+,
H^1_0(\Omega ))\cap W^{2,\infty}_{loc} ({\mathbb R}^+,  L^2(\Omega))$$
 of \eqref{Equ3.21} we have
\begin{equation}\label{Equ3.37}  \int_\Omega\{\vert \nabla u \vert ^2+ u_t^2\}(t, x) dx \leq M \Bigl(\int_\Omega\vert \nabla u(0,
x) \vert ^2 dx, \int_\Omega\vert  u_t(0,
x) \vert ^2 dx \Big) t^{-1\over{p-1}}
\end{equation}
\end{thm}
 \noindent{\bf
Idea of the proof.} Let $$   \Phi_{\varepsilon}(t) = \Vert u(t)\Vert^2 + \vert u'(t)\vert ^2  +
\varepsilon \{ \Vert u(t)\Vert^2 + \vert u'(t)\vert ^2 \}^{{p-1}\over2} (u(t), u'(t))$$   By adapting
the proof of Theorem \ref{thm3.3.3} and Proposition \ref{prop3.4.1}, we establish$$\Phi'_{\varepsilon} \leq
 - \delta(\Phi_{\varepsilon})^{(p+1)/2},$$ valid for $\varepsilon>0$ small enough and
some $\delta>0 $ depending on the initial \index{energy} energy.

\begin{rem}{\rm It is not known whether \eqref{Equ3.37} is optimal when for instance
  $$g(v) = c\vert v\vert ^{p-1} v   ,\ c>0,\ p>1.$$ A very partial result in this direction (lower estimate comparable to the  upper decay estimate raised to the power $\frac{3}{2}$) can be found in  \cite{Ha94-09} in the case $N= 1$, relying on an argument specific to dimension 1.  }
\end{rem}

 \section[The Barbashin-Krasovski-LaSalle criterion]{The Barbashin-Krasovski-LaSalle criterion for asymptotic stability} 
 After Liapunov, the stability theory has been pursued mainly by the russian school which was also involved in control theory of ODE under the impulsion of major russian experts such as L. S. Pontryaguin. In this context, interesting contacts have been established between the russian school and american mathematicians such as J.K. Hale and J.P. LaSalle. The exchanges between J.P. LaSalle, E.A. Barbashin and N.N. Krasovskii led to the now well-known invariance principle, and LaSalle in his papers is quite clear about the influence of the russian school on his own research. To illustrate the progression of ideas, we start with a simple and convenient result about asymptotic stability.
 
 \begin{thm}\label{thm7.3.1} Let $f\in C^1(\R^N)$ and consider the differential system \eqref{ODEFirstOrder}.  Let $a\in \R^N$ be such that $f(a) = 0$ and $U$ be a bounded open set with $a\in U$ such that
\begin{itemize}
\item[(i)] For any $x$ close enough to $a$, the solution $u$ of \eqref{ODEFirstOrder} with $u(0) = x$ is global \index{global solution} and remains in $U$.
\item[(ii)] $\exists V \in C^1(\R^N)$  such that   $$ \forall u \in U, \quad \langle V'(u), f(u)\rangle \le  0 .$$  
\item[(iii)]  The set ${u \in \overline{U}, \quad \langle V'(u), f(u)\rangle =  0}$ contains the range of no trajectory of \eqref{ODEFirstOrder} except the constant trajectory $a$.
 \end{itemize}
Then $a$ is a strict local minimum of $V$, it is the only equilibrium \index{equilibrium point} point in $ \overline{U}$  and $a$ is an  asymptotically stable \index{asymptotically stable} equilibrium \index{equilibrium point} point of \eqref{ODEFirstOrder}.
\end{thm}

 \begin{proof}[{\bf Proof.}]  Given  a trajectory $u$ of \eqref{ODEFirstOrder} with $u(0)$ close enough to $a$ so that $u$ remains in $U$, let $\varphi\in \omega (u)$ \index{$\omega$-limit set}  and let $z$ be the solution of
  $$ z' = f(z)\quad z(0) = \varphi $$  Since $V(u(t)) $ tends to a limit $l$ as $t \rightarrow + \infty $, we have 
  $$ \forall t\ge 0, \quad V(z(t)) = l$$ In addition $  \forall t\ge 0, \quad z(t) \in \overline{U}$ and  $\quad \langle V'(z(t)), f(z(t))\rangle= {d\over {dt}} V(z(t)) = 0. $ \\
  
  \noindent In particular as a consequence of (iii) we have $\forall t\ge 0, \quad z(t) = a, $ hence $ \varphi= a$.  So $u(t)$ converges to $a$ as $t\rightarrow \infty .$ Moreover if $u(0)\not = a, $ by (iii) there is some $T\in \R^+$ for which $ \langle V'(u(T)), f(u(T)) <0$ and then $ V(u(0)> V(a).$ Therefore  $a$ is a strict local minimum of $V$ and the conclusion now follows from Corollary \ref{stab-comp}.  \end{proof}
  
\begin{example}{\rm Let us consider the system 
 $$ u' = v; \quad v' = -u - g(v) + c $$ where $c\in \R$ and $g$ is increasing with $g(0) = 0$. Setting $$ V(u,v) = (u-c)^2 + v^2 $$ we find easily that 
 $ \forall (u, v) \in \R^2, \quad \langle V'(u, v)), f(u, v)\rangle  = -2 g(v)v \le  0 .$  Taking for $U$ any ball centered at $(c, 0)$ , conditions i) and ii) are obviously fulfilled. Then if a trajectory $(u,v)$ satisfies  $\langle V'(u, v)), f(u, v)\rangle  =0$ , from $  -2 g(v)v\equiv 0 $ we deduce $v\equiv 0$, hence $ v'\equiv 0$ and by the second equation $u\equiv c$. Finally $(c, 0)$ is the only equilibrium \index{equilibrium point} and is globally asymptotically stable \index{asymptotically stable} as a consequence of Theorem \ref{thm7.3.1}.}
\end{example}

 \begin{example}{\rm  Let us consider the system 
 $$ u' = v; \quad v' = J^{-1}( -p\sin u - kv + c) $$ where $c>0$ and $J, p, k$ are positive with $c<p$.  This represents the motion of a robot arm with one degree of freedom submitted to a constant torque. Setting $$ V(u,v) = {J\over 2} v^2 + p(1- \cos u)-cu$$ we find easily that 
 $$ \forall (u, v) \in \R^2, \quad \langle V'(u, v)), f(u, v)\rangle  = -k v^2 \le  0 $$ We claim that  the hypotheses of Theorem \ref {thm7.3.1} are satisfied when 
 $\alpha = \arg\sin {c\over p} $ and $a = ( \alpha, 0). $  Indeed from the equation above it follows that the function $V(u(t),v(t))$ is constant if and only if $(u(t), v(t))= (\beta, 0 )$ and $p\sin\beta =  c $.  Moreover,  setting $F(u) = p(1- \cos u)-cu $ we see immediately that 
 $$ F'(\alpha) = p\sin \alpha -c = 0, \quad   F''(\alpha) = p\cos \alpha> 0$$
 Therefore $a = ( \alpha, 0)$ , is a strict minimum of $V$, and is consequently a stable \index{stable} equilibrium \index{equilibrium point} by Corollary \ref{stab-comp}. Since $a$ is an isolated solution of this equation, the only possibility is $\beta = \alpha$.  By Theorem \ref{thm7.3.1} we conclude that $a$ is \index{asymptotically stable} asymptotically stable. The same property holds true for the other equilibria of the form $(\alpha+ 2k\pi, 0)$.}\end{example}

 \section{The general Lasalle's invariance principle}Let $(Z, d)$ be a complete metric space and $\{S(t)\}_{t\geq0} $ a
dynamical system \index{dynamical system} on $Z$.
\begin{defn}\label{Definition4.2.1.} A function $\Phi\in C(Z,{\mathbb R})$ is called a
{\it Liapunov  function} \index{Liapunov function}  for $\{S(t)\}_{t\geq0} $ if we have
\begin{equation}\label{Equ4.9} \Phi(S(t)z) \leq \Phi(z),\quad \forall z\in Z,\, \forall t\geq0. \end{equation}
\end{defn}
    \begin{rem}\label{Remark 4.2.2.}{\rm By using the semi-group property of $S(t)$, it is immediate to see that   $\Phi$ is a Liapunov \index{Liapunov function} function for
$\{S(t)\}_{t\geq0} $ if, and only if for each $z\in Z$ the function
$t\mapsto\Phi(S(t)z)$ is nonincreasing.}
\end{rem}
The following result is known as {\it LaSalle's invariance principle}.
\begin{thm}\label{Theorem 4.2.3.} (cf. \cite{{MR0127473}}) Let  $\Phi$ be a Liapunov \index{Liapunov function} function for
$\{S(t)\}_{t\geq0} $, and let $z\in Z$ be such that
${\displaystyle \bigcup_{t\geq0}\{ S(t)z}\}$ is precompact in  Z.
Then
\begin{itemize}
\item[(i)]$\displaystyle c= \lim_{t\rightarrow +\infty} \Phi(S(t)z) $ exists.
\item[(ii)]$ \Phi(y)=  c, \ \forall y\in \omega(z). $
\end{itemize}
In particular :   $$   \forall y\in \omega(z),\forall t \geq0,\quad  \Phi(S(t)y) = \Phi(y).$$
\end{thm}
 \begin{proof}[{\bf Proof.}] (i) $\Phi(S(t)z)$ is
nonincreasing and bounded since ${\displaystyle \bigcup_{t\geq0}\{
S(t)z}\}$ is precompact. This implies the existence of the limit
c.\\
    (ii) If $y\in \omega(z)$, there exists a sequence $t_n\rightarrow +\infty$ such that
$S(t_n)z\rightarrow y$. Hence $\Phi(S(t_n)z)\rightarrow \Phi(y)$  and this implies $
\Phi(y) = c.$

The last property is now an immediate consequence of the invariance of $\omega(z)$
(theorem \ref{thm 4.1.8.}, i)).
 \end{proof}

\begin{rem}\label{Remark 4.2.4.} {\rm Practically, Theorem \ref{Theorem 4.2.3.} is used to
show the convergence of some trajectories of $\{S(t)\}_{t\geq0} $
to an \index{equilibrium point} equilibrium. Therefore the following definition and theorem
are basic.}
\end{rem}
\begin{defn}\label{Definition 4.2.5.} A Liapunov  function \index{Liapunov function} $\Phi$ for $\{S(t)\}_{t\geq0} $
is called a  {\it strict Liapunov function } \index{strict Liapunov function} if the following
condition is fulfilled : Any $z\in Z$ such that
$\Phi(S(t)z)=\Phi(z)$ for all $t \geq0$ is an \index{equilibrium point} equilibrium of
$\{S(t)\}_{t\geq0} $.
\end{defn}
\begin{thm}\label{Theorem 4.2.6.} Let $\Phi$ be a strict Liapunov  function
for $\{S(t)\}_{t\geq0} $, and let $z\in Z$ be such that
${\displaystyle \bigcup_{t\geq0}\{ S(t)z}\}$ is precompact in  Z.
Let ${\cal E}$   be the set of equilibria of $\{S(t)\}_{t\geq0} $.
Then\begin{itemize}
\item[(i)] ${\cal E}$   is a closed, nonempty subset of $Z$;
\item[(ii)] $d( S(t)z,{\cal E} )\rightarrow 0$  as  $t\rightarrow +\infty$, i.e. $\omega(z)\subset \cal E$.
\end{itemize}
\end{thm}
 \begin{proof}[{\bf Proof.}] By continuity of  $S(t), {\cal E}$   is closed. By Theorem \ref{thm 4.1.8.} (i),
 $\omega(z)\not=\emptyset.$ Now let $y \in\omega(z).$ The last  assertion of
Theorem \ref{Theorem 4.2.3.}  gives
$$  \Phi(S(t)y) = \Phi(y) , \quad \forall t \geq0$$ and therefore, since $\Phi$ is a strict Liapunov \index{strict Liapunov function} function,  $y$ is an \index{equilibrium point} equilibrium : in particular we have (i) and  $\omega(z)\subset \cal E$. Then Theorem \ref{thm 4.1.8.}  (iii) implies (ii).
 \end{proof}
\begin{rem}\label{Remark 4.2.8.} {\rm Theorem \ref{Theorem 4.2.6.} means that the set of equilibria
attracts all trajectories of $\{S(t)\}_{t\geq0} $.}
\end{rem}
\begin{cor}\label{Corollary  4.2.9.} Under the hypotheses of Theorem \ref{Theorem 4.2.6.},
let $$  c = \lim _{t\rightarrow +\infty} \Phi(S(t)z)\quad\hbox{
and } \quad  {\cal E}_c =\{x\in{\cal E}, \Phi(x) = c\}.$$
Then $ {\cal E}_c$  is a closed, nonempty subset of $ Z $ and
$d(S(t)z,{\cal E}_c )\rightarrow 0$ as $t\rightarrow+\infty$. If
in addition $ {\cal E}_c$ est discrete, there exists $y\in{\cal
E}_c$  such that $S(t)z \rightarrow y$  as $t\rightarrow+\infty$.
\end{cor}
 \begin{proof}[{\bf Proof.}] Since $ {\cal E}$ is closed and $\Phi$ is continuous, $ {\cal
E}_c$ is closed. The rest of the corollary is a consequence of Theorems \ref{Theorem 4.2.3.}, \ref{Theorem 4.2.6.} and \ref{thm 4.1.8.} (ii).
 \end{proof}

\section{Application to some differential systems in
$\R^N$}

Theorem \ref{Theorem 4.2.3.}, Theorem \ref{Theorem 4.2.6.} and Corollary \ref{Corollary  4.2.9.} allow to recover easily the results of chapter 6 on gradient systems \index{gradient system} and second-order gradient-like \index{gradient-like} systems with linear dissipation. But they show  their full power in more complicated situations in which calculations implying convergence to $0$ of the time-derivative become less natural. As a typical example we can consider the  equation
\begin{equation}\label{2Bisgradientsystem} u''(t) + g(u'(t)) +  \nabla F(u(t)) = 0.\end{equation}
where $F\in C^1(\R^N,\R)$  and $g:\R^N\longrightarrow\R^N$ is a  continuous function such that $$\forall v\in \R^N\setminus\{0\},\quad\langle g(v), v\rangle>0.$$
\begin{cor}\label{gradientsyst2Bis} Any solution $u(t)$ of \eqref{2Bisgradientsystem} defined and
bounded on ${\R}^+$ together with $ u'$ satisfies
$$\lim_{ t \rightarrow +\infty}\Vert u'(t)\Vert=\lim_{ t \rightarrow +\infty} \,\text{ dist} \{u(t), \cE  \} = 0 $$
with
$$ \cE  = \{ z\in {\mathbb R}^N, \nabla F(z) = 0\}.$$
If in addition for each $c$, the set $\cE_c =\{u\in \cE, F(u) = c\} $ is
discrete, then there exists $u^*\in \cE$   such that
$$\lim_{ t \rightarrow +\infty} u(t) = u^*.$$
\end{cor}
 \begin{proof}[{\bf Proof.}] We consider the dynamical system \index{dynamical system} generated by (3) on
the closure of the range of $U = (u, u') $. Here the set $\cF$ of fixed points of $S(t)$ is made of
points $ (y, z)\in{\mathbb R}^N\times {\mathbb R}^N$ for which the solution $u$ of
\eqref{2Bisgradientsystem} of initial data $ (y, z)$ is independent of $t$. Consequently
$\cF=\cE\times \{0\}$.  Multiplying by
$u'$ in the sense of the inner product of
${\R}^N$ and integrating we find $$ \frac {d}{dt} (\frac{1}{2}\Vert u'(t)\Vert^2 + F(u(t)))= -\langle g(u'), u'\rangle \le 0 $$ hence $$\Phi (u, v) : =  \frac{1}{2}\Vert v\Vert^2 + F(u) $$ is a \index{Liapunov function} Liapunov function. On the other hand if $\Phi $ is constant on a trajectory $(u(t), u'(t)) $ we have $u' \equiv 0$. Hence $\Phi $ is a strict Liapunov function \index{strict Liapunov function} and the result follows. 
\end{proof} 

As an example of application of Corollary \ref {gradientsyst2Bis}, the equation $$ u'' + u' + u^3 - u  =  0 $$ already considered in Section 4.2 provides a good illustration. Here the set of equilibria has only  points solutions : (-1, 0), (0, 0) and (1, 0). Note that here and more generally under the hypotheses of Corollary \ref {gradientsyst2Bis}, the t-derivative of the Liapunov function vanishes at some point $t_0$ only if $u'(t_0)= 0.$ Then it follows easily that energy \index{energy} conserving trajectories are made of equilibria.  In the next example the condition $u'= 0$ does not follow immediately, but as a consequence of the connnectedness of trajectories:
\begin{example}{\rm
Let us consider the scalar equation \begin{equation} \label{ODEuu'} 
u''+au^2 u'+ u^3-u=0 
\end{equation}
where $a>0$.
Let
$$E(t)= \frac12u'^2+\frac14 u^4 -\frac12  u^2.$$
$$\frac{d}{dt} E(t)=-au^2 u'^2$$
Since  $E$ is non-increasing,  $(u,u')$ are bounded and we are in a good position to apply the invariance principle. Indeed let $u$ be a solution of \eqref{ODEuu'} for which $E$ is constant, then $uu' = 0$  hence $u^2$ is constant and then, by connectedness, $u$ is constant. So $u' = u''= 0$. As in the previous example,the stationary equation $u^3-u = 0 $ has only three solutions : -1, 0, 1. So that we have convergence of all solutions, although the t-derivative of the Liapunov function \index{Liapunov function} vanishes also at points $t_0$ for which $u(t_0)= 0$ and the equation is not a special case of Corollary \ref {gradientsyst2Bis} }\end{example}

The next example shows that sometimes, the invariance principle provides some information which is not so easy to recover by more elementary methods. 
 
\begin{example}{\rm Let us consider the coupled system of second order scalar ODE:
\begin{equation} \label{ODEcouple}
\left\{ \begin{array}{ll}
u''+u'+\lambda u+cv=0, & \\[2mm]
v''+\qquad \lambda v + cu=0, &  
\end{array} \right.
\end{equation}
$\lambda>0$ and $c\not=0$ with $c^2 <\lambda^2$.
Let
$$E(t) = \cE(u,u',v,v') = \frac12\left[   u'^2+   v'^2+\lambda(u^2+v^2)\right]+ c uv .$$
$$\frac{d}{dt} E(t)=- u'^2$$
Since  $E$ is non-increasing,  $(u,v,u',v')$ are bounded and we are in a good position to apply the invariance principle. Indeed let $(u,v)$ be a solution of \eqref{ODEcouple} for which $E$ is constant. Then $2u'^2 = 0$ implies $u'= 0$, hence $u$ is constant and $u''= 0$. Then by the first equation $v = -\frac{\lambda}{c}u$ is also constant. Finally since by the hypothesis  $c^2 <\lambda ^2 $ , the stationary system $ \lambda u  + cv =   cu+\lambda v = 0 $ has no non-trivial solution, we conclude that  $u= v= 0$  and therefore $(0,0,0,0)$  is \index{asymptotically stable} asymptotically stable. Because the system is linear and finite-dimensional, by taking a basis in $\R^4$ it follows immediately that the norm of the fundamental matrix tends to $0$ as $t$ tends to infinity, and using the semi-group property it follows that convergence is exponential. The general theory designed by Liapunov \index{Liapunov} in his seminal paper (1892) shows the existence of a quadratic form $\Phi$ on $\R^4$  satisfying the identity
$$ \frac{d}{dt} \Phi(Y(t)) = - |Y(t)| ^2 \le -\eta \Phi(Y(t)) $$ (with $\eta>0$ )for any solution $Y = (u, v, u', v')$ of \eqref{ODEcouple} which means that the older method  of quadratic energies must allow to recover directly that $(0,0,0,0)$  is \index{asymptotically stable} asymptotically stable, with quantitative information about the decay rate. Since the form can be computed on a basis of $\frac{4\times (4+1)}{2} = 10 $ monomials in $(u,v,u',v')$, the challenge is now to find one of the strict quadratic Liapunov functions \index{strict Liapunov function} (they form a non-empty open set in the space of coefficients) by a direct method. It turns out that for any $p>1$ and for all $\varepsilon>0$ small enough the quadratic form 
  \begin{equation}\label{FtLiapunovSystDiff}
H=\cE-\varepsilon vv'+ p\varepsilon uu'+\frac{(p+1)\lambda\varepsilon}{2c}(u'v-uv')
\end{equation} is a strict Liapunov \index{strict Liapunov function} function for  our system. The calculations are not immediate, especially if we do not know in advance the formula! Here, LaSalle's invariance principle was very useful since,without the information of asymptotic stability obtained by a very simple sequence of calculations, it would have been very difficult to imagine that such a function can be devised.}\end{example}

\section{Two infinite dimensional examples }

\begin{example}{\rm Let us consider the coupled system of second order scalar ODE:
\begin{equation} \label{DEcouple}
\left\{ \begin{array}{ll}
u''+u'+Au +cv=0, & \\[2mm]
v''+\qquad Av + cu =0, &  
\end{array} \right.
\end{equation} where  $A$ is a possibly unbounded linear operator on a Hilbert space $H$ with norm denoted by $|.|$ such that for some $\lambda>0$,  
$$A= A^*\ge \lambda I $$ and $c\not=0$ with $c^2 <\lambda^2$. In addition we assume that the unit ball of $ D(A^{1/2})$ is compact in $H$. 
Let
$$E(t) = \cE(u,u',v,v') = \frac12\left[   u'^2+   v'^2+ |A^{1/2} u|^2+|A^{1/2} v|^2 \right]+ c (u, v) .$$

We have the formal energy identity \index{energy} :   $\frac{d}{dt} E(t)=- |u'|^2. $ Since  $E$ is non-increasing,  the vector $(u,v,u',v')$ is bounded in $ D(A^{1/2})\times D(A^{1/2})\times H \times H.$ Actually it is not difficult to check that if $(u(0),v(0),u'(0),v'(0)) \in D(A)\times D(A)\times D(A^{1/2})\times D(A^{1/2}) = W $, then the vector $(u,v,u',v')$ remains and is  bounded in $W$  for $t\ge0$ and the energy\index{energy} identity is rigorously satisfied. Then the trajectory is precompact and if  $(u,v)$ be a solution of \eqref{DEcouple} for which $E$ is constant. Then  $u'= 0$, hence $u$ is constant and $u''= 0$. Then by the first equation $v = -\frac{Au}{c}$ is also constant. Finally since by the hypothesis  $c^2 <\lambda ^2 $ , the stationary system $ A u  + cv =   cu+ A v = 0 $ has no non-trivial solution, we conclude that  $u= v= 0$  and therefore the solution tends to $(0,0,0,0)$ . In fact, the system generates a uniformly bounded semi-group in  $ D(A^{1/2})\times D(A^{1/2})\times H \times H$  and it is then easy to conclude that $(0,0,0,0)$ is \index{asymptotically stable} asymptotically stable. For the exact nature of the convergence we refer to \cite{MR1914654} }\end{example}

 \begin{example}{\rm Consider the nonlinear
 wave \index{string} equation
\begin{equation}\label{stringeq} u_{tt} - u_{xx} +g( u_t) = 0 \quad\hbox{  in} \, \,{\mathbb R}^+ \times
\Omega,
\quad     u = 0  \   \hbox{on} \, \, {\mathbb R}^+ \times \partial\Omega \end{equation}
where $\Omega = (0, L) $ is an bounded interval of $\R$ and $ g$ in a non-decreasing locally Lipschitz continous function on 
$ {\mathbb R}$  which satisfies $g(0) = 0$ and does not vanish identically in any neighborhood of $0$. Then for any solution $u$ of  \eqref{stringeq} such that $(u(0), u_t(0)) \in H^2\cap H^1_0(0, L)\times H^1_0(0, L)= W $,  the vector $(u,u_t)$ remains and is  bounded in $W$  for $t\ge0$ and we have $$\frac{d}{dt} \int_\Omega (u_t^2(t, x) +  u_x^2(t, x))dx=- 2\int_\Omega g(u_t) u_t(t, x)dx $$  By using the fact that for every regular solution $v$ of the usual string equation, $v_t(t, x)$ is $2L$-periodic with mean-value $0$, the invariance principle now shows that $(u,u_t)$   tends to $(0, 0)$ in $H^1_0(0, L)\times L^2(0, L)$ as $t$ tends to infinity .}\end{example}
\begin{rem}\label{monotonewave.} {\rm  The analog of the last example is valid in higher dimension in a much more general context, relying on the theory of monotonicity in Hilbert spaces and the concept of almost periodic functions. Since these methods fall outside the scope of this text, we refer to \cite{ MR0559667, MR0610796} for the statements and proofs of the general results.}
\end{rem}


\chapter[Some basic examples]{Some basic examples}\label{Somebasicexamples}
In this chapter, we consider a few special cases in which asymptotic behavior can be studied completely by simple direct methods. These examples will serve later as models to undersand more complicated systems.
\section{Scalar first order autonomous \index{autonomous} ODE}

In this section we consider the simplest possible differential equation
\begin{equation}\label{scalarFirstOrderODE}
u'+f(u)=0,\quad t\geq0
\end{equation} The asymptotic behavior of bounded trajectories is obvious as shown by the following result
\begin{thm}\label{ThmCVFirstOrderODE}
Let $f\in W^{1,\infty}_{loc}(\R,\R)$.
 Each   global \index{global solution} and bounded solution $u(t)$ of \eqref{scalarFirstOrderODE} on ${\mathbb R}^+$ tends to a  limit \index{convergence result}
$c$ with  $f(c) = 0$.
\end{thm}
 \begin{proof}[{\bf Proof.}] If for some $\tau >0$ we have $f(u(\tau)) = 0$, then $u(t) =
u(\tau)$ for all $t$  and the result is trivial. If $f(u(t))$
never vanishes on ${\mathbb R}^+$, it keeps a constant sign and $u(t)$
is monotone on ${\mathbb R}^+$. Since by hypothesis $u(t)$ is bounded
on ${\mathbb R}^+$, it follows immediately that $u(t)$ tends to a
limit $c$ as  $t\rightarrow +\infty.$ The equation shows that
$u'(t)$ tends to $-f(c)$, and we conclude that  $f(c) = 0.$
 \end{proof}
\section{Scalar second order autonomous \index{autonomous} ODE}

We now consider the slightly more complicated case of the equation
\begin{equation}\label{scalarSOrderODE}
u''+g(u')+f(u)=0,\quad t\geq0
\end{equation}
where $f,g:\R\longrightarrow \R$ are  locally Lipschitz continuous  such that
$$\forall v\in \R\setminus\{0\},\quad g(v)v>0.$$
The term $ - g(u')$ can be viewed as a  dissipation while $- f(u)$ represents a restoring force. We will show that convergence or divergence of the general solution of equation \eqref{scalarSOrderODE} depends on the strength of the dissipative term $\vert g(v)\vert $ for small values of the velocity $v.$  As a consequence of Corollary 7.5.1, \eqref{scalarSOrderODE} generates a gradient-like \index{gradient-like} system. 
\subsection{A convergence result}
\begin{thm}\label{Th1H86ODE}
Assume that $f,g$ are as above and   in addition, for some $\varepsilon \in(0,1]$ and $\delta > 0$, we have 
\begin{equation}\label{Eq1.1H86ODE}
 \forall v\in\R,\quad g(v)v\geq\delta\inf\{1,\vert v\vert^{3-\varepsilon}\}.
\end{equation}
Then if $u\in W^{1,\infty}(\R^+)$ is a solution of \eqref{scalarSOrderODE}, we have \index{convergence result}
\begin{equation*}
\lim_{t\to+\infty}^{}\{\vert u'(t)\vert+\vert u(t)-c\vert=0,
\end{equation*}
for some $c\in f^{-1}\{0\}$.
\end{thm}
\begin{rem}{\rm A typical example of function $g$ that satisfied hypothesis \eqref{Eq1.1H86ODE} is $g(s)=\vert s\vert^\alpha s$ with $\alpha\in[0,1)$.}
\end{rem}
 \begin{proof}[{\bf Proof.}]
 First , since the system is gradient-like, \index{gradient-like} we have
 $$\omega(u_0,u_1)\subset f^{-1}\{0\}\times\{0\}.$$
 By connectedness we have either $\omega(u_0,u_1)= \{a\}\times\{0\}$ for some $a\in f^{-1}\{0\}$ and the result is established, or$$\omega(u_0,u_1)=[a,b]\times\{0\}, \quad (a<b).$$
In this case, we set $c:=\frac{a+b}{2}$.
As a consequence of the definition of $\omega(u_0,u_1)$, there exists a sequence $(t_n)$ of positive numbers such that
$$\lim_{n\to+\infty}^{}t_n=+\infty,\quad u(t_n)=c.$$
For any $n\in\N$, there exists $\delta_n > 0$ such that
\begin{equation}\label{Eq1.9haraux86ODE}
u(t) \in[a,b],\ \forall t\in [t_n, t_n + \delta_n].
\end{equation}
We claim that for all $n\in\N$ large enough, we can take $\delta_n=+\infty$ in \eqref{Eq1.9haraux86ODE}. Indeed. let
$$\theta_n= Inf\{t > t_n,, u(t)\not\in[a,b]\}$$
and assume $\theta_n < +\infty$. Then we have
\begin{equation}\label{Eq1.11haraux86ODE}
\forall t\in[t_n,\theta_n],\ u''(t)+g(u'(t))=0.
\end{equation}
We may assume that $n$ is large enough to imply $\vert u'(t)\vert\leq 1$ on $[t_n, \infty[$, so that from
 \eqref{Eq1.11haraux86ODE} we
deduce as a consequence of \eqref{Eq1.1H86ODE}
\begin{equation}\label{Eq1.12haraux86ODE}
\forall t\in[t_n,\theta_n],\quad
\vert u'(t)\vert\leq\left\{(1-\varepsilon)\delta(t-t_n)+\vert u'(t_n)\vert^{\varepsilon-1}\right\}^{\frac{-1}{1-\varepsilon}}.
\end{equation}
In fact, if there is $s\in[t_n,\theta_n]$ such that $u'(s)=0$, then $u'(t)=0$ for all $t\in[t_n,\theta_n]$ and
\eqref{Eq1.12haraux86ODE} is obviously satisfied. Otherwise
 $u'$  has a constant sign, then
\begin{eqnarray}
\nonumber\frac{d}{dt}\vert u'\vert^{\varepsilon-1}&=&(\varepsilon-1)u''\vert u'\vert^{\varepsilon-2}\text{ sign }(u')\\
\nonumber &=&(1-\varepsilon)g(u')\vert u'\vert^{\varepsilon-2}\\
\label{Eq12.5Haraux86ODE}&\geq& (1-\varepsilon)\delta.
\end{eqnarray}
By integrating \eqref{Eq12.5Haraux86ODE} over $(t_n,t)$ $(t\in[t_n,\theta_n])$, we get \eqref{Eq1.12haraux86ODE}.
Now from \eqref{Eq1.12haraux86ODE} we deduce by integration
$$\int_{t_n}^{t}\vert u'(s)\vert\,ds\leq\frac{1}{\varepsilon\delta}\vert u'(t_n)\vert^\varepsilon,
\quad \forall t\in[t_n,\theta_n]. $$
For $n$ large enough, the right-hand side is less than $\frac{b-a}{2}=\vert b-c\vert=\vert a-c\vert.$\\
Therefore there exists $n_0\in\N$ such that we have
$$\forall n\geq n_0,\ u(t)\in J,\ \forall t\in[t_n,+\infty[ \hbox{ and }$$
$$\vert u(t_n)-u(t)\vert\leq \frac{1}{\varepsilon\delta}\vert u'(t_n)\vert^\varepsilon,\quad \forall t\in[t_n,+\infty[.$$
Since $u(t_n) = c$ for all $n\in \N$, we deduce
\begin{equation}\label{Eq1.15haraux86ODE}
\forall n\geq n_0,\ \forall t\in[t_n,+\infty[,\
\vert u(t)-c\vert\leq \frac{1}{\varepsilon\delta}\vert u'(t_n)\vert^\varepsilon.
\end{equation}
Since $u'(t_n)\longrightarrow0$ as $n\to +\infty$, it is clear that \eqref{Eq1.15haraux86ODE} implies
$$\lim_{t\to+\infty}^{}\vert u(t)-c\vert=0.$$
Therefore $J = \{c\}$ and this contradicts the hypothesis $J = [a, b]$ with $a < b$.
The proof of theorem \ref{Th1H86ODE} is completed.
 \end{proof}
\subsection{A non convergence result}
\begin{thm}\label{ThmNonCVSOODE} Assume that there exists $a,b\in\R$ with $a<b$ and a positive constant $C$ such that
\begin{eqnarray}
\nonumber & &f(s)<0,\quad \forall s< a\\
\nonumber& &f(s)=0,\quad \forall s\in[ a,b]\\
\nonumber& &f(s)>0,\quad \forall s>b\\
\label{HypothesisOngNonCVSOODE}& & \vert g(v)\vert\leq C v^2,\quad \forall \vert v\vert\leq1.
\end{eqnarray}
Then for every bounded non constant solution of \eqref{scalarSOrderODE}, there exist a sequence
$t_n\longrightarrow+\infty$ such that $u(t_n)<a$ for all $n$ and a sequence $\theta_n\longrightarrow+\infty$ \index{non convergence result}
such that $u(\theta_n)>b$ for all $n$.
\end{thm}
\begin{rem}{\rm A typical example of function $g$ that satisfied hypothesis \eqref{HypothesisOngNonCVSOODE}
 is $g(s)=\vert s\vert s$.}
\end{rem}
In the proof, we have to use the following lemma.
\begin{lem}\label{lemLivreSDPage61} Let $v\in C^2(\R^+,\R)$  satisfying
$$v'(0)>0,\quad v''(t)\geq -C v'(t)^2,\quad \forall t\in\R^+,$$
where $C>0$ is a constant. Then $v$ is nondecreasing and $\displaystyle\lim_{t\to+\infty}^{}v(t)=+\infty$.
\end{lem}
 \begin{proof}[{\bf Proof.}] It is clear that $v'(t)>0$ for $t$ small enough. Let
$$T=\sup\{\tau\geq0,\ v'(t)>0, \ \forall t\in[0,\tau)\}.$$
For all $t\in[0,T)$, we have
$$\frac{d}{dt}(\frac{1}{v'(t)})=\frac{-v''(t)}{(v'(t))^2}\leq C.$$
By integrating over $(0,t)$, we get
\begin{equation}\label{EqLemmaharaux86ODE}
\forall t\in[0,T),\ v'(t)\geq\frac{1}{Ct+\frac{1}{v'(0)}}.
\end{equation}
If $T<+\infty$, we obtain that $v'(t)>0$  in a right neighborhood of $T$ which contradicts the definition of $T$.
 Then $T=+\infty$ and \eqref{EqLemmaharaux86ODE} becomes
$$\forall t\in[0,+\infty),\ v'(t)\geq\frac{1}{Ct+\frac{1}{v'(0)}}.$$
By integrating this inequality, we get the last part of the lemma.
 \end{proof}
 \begin{proof}[{\bf Proof of theorem \ref{ThmNonCVSOODE}}]Since the system in $(u, v)$ is gradient-like \index{gradient-like} we have
$$\lim_{t\to+\infty}^{}\vert u'\vert+\hbox{dist}(u(t),[a,b])=0.$$
Assume that $u(t)\geq a$ for $t\geq t_0$. We must prove that $u$ is constant. We distinguish two cases :

- If $u'(t)\geq0$ on $[t_0,+\infty)$, then $u$ is nondecreasing and tend to $c\in f^{-1}(\{0\}).$
So $f(u(t))=0$ on $[t_0,+\infty)$ and we have
$$u''+g(u')=0, \hbox{ on }[t_0,+\infty).$$
If $u'=0$ on $[t_0,+\infty)$, then $u$ is constant. Otherwise, there exists $t_1\geq t_0$ such that $u'(t_1)>0$.
Applying lemma \ref{lemLivreSDPage61} to $v(t):=u(t+t_1)$, we get a contradiction.

- If there exists $t_1\geq t_0$ such that $u'(t_1)<0$, therefore (since $u(t)\geq a$ when $t\geq t_0$)
$$u''+g(u')=-f(u)\leq0 \hbox{ on } [t_0,+\infty).$$
In particular, $u''\leq -g(u')\leq Cu^2$ on  $[t_0,+\infty)$ and then $v(t):=-u(t+t_1)$ verify
$$v'(0)>0,\quad v''(t)\geq -C v'(t)^2,\quad \forall t\in\R^+.$$
Applying lemma \ref{lemLivreSDPage61} to $v$, we get a new contradiction.
\end{proof}
\section{Contractive and unconditionally stable \index{unconditionally stable} systems}

In this section, $(Z, d)$ denotes a complete metric space and we consider a dynamical system \index{dynamical system} $\displaystyle \{S(t)\}_{t\geq0}$ on $(Z, d).$ The main result is as follows. 
\begin{thm}\label{Uncond-stab}
Assume that the system  $\displaystyle \{S(t)\}_{t\geq0}$ is unconditionally stable \index{unconditionally stable} in the following sense 
\begin{equation}\label{UNCOND}
 \forall\varepsilon>0, \exists \delta > 0,\quad \forall (x, y) \in X\times X, \,\, d(x, y) < \delta\Longrightarrow \sup _{t\ge 0} d(S(t)x, S(t) y) < \varepsilon.
\end{equation}
Let $\cal{F}$ be given by \eqref{equilibriumset}.
Then if $u_0\in X $ generates a precompact trajectory under $S(t)$ and if \index{$\omega$-limit set} $\omega (u_0)\cap {\cal{F}} \not = \emptyset $, the trajectory $S(t)u_0$ converges to some limit \index{convergence result} $a\in \cal{F}$ as $t\rightarrow\infty$.
\end{thm}
 \begin{proof}[{\bf Proof.}] Let \index{$\omega$-limit set} $a\in \omega (u_0)\cap {\cal{F}}.$ Given any $\varepsilon>0$ and choosing 
 $\delta>0$ so that \eqref{UNCOND} is fulfilled, by definition there is $\tau>0$ for which $$ d(S(\tau)u_0, a)<\delta $$ Then we have 
 $$ \forall t\ge \tau, \quad d( S(t)u_0, a) = d(S(t-\tau)S(\tau)u_0, S(t-\tau)a)<\varepsilon. $$
 \end{proof}
 
 \begin{rem}{\rm Actually the above proof shows the following more general result: if $u_0\in X $ generates a precompact trajectory under $S(t)$ and if \index{$\omega$-limit set}$\omega (u_0)\cap {\cal{F}} $ contains a {\it {stable}} \index{stable} equilibrium \index{equilibrium point} point $a$, the trajectory $S(t)u_0$ converges to $a\in \cal{F}$ as $t\rightarrow\infty.$}
\end{rem}
A classical class of unconditionally stable \index{unconditionally stable} systems is the class of contractive systems:

\begin{defn} A dynamical system \index{dynamical system} $\displaystyle \{S(t)\}_{t\geq0}$ on $(Z, d)$ is said to be {\it contractive} if 
\begin{equation}\label{contract} \forall (x, y) \in X\times X, \forall t\ge 0,  d(S(t)x, S(t) y) \le d(x, y)
\end{equation} \end{defn}

An obvious consequence of Theorem \ref{Uncond-stab} is the following

\begin{cor}\label{Contract}
Assume that the system  $\displaystyle \{S(t)\}_{t\geq0}$ is contractive. Then if $u_0\in X $ generates a precompact trajectory under $S(t)$ and if \index{$\omega$-limit set} $\omega (u_0)\cap {\cal{F}} \not = \emptyset $, the trajectory $S(t)u_0$ converges \index{convergence result} to some limit $a\in \cal{F}$ as $t\rightarrow\infty$.
\end{cor} 

 More generally, we have 
 
 \begin{cor}\label{almost-contract}
Assume that the system  $\displaystyle \{S(t)\}_{t\geq0}$ is such that for some $M\ge 1$ 
\begin{equation}\label{contract} \forall (x, y) \in X\times X, \forall t\ge 0,  d(S(t)x, S(t) y) \le M d(x, y)
\end{equation} Then if $u_0\in X $ generates a precompact trajectory under $S(t)$ and if \index{$\omega$-limit set} $\omega (u_0)\cap {\cal{F}} \not = \emptyset $, the trajectory $S(t)u_0$ \index{convergence result} converges to some limit $a\in \cal{F}$ as $t\rightarrow\infty$.
\end{cor}

Theorem \ref{Uncond-stab} especially applies to gradient-like \index{gradient-like} systems.

\begin{defn} A dynamical system \index{dynamical system} $\displaystyle \{S(t)\}_{t\geq0}$ on $(Z, d)$ is said to be {\it gradient-like } \index{gradient-like} if whenever $u_0\in X $ generates a precompact trajectory under $S(t)$, we have $\omega(u_0) \subset \cal F.$
 \end{defn}

 \begin{cor}\label{grad-stab}
Assume that the system  $\displaystyle \{S(t)\}_{t\geq0}$ is gradient-like and \index{unconditionally stable} unconditionally stable. Then if $u_0\in X $ generates a precompact trajectory under $S(t)$ , the trajectory $S(t)u_0$ converges to some limit \index{convergence result} $a\in \cal{F}$ as $t\rightarrow\infty$.
\end{cor}

 \begin{rem}{\rm If we consider the ODE $$ u'' + u = 0 $$ written on $\R^2$ as a system $$ u'= v;\quad v'= -u$$ it is easy to check that any trajectory starting from $U_0 = (u_0, v_0)\not= (0, 0)$ is non-convergent. Here $S(t)$ is an isometry group on $\R^2$ , hence trivially contracting. What happens here is that the system is not \index{gradient-like} gradient-like. More precisely, whenever $U_0 = (u_0, v_0)\not= (0, 0)$, we have $\omega(U_0)\cap {\cal{F}} = \emptyset $ since ${\cal{F}} = \{0\}$ and the norm of $S(t)U_0$ is constant.}
\end{rem}

As a basic application of theorem \ref{Contract}, Let $ N \geq 1$ and  $F \in C^{2}({\R}^N)$ be {\bf{convex}}. We
consider the equation \eqref{gradientsystem}
$$ u'(t) +  \nabla F(u(t)) = 0$$ We obtain
\begin{cor}\label{gradientsystconvex} Assume that $ \cE  = \{ z\in {\mathbb R}^N, \nabla F(z) = 0\}\not= \emptyset.$ Then any solution $u(t)$ of \eqref{gradientsystem} is
bounded on ${\R}^+$ and converges, as $t\rightarrow\infty$ to some limit \index{convergence result} $a\in  \cE  = \{ z\in {\mathbb R}^N, \nabla F(z) = 0\}.$
\end{cor}
 \begin{proof}[{\bf Proof.}] We already showed that  the dynamical system \index{dynamical system} $S(t)$ generated by \eqref{gradientsystem} on the closure of the range of $u $ is gradient-like \index{gradient-like} with set of equilibria $ \cE. $ Under the hypothesis that $F$ is convex, it is easy to check that the operator $\nabla F \in C^{1}({\R}^N, {\R}^N)$ is monotone, which means 
 $$ \forall (u, v) \in {\R}^N\times{\R}^N, \quad \langle \nabla F(u)- \nabla F(v), u-v\rangle\ge 0$$ Then if $(u, v)$ are 2 solutions of \eqref{gradientsystem}, we have 
 
 $$ \forall t\ge 0,\quad  \frac{d}{dt} \Vert u(t)-v(t)\Vert^2 = - 2 \langle \nabla F(u(t))- \nabla F(v(t)), u(t)-v(t)\rangle\le 0$$ Hence the system $S(t)$ is contractive in the usual norm. In particular, since any $a\in  \cE  = \{ z\in {\mathbb R}^N, \nabla F(z) = 0\}$ is a solution of \eqref{gradientsystem} independent of t, the function 
 $$t\mapsto   \Vert u(t)-a\Vert $$ is non-increasing and all trajectories are bounded. Finally Corollary \ref{grad-stab} gives the result.
\end{proof}
 \begin{rem}{\rm A much more general convergence result holds true for the equation $$ 0\in u'+ \partial \Phi (u)  $$ where $ \partial \Phi (u) $ is the (possibly multivalued) subdifferential of any proper convex lsc function with arbitrary domain on a Hilbert space $H$, cf. Bruck \cite{MR0377609}. In general only weak convergence is obtained, cf \cite{MR0496964}. Besides, the asymptotic behavior of precompact trajectories of nonlinear contraction semi-groups has been the object of intensive study in the seventies, cf. e.g. \cite {MR0513814, MR0346611, MR0610796}.}
\end{rem}

\section{The finite dimensional case of a result due to Alvarez}

In this section, we consider  the equation \eqref{2gradientsystem}
$$u''(t) + u'(t) +  \nabla F(u(t)) = 0$$
where $ N \geq 1$ and  $F \in C^{2}({\R}^N)$ is {\bf{convex}}. In contrast with the \index{gradient system} gradient system \eqref{gradientsystem}, the system generated by \eqref{2gradientsystem} is  gradient-like \index{gradient-like} but  generally non-contractive. However we have a convergence result similar to Corollary \ref{gradientsystconvex} which is a special case of a more general weak convergence theorem due to Alvarez, cf. \cite{MR1760062}.

\begin{cor}\label{gradientsystConvex2} Assume that $ \cE  = \{ z\in {\mathbb R}^N, \nabla F(z) = 0\}\not= \emptyset.$ Then any solution $u(t)$ of \eqref{2gradientsystem} is \index{global solution} global, bounded on ${\R}^+$ and converges, as $t\rightarrow\infty$ to some limit \index{convergence result} $a\in  \cE  = \{ z\in {\mathbb R}^N, \nabla F(z) = 0\}.$
\end{cor}
 \begin{proof}[{\bf Proof.}]  From our Hypothesis it follows that $F$ is bounded from below. First we consider a local solution $u$ of \eqref{2gradientsystem} on some interval $[0, L)$ Given any positive $T< L$, the identity 
 $$\int_0^T \Vert u'(t)\Vert^2 dt + \frac{1}{2}\Vert
u'(t)\Vert^2 = F(u(0)) - F(u(t)) +\frac{1}{2}\Vert u'(0)\Vert^2 $$  shows that  $ u'\in L^\infty (0, T; {\mathbb R}^N)$, therefore the solution is \index{global solution} global and uniformly Lipschitz. In addition $ u' \in L^2(\R^+, X)$ with $X= {\R}^N$. We already showed that if all solutions  $U = (u, u') $ are bounded, the system $S(t)$ generated by \eqref{2gradientsystem} is gradient-like \index{gradient-like} and the set  of fixed points of $S(t)$ is $\cF=\cE\times \{0\}$.  We now show that in fact $u$ is bounded and the numerical function $\varphi(t)= \Vert u(t)-a\Vert^2 $ has a limit at infinity whenever $a\in \cE$. 
Indeed a straightforward calculation shows that $\varphi\in C^{2} $ with 
$$\varphi ''+ \varphi ' = -2 \langle \nabla F(u(t)),  u(t)-a)\rangle +2  \Vert u'(t)\Vert^2 \le 2  \Vert u'\Vert^2  = h \in L^1(\R^+) $$ Writing this inequality as
$$ (e^t  \varphi ')' \le e^t  h(t) $$ provides 
$$ \varphi ' (t)  \le e^{-t} \varphi ' (0) +  \int_0^t e^{s-t}  h(s) ds : = H(t) + e^{-t} \varphi ' (0) = K(t) $$  Now we have 
$$ \int_0^T  H(t) dt = \int_0^T \int_0^t e^{s-t}  h(s) ds dt = \int_0^T e^{s}  h(s) \int_s^T e^{-t} dt  ds $$
$$ =\int_0^T e^{s}  h(s) (e^{-s}- e^{-T} ) ds  \le \int_0^T  h(s) ds $$ Thus $H, K \in L^1(\R^+) $ and since $\varphi\ge 0$, the function  $\psi(t): = \varphi (t) - \int_0^t  K(s) ds $ is bounded with non-positive derivative. It tends to a limit at infinity and so does  $\varphi$. In particular $u$ is bounded, and since $S(t)$ is \index{gradient-like} gradient-like, the omega-limit set is contained in $\cF$. Picking $(a, 0) \in \omega (U_0)$ \index{$\omega$-limit set}, the limit of  $\varphi$ at infinity  is $0$ and we end up with convergence of $u$ to $a$ and $u'$ to $0$

\end{proof}

\chapter[The convergence problem in finite dimensions]{The convergence problem in finite dimensions}

\section{A first order system}

In this section we consider the first order gradient system \index{gradient system}
\begin{equation}\label{SDPremierOrdre}
u'+\nabla\varphi(u)=0
\end{equation}
where $\varphi:\R^N\longrightarrow\R$ is assumed to be $C^1$, and we set $${\cal S}=\{a\in\R^N,\ \nabla \varphi(a)=0\}.$$
As we saw in Section \ref{SystemGradient}, any bounded solution of \eqref{SDPremierOrdre} approaches the set ${\cal S}$ as $t$ goes to infinity. The question is then to determine whether or not it actually converges to a point in ${\cal S}$. The next result shows that this is not always true. 
\subsection{A non convergence result}

\begin{thm}\label{PalisDeMeloAbsil} \index{non convergence result} Let $k$ be a positive integer and let us consider
\begin{equation} \label{CoordCartPolaire}\varphi(x,y)= f(r,\theta)=  \left\{ \begin{array}{ll}
e^{-\frac{1}{(1-r^2)^{k}}} \left[ 1-\frac{4k^2r^{4}}{4k^2r^{4}+{(1-r^{2})}^{2k+2}}\sin(\theta-\frac{1}{{(1-r^2)}^{k}})\right] & \hbox{ if } r<1 ,\\[2mm]
0  & \hbox{ if } r\geq1.
\end{array} \right.
\end{equation} where we use the polar coordinates $(x,y)=(r\cos\theta,r\sin\theta)$. Then there exists  a bounded  solution $u$
of \eqref{SDPremierOrdre} whose
$\omega$-limit \index{$\omega$-limit set} set is homeomorphic to ${ S}^1$.
\end{thm}
\begin{proof}[{\bf Proof.}] For $N=2$, by setting $u=(x,y)$ equation \eqref{SDPremierOrdre} becomes 
\begin{equation} \label{SDPremierOrdreN=2}
\left\{ \begin{array}{ll}
x'+\frac{\partial\varphi}{\partial x}(x,y)=0, &  \\[2mm]
y'+\frac{\partial\varphi}{\partial y}(x,y)=0 , & 
\end{array} \right.
\end{equation}

 The system \eqref{SDPremierOrdreN=2} becomes
\begin{equation} \label{EquCoordPolaire}
\left\{ \begin{array}{ll}
r'+\frac{\partial f}{\partial r}(r,\theta)=0, &  \\[2mm]
\theta'+\frac{1}{r^2}\frac{\partial f}{\partial \theta}(r,\theta)=0 , & 
\end{array} \right.
\end{equation}
We define

Let $r_0\in(0,1)$ and let $r$ be the local  solution of 
\begin{equation*}
\left\{ \begin{array}{ll}
r'-\frac{2kr {(1-r^{2})}^{k+1}}{4k^2r^{4}+{(1-r^{2})}^{2k+2}}
e^{-\frac{1}{(1-r^2)^{k}}}=0.&  \\[2mm]
r(0)=r_0 & 
\end{array} \right.
\end{equation*}
 Clearly, $r$ is global\index{global solution} and  satisfies 
$$\forall t\in(0,+\infty), \ 0<r(t)<1,\quad\hbox{and }\ \displaystyle\lim_{t\to\infty}^{}r(t)=1.$$
 Now if we impose that \begin{equation}\label{CdCourbe}\theta=\frac{1}{(1-r^2)^{k}} \end{equation}
 then a straightforward calculation shows that $(r,\theta )$ is a solution of  \eqref{EquCoordPolaire}. Hence, the solution $(r,\theta)$ verifies
 $$\lim_{t\to\infty}^{}r(t)=1,\quad \lim_{t\to\infty}^{}\theta(t)=\infty.$$
Clearly, the $\omega$-limit \index{$\omega$-limit set} set of the trajectory $u=(r\cos\theta,r\sin\theta)$ of  \eqref{SDPremierOrdreN=2} with $\varphi$ given by \eqref{CoordCartPolaire}
and $(r,\theta)$ satisfying \eqref{CdCourbe} is the entire circle $\{(r,\theta)/\ r=1\}$. \end{proof}

\begin{rem}{\rm We  recall that a function $f\in C^{\infty}(\R^N,\R)$ is in the uniform Gevrey class $G_{1+\delta}(\R^N,\R)$ if there exists a constant $M = M(f) >0$ for which 
$$ \forall m\in \N^N, \quad \Vert D^mf\Vert _{L^\infty} \le M^m \vert m\vert ^{(1+\delta) \vert m\vert}$$ where 
$$\vert m\vert : = \sum_{j=1}^N{m_j} $$ is the length of the differentiation index 
$m$. It is natural to conjecture that, written in cartesian coordinates,   $\varphi\in G_{1+\frac{1}{k}}$ outside any ball centered at 0  and therefore $\rho\varphi\in G_{1+\frac{1}{k}}(\R^2,\R)$ for any $\rho\in G_{1+\frac{1}{k}} (\R^2,\R)$ which vanishes in a small ball around 0 and is equal to 1 ouside the ball of radius $\varepsilon<1$. If the conjecture is valid, this reinforces to the stronger regularity class  $G_{1+\delta}(\R^2,\R) \subset  C^{\infty}(\R^2,\R)$  with $\delta =\frac{1}{k}$  the non-convergence result from J. Palis and W. De Melo \cite {MR0669541} which stated the existence of $\varphi\in C^{\infty}(\R^2,\R)$ for which there is  a bounded  solution $u$ of \eqref{SDPremierOrdre} whose $\omega$-limit \index{$\omega$-limit set} set is homeomorphic to ${ S}^1$. As $\delta$ tends to $0$  the space $G_{1+\delta}(\R^2,\R) $ approaches the space of analytic \index{analytic} functions $G_{1}(\R^2,\R) $ , showing that the next result is optimal if we look for a regularity class in which convergence of bounded trajectories is always true. }\end{rem}
\subsection{The analytic \index{analytic} case}\label{Cvpremierordre-Analytique}
In \cite{Lo65,MR0160856}, S. {\L}ojasiewicz proved the following result which implies that the "bad" situation of Theorem \ref{PalisDeMeloAbsil} cannot happen for analytic \index{analytic} functions.
\begin{thm}\label{ThmInegaliteLojasiewicz} ({\L}ojasiewicz \index{Lojasiewicz}Theorem \cite{Lo65,MR0160856})
 Let  $\varphi: \R^N\longrightarrow {{\R}}$ be an analytic \index{analytic} function. Then for all $a\in
 {\cal S}$, there exists $c_a>0$, $\sigma_a>0$ and $0<\theta_a\leq{1\over 2}$ such that :
\begin{equation}\label{InegaliteLojasiewicz}\Vert \nabla \varphi(u)\Vert\geq c_a\vert \varphi(u)-\varphi(a)\vert^{1-\theta_a}\quad \forall u\in{{\R}}^N\ \Vert u-a\Vert<\sigma_a.\end{equation}

\end{thm}

\begin{rem}\label{theta}{\rm In the sequel, $\theta_a$ will be called a {\L}ojasiewicz \index{Lojasiewicz} exponent of $\varphi$ at point $a$. Each $\theta'<\theta_a$ is also a {\L}ojasiewicz \index{Lojasiewicz} exponent of $\varphi$ at point $a$, associated to a possibly smaller radius $ \sigma< \sigma_a$. Moreover when considering  $\theta'$ and reducing $ \sigma$ if needed, the constant $c_a$ can be replaced by arbitrarily large constants, in particular by $1.$  This was the choice made by {\L}ojasiewicz in his pioneering paper. On the other hand, in the cases where an optimal (= largest) $\theta$ can be reached for instance by a direct calculation, it  may happen that the choice $c=1$ is irrelevant. For instance if $N=1$ and  $\varphi(u) = \varepsilon u^2$, we have $\Vert \nabla \varphi(u)\Vert = 2\varepsilon |u| $ so that in particular  $$\Vert \nabla \varphi(u)\Vert = 2{\varepsilon}^{\frac{1}{2}} \varphi(u) ^{1- \frac{1}{2}} $$ In this case the optimal value $\theta = \frac{1}{2}$ is associated to a maximal constant $c_0$ which tends to $0$ with the parameter $\varepsilon$. Similar examples can be built with any super-quadratic power function.}\end{rem}

\begin{rem}{\rm If $a\not\in {\cal S}$,   the inequality becomes trivial since $\varphi$ is of class $C^1$.}
\end{rem}
\begin{thm}\label{ThmCVAnalyticSGDimFini} ({\L}ojasiewicz \index{Lojasiewicz} Theorem \cite{Lo65,MR0160856})
Assume that $\varphi$ satisfies \eqref{InegaliteLojasiewicz} at any equilibrium point $a$ and let $u\in L^{\infty}({\R}^+,{\R}^N)$ be a solution of \eqref{SDPremierOrdre}.
Then there exists $a\in {\cal S}$  such that \index{convergence result}
$$\lim_{t\to+\infty}^{}\Vert u(t)-a\Vert= 0.$$
Moreover, let $\theta$ be any {\L}ojasiewicz\index{Lojasiewicz} exponent of  $\varphi$ at point $a$. Then we have
\begin{equation}\label{DecayEstimaFirstOrder}
\| u(t) - a \| = \left\{ \begin{array}{ll}
O(e^{-\delta t} )  & \hbox{ for some }\delta>0 \text{ if } \theta =\frac12, \\[2mm]
O(t^{-\theta/(1-2\theta)} )  & \text{if } 0<\theta <{1\over2}.
\end{array} \right.
\end{equation}
\ In particular if $\varphi$ is analytic \index{analytic}, all bounded solutions of \eqref{SDPremierOrdre} are convergent.
\end{thm}
\begin{proof} [Proof.] 
We define the function $z$ by $z(t)= \varphi (u(t))$. Then
\begin{equation}\label{inegalite energiephi} z'(t) =-\Vert \nabla   \varphi (u(t))\Vert^2,\quad \forall t\geq0.
\end{equation} 
 So $z$  is nonincreasing. Since $u$ is bounded and  $\varphi$ is continuous,  it follows that $K=\displaystyle\lim_{t\to\infty} \varphi (u(t))$ exists.  Replacing $ \varphi$ by $\varphi-K$ we may assume $K = 0$.  If $z(t_0)=0$ for some $t_0\geq0$, then $z(t)=0$ for every $t\geq t_0$, and therefore, $u$ is constant for $t\geq t_0$. In this case, there remains nothing to prove. Then we can assume that $z(t)>0$ for all $t\geq 0$.\\
Define $\Gamma:=\omega(u)$. Theorem \ref{thm 4.1.8.} ii) implies that $\Gamma$ is compact and connected. Let  $a\in \Gamma$, then there exists $t_n\to+\infty$ such that $u(t_n)\longrightarrow a$. Then we get
$$\lim_{n\to+\infty}^{}\varphi(u(t_n))=\varphi(a)=K=0.$$
On the other hand, $\varphi$ satisfies the {\L}ojasiewicz \index{Lojasiewicz} inequality \eqref{InegaliteLojasiewicz} at every point $a\in {\cal S}$.  Applying Lemma \ref{InegLojsABSRAITLemme} with $W=X=\R^N$, $E=\varphi$ and ${\cal G}=\nabla\varphi$ we obtain,
$$ \exists \sigma, c>0,\ \exists \theta\in(0,\frac12]/\ \left[\hbox{dist}(u,\Gamma)\leq \sigma\Longrightarrow\Vert \nabla \varphi(u)\Vert\geq c\vert \varphi(u)\vert^{1-\theta}\right].$$
Now since $\Gamma=\omega(u)$, by Theorem \ref{thm 4.1.8.} iii),  there exists $T>0$ such that $\hbox{dist}(u,\Gamma)\leq \sigma$. Then we get for all $t\geq T$
\begin{equation}\label{ineqLojTraj}
\Vert \nabla \varphi(u)\Vert\geq c\vert \varphi(u)\vert^{1-\theta}.
\end{equation}
By combining \eqref{inegalite energiephi} and \eqref{ineqLojTraj}, we get
\begin{equation}\label{inequationdifferentielll}z'(t) \le -c^2 (z(t))^{2 (1-\theta)  },\quad \forall t\geq T.
\end{equation}
In the case $\theta\in(0,\frac12)$, by integrating \eqref{inequationdifferentielll} over $(T,t)$ we find
$$ z(t)\leq \frac{1}{(z(T)^{2\theta-1}+(1-2\theta)c^2(t-T))^{\frac{1}{1-2\theta}}}\leq C_1 t^{-\frac{1}{1-2\theta}},
\quad \forall t\geq T.$$
Now since $$ \Vert u'(t)\Vert^2 = -z'(t)$$ we have $$\int_t^{2t}\Vert u'(s)\Vert^2 ds
= z(t)- z(2t)\le C_1 t^{-\frac{1}{1-2\theta}}.$$
 Applying Lemma \ref{LemmaZelenyakPol} to $p(t): =  \Vert u'(t)\Vert $, we get
\begin{equation}\label{int de t a 2t}\int_t^{\infty}\Vert u'(s)\Vert ds\leq C_2 t^{-\frac{\theta}{1-2\theta}}.
\end{equation} 
By Cauchy's criterion, $a:=\displaystyle\lim_{t\to+\infty}^{}u(t)$ exists and
$$\forall t\geq T,\quad \Vert u(t)-a\Vert\leq C_{2} t^{-\frac{\theta}{1-2\theta}}.$$
On the other hand, if $\theta=\frac12$, the application of Lemma \ref{LemmaZelenyakExp} to $p(t): =  \Vert u'(t)\Vert$ gives the exponential decay. To conclude the proof, we remark that at the end, the global  {\L}ojasiewicz exponent used to prove convergence can be replaced by any {\it{local }} {\L}ojasiewicz\index{Lojasiewicz} exponent of  $\varphi$ at $a$. \end{proof}
\begin{rem}{\rm Since the {\L}ojasiewicz \index{Lojasiewicz} theorem is actually local, it suffices to assume that $\varphi$ is analytic \index{analytic} in a ball where the solution stays for all $t$.
}\end{rem}

\section{A second order system}
We now consider the gradient-like \index{gradient-like} system
\begin{equation}\label{SDSecondOrdre}
u''+u'+\nabla\Phi(u)=0
\end{equation}
where $\Phi:\R^N\longrightarrow\R$ is assumed to be $C^1$, and we set $${\cal S}=\{a\in\R^N,\ \nabla \Phi(a)=0\}.$$
\subsection{A non convergence result}
The non-convergence result of Curry - Palis - De Melo (cf. Theorem \ref{PalisDeMeloAbsil}) has been extended to \eqref{SDSecondOrdre} by V\'eron \cite{MR0566294} (see also \cite{MR1753136, MR2018329}). More precisely 
\begin{prop}
\label{prop-friction} 
Given any $\varphi\in C^k(\R^2,\R)$, $1\le k\le\infty$, there is a
$\Phi\in C^{k-1}(\R^2,\R)$  such that  each  solution  of \eqref{SDPremierOrdre}
is at the same time a solution of \eqref{SDSecondOrdre}.
\end{prop}
 \begin{proof}[{\bf Proof.}] The statement is readily satisfied for
   $\Phi=\varphi
-\vert\nabla\varphi\vert^2/2$.
\end{proof}
\begin{cor} \index{non convergence result}
There exist $\Phi\in C^{\infty}(\R^2,\R)$ and a bounded  solution $u$ of \eqref{SDSecondOrdre} whose $\omega$-limit \index{$\omega$-limit set} set is homeomorphic to ${ S}^1$.
\end{cor}
 \begin{proof}[{\bf Proof.}]
Take  $\varphi$ as in Theorem \ref{PalisDeMeloAbsil}. Then \eqref{SDPremierOrdre} has a  bounded  solution $u$
 whose $\omega$-limit \index{$\omega$-limit set} set is homeomorphic to ${S}^1$.
By  Proposition~\ref{prop-friction}, $u$ is also a solution of \eqref{SDSecondOrdre}
for some smooth $\Phi$, which proves the corollary.
\end{proof}
\subsection{A convergence result}\label{CvAnalytiqueSndordre}
\begin{thm}\label{ThmCVAnalyticSGDimFinideuxOrdre}
Assume that $\Phi$ is analytic \index{analytic} and let $u\in W^{1,\infty}({\R}^+,{\R}^N)$ be a solution of \eqref{SDSecondOrdre}.
Then there exists $a\in {\cal S}$  such that \index{convergence result}
$$\lim_{t\to+\infty}^{}\Vert u'(t)\Vert+\Vert u(t)-a\Vert= 0.$$
Moreover, let $\theta$ be any Lojasiewicz\index{Lojasiewicz} exponent of $\varphi$ at $a$. Then we have for some constant $C>0$
$$\Vert u(t)-a\Vert\leq C t^{{-\theta}\over {1-2\theta}},\quad\hbox{if }\ 0<\theta<{1\over2}$$
$$\Vert u(t)-a\Vert\leq C \exp(-\delta t),\quad\hbox{ for some}\ \delta>0
\hbox{ if
}\theta={1\over2}.$$
\end{thm}
\begin{proof}[{\bf Proof.}] Let $E(t)={1\over 2}\Vert u'(t)\Vert ^2 +   \Phi (u(t))$. We  have
\begin{eqnarray*} \frac{d}{dt}(E(t))& =&
\langle u'', u'\rangle + \langle \nabla    \Phi(u), u'\rangle \\
&=&\langle  u'' +  \nabla  \Phi(u), u'\rangle = -\Vert u'(t)\Vert^2
\end{eqnarray*} From Theorem \ref{thm 4.1.8.} ii) we know that $\omega(u,u')$ is a non-empty  compact, connected set. 
We also know that $\lim_{t\to+\infty}\Vert u'\Vert=0$ and $\omega(u,u')\subset {\cal S}\times\{0\}$ (see corollary \ref{gradientsyst2}). Let $\Gamma=\{a/\ (a,0)\in\omega(u,u')\}$ and $K=\lim_{t\to\infty}E(t)$.
As in the proof of theorem \ref{ThmCVAnalyticSGDimFini} we may assume
$K = 0 $  and  for all $a \in  \Gamma  ,
\  \Phi (a) = 0 .$ 
Then we introduce
 $$ H(t) = {1\over 2}\Vert u'(t)\Vert ^2 +   \Phi (u(t))+
 \varepsilon  \langle\nabla    \Phi (u(t)), u'(t)\rangle$$ 
  where $ \varepsilon   $ is to be fixed later. Therefore
\begin{eqnarray*} H'(t) &=& -\Vert
u'\Vert^2 +  \varepsilon  \langle \nabla   \Phi (u), u''\rangle +  \varepsilon \langle\nabla^2  \Phi(u)u', u'\rangle \\
 &=&-\Vert u'\Vert^2  + \varepsilon\langle\nabla  \Phi(u), - u' -\nabla   \Phi(u)\rangle + \varepsilon\langle\nabla^2  \Phi(u)\cdot u', u'\rangle\\
&=&-\Vert u'\Vert^2 - \varepsilon\Vert\nabla  \Phi(u)\Vert^2-\varepsilon\langle \nabla  \Phi(u),u'\rangle + \varepsilon\langle\nabla^2  \Phi(u)\cdot u', u'\rangle.
\end{eqnarray*}
Since $u$ is bounded we have
$$\varepsilon\langle \nabla^2  \Phi(u)\cdot u', u'\rangle \leq C_1 \varepsilon \Vert u'\Vert ^2.$$
Thanks to Cauchy-Schwarz and Young inequalities we have
$$\varepsilon\langle \nabla  \Phi(u),u'\rangle\leq \frac{\varepsilon}{2}\Vert\nabla
 \Phi(u)\Vert^2+\frac{\varepsilon}{2}\Vert u'\Vert^2.$$
Therefore
selecting $\varepsilon \leq\varepsilon_0 $ we find
\begin{eqnarray}\nonumber H'(t) &\leq& -(1-C_2\varepsilon)\Vert u'\Vert^2 -\frac{\varepsilon}{2}\Vert\nabla   \Phi(u)\Vert^2\\
\label{InegEnergieSSO1}& \leq& -\frac{\varepsilon}{2}  \big(\Vert u'\Vert^2 +\Vert\nabla   \Phi (u)\Vert^2\big).
\end{eqnarray} Then $H$ is nonincreasing with limit $0$, we have in particular
$H$ is nonnegative. As in the proof of the Theorem \ref{ThmCVAnalyticSGDimFini}   we can assume that $H(t)>0$ for all $t\geq 0$. On the other hand,  since $\Phi$ is analytic \index{analytic} then  by using Lemma \ref{InegLojsABSRAITLemme} once again as in the proof of Theorem \ref{ThmCVAnalyticSGDimFini}, there exist $\theta\in(0,\frac12]$, $T>0$ such that for all $t\geq T$ we get
\begin{eqnarray}\nonumber \Vert u'\Vert^2 +\Vert\nabla   \Phi(u)\Vert^2 &\ge& \Vert
u'\Vert^2 +{1\over 2}\Vert\nabla   \Phi(u)\Vert^2 + \frac{c^2}{2}\vert   \Phi(u)\vert^{2 (1-\theta) } \\
\nonumber&\geq& c_3 \big(\Vert u'\Vert^2 +\Vert\nabla   \Phi(u)\Vert^2 +\vert    \Phi (u)\vert \big)^{2  (1-\theta)  }\\
\label{InegEnergieSSO2}& \geq& c_4 \big(H(t)
\big)^{2  (1-\theta)  }
\end{eqnarray}  Combining the inequalities \eqref{InegEnergieSSO1} and \eqref{InegEnergieSSO2} we find
$$ H'(t) \leq -c_5 (H(t))^{2 (1-\theta)  }.$$
If $\theta\in(0,\frac12),$ intergrating this differential inequality we get
$$ H(t) \le C_6 t^{-\frac{1}{1-2\theta}}.$$
When $\theta=\frac12$, we find that $H$ decays exponentially.\\
Now from \eqref{InegEnergieSSO1}, we get $$\int_t^{2t} \big(\Vert
u'\Vert^2 +\Vert \nabla    \Phi
 (u)\Vert^2\big) ds \leq \frac{2}{ \varepsilon}H(t).$$ The proof
concludes exactly as in Theorem \ref{ThmCVAnalyticSGDimFini}.
\end{proof}
\section{Generalization}
The goal of this section is to give a general framework which covers the results of section \ref{Cvpremierordre-Analytique} and \ref{CvAnalytiqueSndordre} as well as some new examples.
For this end, we consider the differential equation
\begin{equation} \label{non}
\dot u (t) + \cF (u(t)) = 0 , \quad t\geq 0 ,
\end{equation}
where $\cF\in C (\R^N ;\R^N )$ .
\begin{thm} \label{alain}  
Let $u\in C^1 (\R_+ ;\R^N )$ be a bounded solution of the differential equation \eqref{non}. Assume that there exists a function $\cE \in C^1 (\R^N )$, $\beta\geq 1$, $\theta\in(0,1)$ and $c, c_1,T >0$ such that 
\begin{equation} \label{Cdbetatheta}
 \beta(1-\theta)<1,
\end{equation}
\begin{equation} \label{energiepositive}
  \cE(u(t) )  \geq 0 \text{ for every } t\geq T,
\end{equation}
\begin{equation} \label{condAngle}
\langle \nabla\cE(u(t)) , \cF (u(t)) \rangle \geq c \,   \Vert \nabla\cE(u(t))\Vert^\beta \, \Vert \cF (u(t))\Vert\  \text{ for every } t\geq T 
\end{equation}
\begin{equation} \label{cond1}
\Vert \nabla\cE(u(t)) \Vert \geq c_1 \,   \cE(u(t))^{1-\theta}\  \text{ for every } t\geq T 
\end{equation}
\begin{equation} \label{cond1a}
\text{for every } a\in \R^N \text{ one has}: \quad  \nabla\cE(a ) = 0 \Rightarrow \cF(a) = 0 ,
\end{equation}
Then  there exists $a\in\R^N$ such that $\displaystyle\lim_{t\to\infty} u(t)=a$.\\
 If, moreover, $ \cE$ satisfies for some  $c_2>0$
\begin{equation} \label{cond2}
 \Vert \cF (u(t)) \Vert\geq c_2  \,   \cE(u(t))^{1-\theta} \text{ for every }  t\geq T ,
\end{equation}
Then, as $t\to\infty$,\index{convergence result}
\begin{equation} \label{decayest}
\| u(t) - a \| = \left\{ \begin{array}{ll}
O(e^{-\delta t} )  & \hbox{ for some }\delta>0 \text{ if }\beta=\frac{\theta}{1-\theta}, \\[2mm]
O(t^{-\frac{1-\beta(1-\theta)}{\beta(1-\theta)-\theta}} )  & \text{if } \beta>\frac{\theta}{1-\theta} .
\end{array} \right.
\end{equation}
\end{thm}
 \begin{proof}[{\bf Proof.}] We apply Lemma \ref{strong} with $X=\R^N$ and $H(t)=\cE(u(t))$.
Let $u$ be a solution of \eqref{non} which is continuously differentiable, then, by the chain rule,
\begin{equation*}
- \frac{d}{dt} \cE (u(t)) = - \langle \nabla \cE(u(t)) ,  u' (t) \rangle = \langle \nabla\cE (u(t)) , \cF (u(t))\rangle.
\end{equation*}
By using \eqref{condAngle}, \eqref{cond1} and equation \eqref{non} we get for all $t\geq T$
\begin{eqnarray}
\nonumber- \frac{d}{dt} \cE (u(t)) &\geq&  c \,   \Vert \nabla\cE(u(t))\Vert^\beta \, \Vert \cF (u(t))\Vert\\
\label{inegalite Etheta}  &\geq& c c_1^\beta \,  \cE(u(t))^{\beta(1-\theta)} \, \| u'(t)\| .
\end{eqnarray}
This is condition \eqref{strongcond2} with $\eta:=1-\beta(1-\theta)$ (thanks to \eqref{Cdbetatheta} $\eta>0$.)\\
It follows that the function $t\longmapsto\cE(u(t))$ is nonincreasing. Now if $\cE(u(t_0))=0$ for some $t_0\geq T$, then $\cE(u(t))=0$ for every $t\geq t_0$, and therefore, by conditions \eqref{condAngle}, \eqref{cond1a} and the equation \eqref{non} the function $u$ is constant for $t\geq t_0$. In this case, there remains nothing to prove. Hence we can assume $\cE(u(t))>0$ for all $t\geq T$. This is condition \eqref{energiepositive2}. By applying Lemma \ref{strong} we deduce the convergence result. Now we will prove the decay estimate \eqref{decayest}. From \eqref{inegalite Etheta} we deduce for all $t\geq T$
\begin{equation}\label{InegalitecEbis}- \frac{d}{dt} [\cE (u(t))]^\eta\geq \eta c c_1^\beta \,   \| u'(t)\|.\end{equation}
By integrating this last inequality we get
 \begin{eqnarray}
\label{e2} \| u(t) - a \| & \leq & \int_{t}^\infty \|  u' (s) \| \; ds \\
\nonumber & \leq & \frac{1}{c \eta c_1^\beta} \cE(u(t))^{\eta}  .
\end{eqnarray}
By using hypothesis \eqref{cond2} and equation \eqref{non}, we get
\begin{equation}\label{InegalitecEbisbis}
[\cE(u(t))^\eta]^{\frac{1-\theta}{\eta}}=\cE(u(t))^{1-\theta}\leq \frac1{c_2}  \Vert \cF (u(t)) \Vert=\frac1{c_2}  \Vert u'(t) \Vert
\end{equation}
Combining \eqref{InegalitecEbis} and \eqref{InegalitecEbisbis}, we obtain
$$ \frac{d}{dt} [\cE (u(t))^\eta]\leq -\eta c c_1^\beta c_2[\cE^\eta]^{\frac{1-\theta}{\eta}}.$$
Solving this differential inequality (we have to distinguish two cases  $\frac{1-\theta}{\eta}=1$ or $\frac{1-\theta}{\eta} >1$), we obtain the estimate 
$$\cE (u(t))^\eta= \left\{ \begin{array}{ll} O(e^{-Ct} )  & \text{if } \beta=\frac{\theta}{1-\theta} , \\[2mm] 
O(t^{-\eta/(1-\eta-\theta)} )  & \text{if } \beta>\frac{\theta}{1-\theta} .
\end{array} \right.
$$
Combining this estimate with \eqref{e2}, the claim follows.
\end{proof}

In the next subsections we discuss several applications of our abstract results.

\subsection{A gradient system in finite dimensions}\index{gradient system}
We start by applying our abstract results to  the gradient system \index{gradient system}
$$ u' (t) + \nabla \varphi (u(t)) = 0 ,$$
where $\varphi \in C^2 (\R^N )$. The system is a special case of \eqref{non} if we take $\cF = \nabla \varphi$.  The function $\varphi$ is nonincreasing along $u$. Now if $u$ is a bounded solution of the above gradient system \index{gradient system} and since $\varphi$ is continuous, it follows that $\varphi_\infty=\displaystyle\lim_{t\to+\infty} \varphi(u(t))$ exists. If we define $\cE$ by
$$\cE(v)= \varphi(v)-\varphi_\infty$$
we see that hypothesis \eqref{energiepositive} is satisfied for all $t\geq0$. If $\varphi$ is real analytic \index{analytic}, then it satisfies {\L}ojasiewicz \index{Lojasiewicz} inequality \eqref{InegaliteLojasiewicz}. Therefore by applying lemma \ref{InegLojsABSRAITLemme} with $W=X=\R^{N}$, $E=\varphi$, ${\cal G}=\nabla\varphi$ and $\Gamma=\omega(u)$ we get
$$\exists T>0,\ \exists c>0, \ \exists \theta\in(0,\frac12]/\quad \Vert \nabla \varphi(u(t))\Vert\geq
c\vert \varphi(u(t))-\varphi_\infty\vert^{1-\theta},\quad \forall t\geq T.$$
Now it easy  to see that all hypotheses of Theorem  \ref{alain} are satisfied (here  $\beta=1$). Then there exists $a\in \R^N$ such that $\displaystyle\lim_{t\to\infty} u(t) = a$ and the estimate
\begin{equation*}
\| u(t) - a \| = \left\{ \begin{array}{ll}
O(e^{-\delta t} )  & \text{if } \theta =\frac12, \\[2mm]
O(t^{-\theta/(1-2\theta)} )  & \text{if } \theta <\frac12 .
\end{array} \right.
\end{equation*}
 We thus recover the result of Section  \ref{Cvpremierordre-Analytique}.

\subsection{A second order ordinary differential system}
Let $\Phi\in C^2 (\R^N)$ and consider the second order ordinary differential system
\begin{equation} \label{second1}
 u'' (t) +  u' (t) + \nabla\Phi (u(t)) = 0 .
\end{equation}
This system is equivalent to the first order system \eqref{non} if we define $\cF :\R^{2N} \to \R^{2N}$ by
$$
\cF (u,v) := \left( \begin{array}{c} -v \\ v + \nabla \Phi(u) \end{array} \right) , \quad u, \, v\in\R^N.
$$
Now let $u\in W^{1,\infty}({\R}^+,{\R}^N)$ be a solution of \eqref{second1}. We define the energy \index{energy} of this system  
$$E(t)=\frac12 \| u'(t)\|^2 + \Phi(u(t)).$$
We know that the function $E$ is nonincreasing and $E_\infty=\lim_{t\to\infty} E(t)$ exists. It is also well known that $\omega(u,u')$ is compact connected subset of $\Phi^{-1}(\{0\})\times\{0\}$ (see  corollary \ref{gradientsyst2}). 
Let $\varepsilon >0$, and define $\cE :\R^{2N} \to \R$ by
$$
\cE (u,v) :=   \frac12 \| v\|^2 + \Phi(u)-E_\infty + \varepsilon \langle\nabla \Phi(u) , v\rangle_{\R^N}  , \quad u , \, v\in\R^N ,
$$
so that
$$\nabla \cE (u,v) =   \left( \begin{array}{c} \nabla \Phi(u)  \\ v \end{array}\right)  +\varepsilon \, \left( \begin{array}{c} \nabla^2 \Phi(u) v \\ \nabla \Phi(u) \end{array} \right).$$
Fix $R\geq 0$, and let $M:= \sup_{\| u\|\leq R+1} \| \nabla^2 \Phi(u)\|$. Choose $\varepsilon \in(0,1)$ small enough so that $(M+\frac12)\varepsilon \leq \frac12$. Then for every $u$, $v\in\R^N$ satisfying $\| u\|\leq R$ we obtain
\begin{eqnarray}
\nonumber& & \langle\nabla \cE(u,v) , \cF (u,v)\rangle_{\R^{2N}} \\[1mm]
\nonumber& = &  \| v\|^2 - \varepsilon \, \langle \nabla^2 \Phi(u) v , v \rangle_{\R^N}   + \varepsilon \, \langle v , \nabla \Phi(u) \rangle_{\R^N} + \varepsilon \, \| \nabla \Phi(u)\|^2  \\[1mm]
\nonumber& \geq & (1-M\varepsilon - \frac{\varepsilon}{2} ) \, \| v\|^2 + \frac{\varepsilon}{2} \, \| \nabla \Phi(u)\|^2 \\[1mm]
\label{E=0implyF=0}& \geq & \alpha' \, (\| v\|^2 + \| \nabla \Phi(u) \|^2 ) .
\end{eqnarray}

Since $\frac{d}{dt}[\cE(u(t),u'(t))]=-\langle\nabla\cE(u,v),\cF (u(t),u'(t)) \rangle\leq0$,  Then  the function $t\longmapsto \cE(u(t),u'(t))$ is nonincreasing. 
Thanks to the fact that $u'\longrightarrow0$ as $t\to\infty$,  it follows that $\displaystyle\lim_{t\to+\infty} \cE(u(t),u'(t))=0$. 
Then $\cE$ satisfy hypothesis \eqref{energiepositive}. Moreover,
\begin{equation}\label{ineg EF}
\| \nabla \cE (u,v) \| + \| \cF (u,v) \| \leq  C (\| v\| + \| \nabla \Phi(u)\| )  .
\end{equation}
By combining  \eqref{E=0implyF=0} and \eqref{ineg EF}, we obtain that
$$\langle\nabla \cE(u,v) , \cF (u,v)\rangle_{\R^{2N}}\geq c'\| \nabla \cE (u,v) \|  \| \cF (u,v) \|.$$ 
This is condition \eqref{condAngle} with $\beta=1$. On the other hand, if $\nabla\cE(a,b)=0$ then by \eqref{E=0implyF=0} we have $b=0$ and $\nabla\Phi(a)=0$,  then $\cF(a,b)=0$, hence \eqref{cond1a}.

Now  if we assume that $\Phi$ is analytic \index{analytic}, then $\cE$ is also analytic \index{analytic} and satisfies {\L}ojasiewicz \index{Lojasiewicz} inequality \eqref{InegaliteLojasiewicz}. Therefore by applying lemma \ref{InegLojsABSRAITLemme} with $W=X=\R^{2N}$, $E=\cE$, ${\cal G}=\nabla\cE$ and $\Gamma=\omega(u,u')$ we obtain 
\begin{equation}\label{inegLoj2}\exists T>0,\ \exists c>0,  \exists \theta\in(0,\frac12]/\quad \Vert \nabla \cE(u(t),u'(t))\Vert \geq c \cE(u(t))^{1-\theta}.
\end{equation}
Then  hypothesis \eqref{cond1} is satisfied. On the other hand,  by using \eqref{ineg EF} we get 
$$\Vert \cF(u,v)\Vert=\Vert v\Vert+\Vert v+\nabla \Phi(u)\Vert\geq \frac12 (\Vert v\Vert+\Vert \nabla \Phi(u)\Vert)\geq \frac1C \| \nabla \cE (u,v) \|.$$
Combining this last inequality with \eqref{inegLoj2} we obtain that hypothesis \eqref{cond2} is satisfied.
Therefore by Theorem \ref{alain}, $\displaystyle\lim_{t\to\infty} (u(t),u'(t)) = (a,0)$ exists. We thus recover the result of Section \ref{CvAnalytiqueSndordre}. 

 In \cite{MR1616968}, also the case of nonlinear damping was considered. The damping, however, should not degenerate in the sense that near $0$ the damping is in principle linear. The case of degenerate damping which is the object of the next  section has been considered by L. Chergui in \cite{MR2429439}.

\subsection{A second order gradient like system with nonlinear dissipation}
Let $\Phi\in C^2 (\R^N,\R)$ and consider the second order ordinary differential system
\begin{equation} \label{EquationChergui}
 u'' (t) +  g(u' (t)) + \nabla \Phi (u(t)) = 0 ,
\end{equation}
where $g\in C (\R^N,\R^N)$ satisfying 
\begin{eqnarray}
\label{hypdamping1}\langle g(v),v\rangle&\geq& c \Vert v\Vert^{\alpha+2} \\
\label{hypdamping2}\Vert g(v)\Vert & \leq &  C\Vert v\Vert^{\alpha+1}
\end{eqnarray}
and $\alpha>0$.
\begin{thm} \label{chergui}   We suppose that
\begin{equation}\label{hypfond}
\exists\theta \in ] 0,\frac{1}{2}],\ \forall a\in S,\,\ \exists
\,\sigma _{a}>0\ /\Vert \nabla \Phi(u)\Vert \geq \vert
\Phi(u)-\Phi(a)\vert^{1-\theta },\ \forall u\in B(a,\sigma _{a}).
\end{equation}
Assume that $\alpha \in [0,\frac{\theta }{1-\theta }) $ and let $u\in
W^{2,\infty }(\R_{+},\R^{N})$ a solution of \eqref{EquationChergui}. Then \index{convergence result}
there exists $a\in S$ such that  $$\displaystyle
\lim_{t\rightarrow +\infty }(\Vert\dot{u}(t)\Vert+\Vert
u(t)-a\Vert) =0.$$
We also have
$$
\Vert u(t) - a \Vert =  O(t^{-\frac{\theta-\alpha(1-\theta)}{1-2\theta+\alpha(1-\theta)}} ) 
$$
\end{thm}  
\begin{proof}[{\bf Proof.}]First of all, we define the energy \index{energy} of this system 
$$E(t)=\frac12 \| u'(t)\|^2 + \Phi(u(t)).$$
We know that the function $E$ is nonincreasing and $E_\infty=\lim_{t\to\infty} E(t)$ exists.
 It is also well known (see corollary \ref{gradientsyst2Bis}) that $\omega(u,u')$ is compact connected subset of $(\nabla\Phi)^{-1}(\{0\})\times\{0\}.$
 In order to apply Theorem \ref{alain}, we must write equation \eqref{EquationChergui} as a first order
 system \eqref{non}. This is the case if we define $\cF :\R^{2N} \to \R^{2N}$ by
$$
\cF (u,v) := \left( \begin{array}{c} -v \\ g(v) +\nabla \Phi(u) \end{array} \right) , \quad u, \, v\in\R^N.
$$
Let $\varepsilon >0$, and define $\cE :\R^{2N} \to \R$ by
$$\cE(u,v)= \frac12 \| v\|^2 + \Phi(u)- E_\infty+ \varepsilon\Vert \nabla \Phi(u)\Vert^\alpha \langle \nabla \Phi(u) , v\rangle_{\R^N}  , \quad u , \, v\in\R^N $$
so that
$$
\nabla \cE (u,v) =   \left( \begin{array}{c} \nabla \Phi(u)+\varepsilon\Vert \nabla\phi(u)\Vert^\alpha\nabla^2\Phi(u)\cdot v+ 
\varepsilon\alpha \Vert \nabla\phi(u)\Vert^{\alpha-2}\langle\nabla\phi(u),v\rangle\nabla^2\Phi(u)\cdot \nabla\phi(u) \\ v+\varepsilon\Vert \nabla\phi(u)\Vert^\alpha\nabla\Phi(u) \end{array}\right).   
$$
Let $B\subset \R^N\times \R^N$ be a suffiently large closed ball which is a neighbourhood of the range of $(u,u')$, then we have
\begin{equation}\label{Fnulle}\Vert \cF (u,v) \Vert\leq C_1 ( \Vert v\Vert+ \Vert \nabla \Phi(u)\Vert) ;
\end{equation} 
$$\Vert \nabla\cE (u,v) \Vert\leq C_2 ( \Vert v\Vert+ \Vert \nabla \Phi(u)\Vert) .$$
Now choosing $\varepsilon \in(0,1)$ small enough and by using Young inequality together with hypotheses \eqref{hypdamping1} and \eqref{hypdamping2}, we get
\begin{equation} \label{Enulle}\langle\nabla\cE(u,v),\cF (u,v) \rangle\geq c_3 ( \Vert v\Vert^{\alpha+2} + \Vert \nabla \Phi(u)\Vert^{\alpha+2} )\geq c_4 ( \Vert v\Vert + \Vert \nabla \Phi(u)\Vert)^{\alpha+2} .
\end{equation}
Combining these three last inequalities we obtain 
\begin{equation} \label{condAnglebisbis}
\langle \nabla\cE(u,v) , \cF (u,v) \rangle \geq c_5 \,   \Vert \nabla\cE(u,v)\Vert^{\alpha+1} \, \Vert \cF (u,v)\Vert. 
\end{equation}
This is \eqref{condAngle} with $\beta=\alpha+1$. Since $\frac{d}{dt}[\cE(u(t),u'(t))]=-\langle\nabla\cE(u,v),\cF (u(t),u'(t)) \rangle\leq0$,  then  the function $t\longmapsto \cE(u(t),u'(t))$ is nonincreasing. Thanks to the fact that $u'\longrightarrow0$ as $t\to\infty$,  it follows that $\displaystyle\lim_{t\to+\infty} \cE(u(t),u'(t))=0$. Then $\cE$ satisfy hypothesis \eqref{energiepositive}.
Now if $\nabla\cE(a,b)=0$, then by \eqref{Enulle} $b=\nabla\Phi(a)=0$ which imply by \eqref{Fnulle} that $\cF(a,b)=0$. This is hypothesis \eqref{cond1a}.
On the other hand  by using Young inequality we get
$$\cE (u,v)^{1-\theta}\leq C_6 ( \Vert v\Vert+ \Vert \nabla \Phi(u)\Vert+\vert\Phi(u)-E_\infty\vert^{1-\theta}) .$$
We also have
$$\Vert \cF (u,v) \Vert\geq c_7 ( \Vert v\Vert+ \Vert \nabla \Phi(u)\Vert).$$
Combining this two last inequalities together with the {\L}ojasiewicz \index{Lojasiewicz} inequality \eqref{hypfond}, we get
\begin{equation*}\Vert  \cF(u(t),u'(t))\Vert \geq c' \cE(u(t))^{1-\theta}.
\end{equation*}
This is \eqref{cond2}.
Since $\alpha \in [0,\frac{\theta }{1-\theta }[ $ then $\beta(1-\theta)=(\alpha+1)(1-\theta)<1$, then \eqref{Cdbetatheta} is satisfied. Theorem \ref{chergui} is proved.
\end{proof}

\chapter[The infinite dimensional case]{The infinite dimensional case} In \cite{{MR0727703}}, L. Simon completed the fundamental one dimensional result of Zelenyak \cite{MR0223758} and Matano\cite{MR0501842} by showing that the pioneering work of S. {\L}ojasiewicz can be extended to some infinite dimensional context, among which the semi-linear parabolic equations with analytic generating function in any space dimension. The objective  of this chapter is  to clarify  to which extent the  {\L}ojasiewicz method can be generalized to infinite dimensional systems. Throughout this chapter,  we consider two real Hilbert spaces $V, H$  where  $V\subset H$ with continuous and dense imbedding and $H'$, the topological dual of $H$ is identified with $H$, therefore $$ V\subset H = H' \subset  V'$$ with continuous and dense imbeddings.
\begin{defn} \label{loja-Simon}
We say that the function $E\in C^1(V,\R)$ satisfies the   {\L}ojasiewicz \index{Lojasiewicz} gradient inequality \index{Lojasiewicz} near some point $\varphi\in V$, if there exist constants $\theta\in (0,\frac{1}{2}]$, $c\geq 0$ and $\sigma >0$ such that for all $u\in V$ with $\| u-\varphi\|_V\leq \sigma$
\begin{equation} \label{eqloja-Simon} \|DE(u)\|_{V'} \geq c|E(u) - E(\varphi )|^{1-\theta}.
\end{equation}
\end{defn}
\begin{rem}\label{PhiNestpasCritiqueTriviale}{\rm  1) The {\L}ojasiewicz\index{Lojasiewicz} gradient inequality  is trivial if $\varphi$  is not a critical point of $E$.\\
2) The number  $\theta$ will be called a {\L}ojasiewicz\index{Lojasiewicz} exponent (of $E$ at $\varphi$).\\
}
\end{rem}

 \section{Analytic  \index{analytic} functions and the {\L}ojasiewicz \index{Lojasiewicz} gradient inequality}
One might wonder if {\L}ojasiewicz \index{Lojasiewicz} gradient inequality is valid for any analytic \index{analytic}function  on an infinite dimensional Banach space. However, even if $V = H$ it is not the case. Actually, if $(H,\langle\cdot,\cdot\rangle)$ is a Hilbert space and $F$ is defined by $F(u)=\langle Ku,u\rangle$ with $K=K^*\geq 0$ and compact, then $F$ does not satisfy the {\L}ojasiewicz \index{Lojasiewicz} gradient inequality. More precisely
\begin{prop} Let $H = l^2(\N)$ and $F: H \rightarrow \R$ be
the continous quadratic (hence analytic \index{analytic})  functional given by
$$ F(u_0, u_1,... u_n,...) : = \sum_{j=0}^{\infty}\varepsilon_j u_j^2$$
where $(\varepsilon_k)_{k\in\N}$ is a real sequence satisfying  $\varepsilon_k>0$ and $   \displaystyle \lim_{k\to\infty}\varepsilon_k=0$.
Then $F$ satisfies no {\L}ojasiewicz \index{Lojasiewicz} gradient inequality.
\end{prop}
\begin{proof}[{\bf Proof.}]  Defining $ (e_i)_j = \delta_{ij}$, an immediate calculation
shows that 
$$ \forall t>0, \quad F(te_k) =t^2\varepsilon_k ; \quad \vert \nabla F(te_k) \vert =2t\varepsilon_k.$$
 In particular for each $\theta>0$ we have
$$\displaystyle  \frac{\vert \nabla F(te_k) \vert }{\vert F(te_k) \vert^{1-\theta}}=2\varepsilon_k^{\theta}  t^{2\theta-1}$$
For any $\theta>0$ small , choosing  $t$ small enough and letting $k$ tend to infinity we can see that the {\L}ojasiewicz \index{Lojasiewicz} gradient inequality fails in the ball of radius $t$.\end{proof} 
More generally, in  \cite{MR2772353}, we  considered a real Hilbert space $(H,\langle\cdot,\cdot\rangle)$, a linear operator $A$ such that
\begin{equation*} A \in L(H); \quad A^* = A
\end{equation*}
and the associated quadratic form
$\Phi:H\longrightarrow\R$ defined by
\begin{equation*} \forall u\in H,\quad \Phi(u)= \frac{1}{2}\langle Au, u \rangle.\end{equation*} In this context a characterization of continous quadratic forms for which the  {\L}ojasiewicz \index{Lojasiewicz} gradient inequality is valid was obtained and expressed  by the following statement
\begin{thm}\label{quad.} The following properties are equivalent\\
i)  $ 0$ is not an accumulation point of sp(A).\\
ii) For some $\rho>0$ we have  $$ \forall u\in \ker(A)^{\perp},  \quad \Vert Au  \Vert_H  \ge \rho \Vert u\Vert_H.$$
iii) $\Phi$ satisfies the {\L}ojasiewicz \index{Lojasiewicz} gradient inequality at the origin for some $\theta>0$.\\
iv) $\Phi$ satisfies the {\L}ojasiewicz \index{Lojasiewicz} gradient inequality at any point for  $\theta = \frac{1}{2}$.\\
\end{thm}

For a general nonlinear potential $F$, one might wander if the equivalent properties above for $A = D^2 F(a)$ are sufficient to obtain a $\L$ojasiewicz gradient inequality near $a$. The proposition  below shows that it is not the case.

\begin{prop} Let $H = l^2(\N)$ and $F: H \rightarrow \R$ be
the analytic \index{analytic} functional given by
$$ F(u_1, u_2,... u_n,...) : = \sum _{k = 2}^{\infty}\frac{\vert u_k\vert ^{2k+2}}{(2k+2)!}.$$
Then $F$ satisfies no {\L}ojasiewicz \index{Lojasiewicz} gradient inequality.
\end{prop}
\begin{proof}[{\bf Proof.}] First we note that $D^2 F(0) = 0 $, hence $ sp(D^2 F(0)) = \{0\} $ and in particular $0$ is isolated in $ sp(D^2 F(0)) $. Defining $ (e_i)_j = \delta_{ij}$, an immediate calculation
shows that $$ \forall t>0, \quad F(te_k) =\frac{t ^{2k+2}}{(2k+2)!} ;
\quad \vert \nabla F(te_k)
\vert =\frac{t ^{2k+1}}{(2k+1)!}.$$ In particular for each $\theta>0$ we have
$$ \frac {F(te_k) ^{1-\theta}}{\vert \nabla F(te_k)
\vert }= c(\theta, k) t^{1- (2k+2)\theta}.$$ 
Choosing k large enough gives a contradiction for $t$
small.\end{proof}

In this example, the difficulty comes from the fact that $ \dim \ker(D^2 F(0))= \infty$.  Assuming the equivalent properties  of Theorem \ref {quad.} and $ \dim \ker(D^2 F(0))\not = \infty$ is  equivalent to the semi-Fredholm\index{semi-Fredholm} character of 
$ D^2 F(0)$ (cf. Theorem \ref {PropInegFredholm}) In the next section  \ref{abstract Lojasiewicz} we shall show that this condition is sufficient in a rather general framework, in particular $V$ will not be assumed equal to $H$ in view of applications to semilinear PDE . 

\section{An abstract {\L}ojasiewicz gradient inequality}\label{abstract Lojasiewicz}

\noindent 

The purpose of this section is to give sufficient conditions on $E$ for  the inequality \eqref{eqloja-Simon} to be satisfied.
Let $E\in C^2(V,\R)$ and $\varphi\in V$ such that $DE(\varphi) = 0.$ Up to the change of variable $u = \varphi+v$ and the change of function $G(v)=E (\varphi+v)-E(\varphi)$, we can assume whithout loss of generality that $\varphi=0$, $E(0)=0$ and $DE(0)=0$.
 Although the formulation of the {\L}ojasiewicz gradient inequality \index{{\L}ojasiewicz gradient inequality} requires only $E\in C^1(V,\R)$, one way of proving it requires $E\in C^2(V,\R)$. In fact the operator $A:=D^2E(0)$ plays an important role.
 \\ 
 \medskip 
 
We start with the following  very simple result
\begin{prop}\label{generalNoyauNul.} Assume that 
$$A \in L(V,V') \; \hbox{ is an isomporphism}.$$
Then the {\L}ojasiewicz gradient inequality \index{{\L}ojasiewicz gradient inequality} is satisfied near 0 with the exponent $\theta=\frac12$ :  there exist two positive constants $\sigma>0$  and $c>0$ such that
$$\Vert u\Vert_V<\sigma\Longrightarrow \Vert D E(u)\Vert_{V '}\geq c\vert E(u)\vert ^{\frac12}.$$
\end{prop}
\begin{proof}[{\bf Proof.}] It is easy to see, using Taylor's expansion formula, that for
$\Vert u\Vert_V$ small enough we have
\begin{equation}\label{InegF}\vert E(u ) \vert \leq C \Vert u\Vert_V^2. 
\end{equation}
On the other hand, since $DE(u) = A u + o(u)$, we have
$$u = A^{-1}DE (u) + o(u),$$
and therefore for any given $\varepsilon > 0$ we can find $\delta (\varepsilon
) >0$ such that if $\Vert u\Vert_V \leq
\delta (\varepsilon )$ then
$$\Vert DE(u)\Vert_{V'} \geq  \Vert A^{-1}\Vert^{-1}\Vert u\Vert_V - \epsilon \Vert u\Vert_V.$$
Choosing $\varepsilon := \varepsilon_0:= \Vert A^{-1}\Vert^{-1}/2$, we obtain for
$\Vert u\Vert_V \leq \delta (\varepsilon_0)$
\begin{equation}\label{InegGrad F}\Vert DE(u )\Vert_{V'} \geq \varepsilon_0 \Vert u\Vert_V.
\end{equation}
The result follows by combining \eqref{InegF} and \eqref{InegGrad F}. 
\end{proof}
\begin{rem}{\rm Since $A=D^2E(0)$ is symmetric, then if $A$ is semi-Fredholm\index{semi-Fredholm} and $d=\dim \ker(A)=0$, by corollary \ref{SFredplusProj} $A$ is an isomorphism. Hereinafter we assume that $d>0$. We denote by $\Pi:V\longrightarrow \ker(A)$  the projection in the sense of $H$.
}\end{rem}

\begin{prop} Assume that $A:=D^2E(0)$ is a semi-Fredholm\index{semi-Fredholm} operator and let 
\begin{eqnarray*}{\cal N}:V &\longrightarrow& V' \\
u &\longmapsto& \Pi u+DE(u).
\end{eqnarray*}
Then there exist a neighborhood of $0$, $W_1(0)$ in $V$, a neighborhood of $0,\ W_2(0)$ in $V'$ and a $C^1$ map $ \Psi : W_2(0) \longrightarrow W_1(0) $ which satisfies
$$ {\cal N} (\Psi (f)) =f\qquad \forall f\in W_2(0),$$
$$\Psi ({\cal N} ( u)) =u\qquad\forall u\in W_1(0),$$
\begin{equation}\label{Eq(2.6)} \Vert\Psi (f) - \Psi (g)\Vert_V\leq C_1
\Vert f - g \Vert_{V'}
\quad \forall f,g \in W_2(0),\quad C_1 >0 .
\end{equation}

\end{prop}
\begin{proof}[{\bf Proof.}]
The function ${\cal N} $ is $C^1$  and $D{\cal N}(0)=\Pi +D^2E(0)$ which by corollary \ref{SFredplusProj} is an isomorphism from $V$ to $V'$. We have just to apply the local inversion theorem.
\end{proof}
Let $(\varphi_1,\varphi_2,...\varphi_d)$ denote an orthonormal basis of $\ker (A)$ relatively to the inner product of $H$. For  $\xi \in {\R}^d$ small enough to achieve
$\displaystyle\sum_{j=1}^d \xi_j \varphi _j \in W_2(0)$, we define the map $\Gamma$ by
\begin{equation}\label{defFtGamma}\Gamma (\xi )= E(\Psi(\displaystyle\sum_{j=1}^d \xi_j \varphi _j)).\end{equation}
Let $\widetilde{W_2}(0)$ be the open neighborhood of 0 in $\R^d$ such that
$$\xi\in \widetilde{W_2}(0)\Longleftrightarrow \sum_{j=1}^d \xi_j\varphi_j\in W_2(0).$$
The function $\Gamma$ is $C^1$ in $\widetilde{W_2}(0)$. Let us define also
$$\widetilde{W_1}(0)=\{u\in W_1(0)/\Pi (u)\in W_2(0)\}.$$

\begin{prop}\label{Inegaliteimportante} Let $u\in \widetilde{W_1}(0)$ and let $\xi\in \widetilde{W_2}(0)$ such that $\Pi ( u ) = \displaystyle\sum_{j=1}^d \xi_j \varphi _j\in  W_2(0)$. Then there are two constants $C,K>0$ such that
\begin{equation}\label{Eq(2.9)} \Vert \nabla \Gamma (\xi ) \Vert _{\R^d}\leq C \Vert{DE}(u) \Vert_{V'},
\end{equation}
\begin{equation}\label{Eq(2.10)}\ \vert  E(u) - \Gamma(\xi)\vert \ \leq K \Vert{DE}(u)\Vert_{V'}^2.
\end{equation}
\end{prop}
\begin{proof}[{\bf Proof.}]
For any $ k \in \{1, \cdots d \} $ we have the formula
\begin{equation}\label{Equati2.4}{\partial{\Gamma}\over {\partial \xi_k}} = {d\over{ds}}
E(\Psi[\displaystyle\sum_{j\neq k}^{}\xi_j \varphi _j +( \xi_k + s)\varphi _k]) 
\vert _{s= 0} = \langle  DE (\Psi(\displaystyle\sum_{j=1}^d \xi_j \varphi_j )), D\Psi(\displaystyle\sum_{j=1}^d \xi_j \varphi_j )\varphi_k \rangle .\end{equation}
  Now we claim that for all $\xi\in\widetilde{W}_2(0)$
\begin{equation}\label{Equati2.5}\Vert\sum_{k=1}^d{{\partial{\Gamma}}\over{\partial
\xi_k}}(\xi)\varphi_k-DE
(\Psi(\displaystyle\sum_{j=1}^d \xi_j \varphi_j
))\Vert_{V'}\leq C_2\vert\xi\vert \Vert DE
(\Psi(\displaystyle\sum_{j=1}^d \xi_j \varphi_j
))\Vert_{V'}.\end{equation}
In fact by using \eqref{Equati2.4}, remarking that $DE (\Psi(\displaystyle\sum_{j=1}^d \xi_j \varphi_j ))\in\ker A$
we obtain
\begin{eqnarray*}&&\Vert\sum_{k=1}^d{{\partial{\Gamma}}\over{\partial \xi_k}}(\xi)\varphi_k-DE
(\Psi(\displaystyle\sum_{j=1}^d \xi_j \varphi_j
))\Vert_{V'}\\
&=&\Vert\sum_{k=1}^d  <DE (\Psi(\displaystyle\sum_{j=1}^d \xi_j \varphi_j )), D\Psi(\displaystyle\sum_{j=1}^d \xi_j \varphi_j ) (\varphi_k)-\varphi_k > \varphi_k \Vert_{V'}.
\end{eqnarray*}
Now by using Cauchy-Schwarz inequality, and the fact that $D\Psi(0)({\cal L}u)=u$, the claim follows. On the other hand, since $E$ is $C^1$, there exists $C_3$ such that
\begin{equation}\label{FtC1}\Vert DE(u)-DE(v) \Vert_{V'} \leq C_3\Vert u-v\Vert_V\quad \forall
(u,v)\in W_1(0).
\end{equation}
 Then by using \eqref{Eq(2.6)},  \eqref{Equati2.5}  and \eqref{FtC1} we obtain
\begin{eqnarray*} \Vert \nabla \Gamma (\xi ) \Vert _{\R^d}&\leq&
C_4 \Vert DE (\Psi(\displaystyle\sum_{j=1}^d \xi_j \varphi_j ))\Vert_{V'} \\
& =& C_4 \Vert DE(\Psi(\Pi (u )))\Vert_{V'} \\
& =& C_4\Vert{DE} (\Psi(\Pi (u ))) - {DE}(u)+{DE}(u)\Vert_{V'}\\
& \leq&  C_4\Vert{DE}(u)\Vert_{V'}+ C_3C_4\Vert\Psi(\Pi (u))-u\Vert_V\\
& = &C_4\Vert{DE}(u)\Vert_{V'} + C_3C_4\Vert\Psi(\Pi(u)) -
\Psi(\Pi u+{DE}(u) )\Vert_V \\
& \leq&  C_4\Vert{DE}(u)\Vert_{V'}+C_5\Vert{DE}(u)\Vert_{V'}
\end{eqnarray*}
hence \eqref{Eq(2.9)}.
 On the other hand
\begin{eqnarray*}  \vert  E(u) - \Gamma(\xi)\vert
&= &\vert E(u)- E(\Psi(\Pi (u ))) \vert \\
&=&\vert \int_0^1 {d\over {dt}}[ E(u+t( \Psi(\Pi
(u ))-u) ]\, dt\,\vert\\
&=&\vert \int_0^1 ({DE} (u+t( \Psi(\Pi(u ))-u)),\Psi(\Pi(u))-u)\,dt\,\vert \\
&\leq&  \Vert\Psi(\Pi(u ))-u \Vert_V  \int_0^1 \Vert{DE} (u+t( \Psi(\Pi(u ))-u)\Vert _{V'}\,dt\\
&\leq& \lbrack\int_0^1 (\Vert{DE}(u)\Vert_{V'}+t\,C_3\Vert\Psi(\Pi(u))-u\Vert_V)\,dt\, \rbrack\,\Vert\Psi(\Pi(u ))-u \Vert_V \\
&\leq& C_6\Vert{DE}(u)\Vert_{V'}\Vert\Psi(\Pi(u ))-\Psi(\Pi(u
)+{DE}(u) ) \Vert_V\\
&\leq& C_1C_6\Vert{DE}(u)\Vert^2_{V'}
\end{eqnarray*}
hence \eqref{Eq(2.10)}.
\end{proof}

\begin{thm}\label{generalEgaliteDimension} Assume that $A:=D^2E(0)$ is a semi-Fredholm\index{semi-Fredholm} operator and let $d=\hbox{\rm dim}\ker A$. Assume moreover that \\
\textbf{(H1)} $d>0$ and there exists $O\subset\R^d$ open, and $h\in C^1(O,V)$ such that $0\in
h(O)\subset (DE)^{-1}(0)$ and $h:O\longrightarrow h(O)$ is a diffeomorphism.\\
Then there  exist two positive constants  $\sigma>0$  and $c>0$ such that
$$\Vert u\Vert_V<\sigma\Longrightarrow \Vert D E(u)\Vert_{V '}\geq c\vert E(u)\vert^{\frac12}.$$
\end{thm}
\begin{proof}[{\bf Proof.}] 
We  have by using \eqref{Equati2.5} (choosing a smaller $\widetilde{W}_2(0)$ if
necessary)
\begin{equation}\label{Equati(2.7)}\Vert DE(\Psi(\displaystyle\sum_{j=1}^d \xi_j
\varphi_j))\Vert_{V'}\leq C_{7} \Vert
\nabla \Gamma (\xi)\Vert.\end{equation}
If $u\in
\widetilde{W}_1(0)$ such that
$DE(u)=0$, then ${\cal N}(u)=\Pi(u)$ which implies that $
u=\Psi(\Pi(u)).$
Moreover by using \eqref{Eq(2.9)} we have $\nabla \Gamma(\xi)=0$ where
$\xi\in\widetilde{W}_2(0)$ with $\Pi u=\displaystyle\sum_{j=1}^d
\xi_j \varphi_j$.\\
On the other hand let $\xi\in\widetilde{W}_2(0)$ with $\nabla \Gamma(\xi)=0$. Then $\Psi(\displaystyle\sum_{j=1}^d
\xi_j \varphi_j)\in W_1(0)$ and $DE(\Psi(\displaystyle\sum_{j=1}^d \xi_j \varphi_j))=0$ by using \eqref{Equati(2.7)}.
So $\Pi(\Psi(\displaystyle\sum_{j=1}^d \xi_j \varphi_j))=\displaystyle\sum_{j=1}^d
\xi_j \varphi_j$.
Consequently $\Psi(\displaystyle\sum_{j=1}^d\xi_j \varphi_j)\in \widetilde{W}_1(0)$ and $DE(\Psi(\displaystyle\sum_{j=1}^d \xi_j \varphi_j))=0$.\\
Finally we have: 
\begin{equation}\label{Equati(2.8)}\{u\in\widetilde{W}_1(0),\ DE(u)=0\}=\Psi(\{\displaystyle\sum_{j=1}^d\xi_j \varphi_j,\ \xi\in\widetilde{W}_2(0)\hbox{ and }\nabla \Gamma(\xi)=0\}).\end{equation}
Now we introduce the $d-$dimensional manifold $$ \gamma = h(O)$$ with $O$ and
$h$ as in \textbf{(H1)}. Let $$\widetilde{O}= h^{-1}(\{u\in\widetilde{W}_1(0),\ DE( u)=0\}). $$
Clearly $\widetilde{O}$ is an open subset of $\R^d$ and $ 0\in h(\widetilde{O}) $.\\
We now have $$ \widetilde{\gamma}:= h(\widetilde{O})
\subset\{u\in\widetilde{W}_1(0),\ DE(u)=0\}\subset\Psi(\{\displaystyle\sum_{j=1}^d\xi_j  \varphi_j,\ \xi\in\widetilde{W}_2(0)\}).$$
Since the extreme terms are $d-$dimensional open manifolds, they must coincide locally. Therefore, changing if necessary $\widetilde{W}_1(0)$ and $\widetilde{W}_2(0)$) to smaller open sets, we obtain
\begin{equation}\label{Equati(2.9)}\widetilde{\gamma} =\{u\in\widetilde{W}_1(0),\ DE( u)=0\}=\Psi(\{\displaystyle\sum_{j=1}^d\xi_j \varphi_j,\ \xi\in\widetilde{W}_2(0)\}.
\end{equation}
Now by comparing \eqref{Equati(2.8)} and \eqref{Equati(2.9)}, we get
$$\Gamma(\xi)=0, \qquad \forall\xi\in\widetilde{W}_2(0).$$
The proof of Theorem \ref{generalEgaliteDimension} follows immediately by using this last equality in \eqref{Eq(2.10)}.
 \end{proof}
 In the next theorem, we will prove inequality like \eqref{eqloja-Simon} under hypotheses of analyticity  \index{analytic} of $E$ and $DE$.
We consider a Banach space $Z$ such that $\ker A\subset Z$ and $Z\subset H$ with continuous and dense imbedding.
\begin{prop}\label{L bijectiveWZ} Assume that $A:=D^2E(0)$ is a semi-Fredholm\index{semi-Fredholm} operator. Let  ${\cal L}:=\Pi +A$. Then $W:={\cal L}^{-1}(Z)$ is a Banach space with respect to $\Vert w \Vert_W=\Vert {\cal L}w\Vert_Z$ and ${\cal L}\in L(W,Z)$ is  an isomporphism.
\end{prop}
\begin{proof}[{\bf Proof.}] Using corollary \ref{SFredplusProj}, we know that ${\cal L}:V\longrightarrow V'$ is one to one and onto. Since $W\subset V$ and by the  definition of $W$ we also have ${\cal L}:W\longrightarrow Z$ is one to one   and onto.  
Obviously   we have   ${\cal L} \in L(W, Z)$ because $\Vert {\cal L} u\Vert_Z=\Vert u\Vert_W$ for all $u\in W$.
Now we prove that $W$ is a Banach space. Let $(w_n)$ be a Cauchy sequence in $W$, then $({\cal L}(w_n))$ is a Cauchy sequence in the Banach space $Z$. Denote by $z$ its limit.
$({\cal L}(w_n))$ is also a Cauchy sequence in $V'$, so $(w_n))$ is also a Cauchy sequence in $V$.  Denote by $w$ its limit, since ${\cal L} \in L(V, V')$, then ${\cal L}w=z$. The claim is proved.
Banach's theorem gives the fact that ${\cal L}^{-1}\in L(Z, W) .$
\end{proof}
 \begin{thm}\label{generalAnalytique.} Assume that $A:=D^2E(0)$ is a semi-Fredholm\index{semi-Fredholm} operator and that  $N:=\ker A\subset Z$.  Assume moreover that  :\\
\textbf{(H2)}  $E: U \rightarrow \R$ is analytic \index{analytic} in the sense of  definition \ref{defFtAnalytique} where
$ U\subset  W $  is an open neighborhood of $0$, that $DE(U)\subset Z$ and  $DE:U\longrightarrow  Z$ is analytic\index{analytic}.\\
Then there exists $\theta \in (0, 1/2]$, $\sigma>0$  and $c>0$ such that
$$\Vert u\Vert_V<\sigma\Longrightarrow \Vert D E(u)\Vert_{V '}\geq c\vert E(u)\vert ^{1-\theta}.$$
\end{thm}
\begin{proof}[{\bf Proof.}] For the proof we need the following result. \begin{lem} Then there exist a neighborhood of $0$, $V_1(0)$ in $W$, a neighborhood of $0,\ V_2(0)$ in $Z$ and an analytic \index{analytic} map $ \Psi_1 : V_2(0) \longrightarrow V_1(0) $ which satisfies
$$ {\cal N} (\Psi_1 (f)) =f\qquad \forall f\in V_2(0),$$
$$\Psi_1 ({\cal N} ( u)) =u\qquad\forall u\in V_1(0),$$
$$\Psi_1=\Psi \qquad \hbox{in }\ V_2(0)\cap W_2(0)$$
\begin{equation}\label{TILVV'}\Vert \Psi(f)-\Psi(g)\Vert_W\leq C_1'\Vert f-g\Vert_{Z}\quad
\forall(f,g)\in V_2(0)\cap W_2(0),
\end{equation}
\end{lem}
\begin{proof}[{\bf Proof.}]  
We first establish that 
 \begin{eqnarray*}{\cal N}:W &\longrightarrow& Z \\
u &\longmapsto& \Pi u+DE(u).
\end{eqnarray*}
 is a $C^1$ diffeomorphism near 0, because $D {\cal N}(0) = \Pi+A={\cal L}\in L(W,Z) $ is an isomorphism (see proposition \ref{L bijectiveWZ}) and the classical local inversion theorem applies. Therefore we can find a neighborhood  $V_1(0)$ of $0$ in $W$ and  a neighborhood $V_2(0)$ of $0$ in $Z$ such that ${\cal N}:V_1(0) \longrightarrow V_2(0)$ is a
$C^1$ diffeomorphism.  Finally it is clear that $\Psi_1 = {\cal N }^{-1} $ in $V_2(0)\cap
W_2(0)$. By Theorem \ref{InvertAnal}  we have $\Psi_1$ is analytic \index{analytic} in $V_2(0)$.
\end{proof}
\noindent {\bf End of proof of Theorem } By using the chain rule (Theorem \ref{ComposAnal}), since $ E:U {\longrightarrow \R}$, $DE:U\longrightarrow  Z$ and $\Psi:V_2(0)\cap W_2(0)\longrightarrow V_1(0)$ are analytic \index{analytic}, the function $\Gamma$ defined in \eqref{defFtGamma} is real analytic \index{analytic} in some neighborhood of $0$ in ${\R}^d$.\\
Applying the classical {\L}ojasiewicz \index{Lojasiewicz} inequality (Theorem \ref{ThmInegaliteLojasiewicz}) to the scalar analytic \index{analytic}  function $ \Gamma $ defined on some neighborhhod of $0$ in $\R^d$ by the formula \eqref{defFtGamma}, we now obtain (since $(1-\theta)\in(0,1))$:
\begin{equation}\label{Eq(2.11)} \vert E(u)\vert^{1-\theta} \leq \vert\Gamma(\xi)\vert^{1-\theta} + \vert\Gamma(\xi)-
E(u)\vert^{1-\theta}\leq \frac{1}{C_0}\Vert \nabla \Gamma (\xi ) \Vert _{\R^d}+ \vert\Gamma(\xi)-E(u)\vert^{1-\theta}.\end{equation}
 By combining \eqref{Eq(2.9)}, \eqref{Eq(2.10)}, \eqref{Eq(2.11)} we obtain
$$\vert E(u)\vert^{1-\theta}\leq \frac{C}{C_0}\Vert{DE}(u)\Vert_{V'}+ K^{1-\theta}\Vert{DE}(u)\Vert_{V'}^{2({1-\theta})}.$$
 Then since $ 2({1-\theta}) \geq 1, $ there exist
$\sigma>0$,
$c>0$ such that
$$\Vert {DE}(u)\Vert_{V'}\geq c \vert E(u)\vert^{1-\theta}\quad\hbox{ for all }  u\in V 
\hbox { such   that } \Vert u\Vert_V<\sigma.$$ Theorem \ref{generalAnalytique.} is proved.
\end{proof}

\section{Two abstract convergence results} This section is exceptionnally devoted to an abstract situation in which a trajectory of some evolution equation is known independently of any well-posedness result for the corresponding initial value problem. In particular there is no underlying continuous semi-group to rely on and we cannot apply directly the simple results of chapters 4 and 6. However, by performing essentially the same kind of calculations as those needed to apply the invariance principle, we end up with a ``gradient-like" property which is the starting point for the {\L}ojasiewicz method to be applicable. Our results contain as special cases the semi-linear examples of section 10.4 (for which the semi-group framework could be applied as an alternative method) but they can also be used for strongly non-linear problems as soon as a solution with the right regularity properties is known, even if the well-posedness is either false or presently out of reach. \\

Let $V$ and $H$ be two Hilbert spaces such that $V$ is a dense subspace of $H$ and  the imbedding of $V$ in $H$ is {\bf compact}. We identify $H$ with its topological dual and we denote by $V'$ the dual of $V$, so that $H \subset V'$ with
continuous imbedding.\\

Let $E\in C^1(V,\R)$. We study the following two abstract evolution equations: the first order equation
\begin{equation} \label{equPremOrdre}
u' (t) + \nabla E  (u(t)) = 0 ,  t\geq 0 \end{equation}
and the second order equation
\begin{equation} \label{secorder}
u'' (t) + u' (t) + \nabla E  (u(t)) =0,  t\geq 0 \end{equation}
 \begin{thm}\label{ThmCvSdOrderAbstr} Let $u\in C^1(\R_+, V)$ be a solution of \eqref{equPremOrdre}, and assume that
\par (i)  $\displaystyle \overline{\cup_{t\geq 1} \{ u(t) \}}$ is compact in $V$;

\par (ii) $E$ satisfies the {\L}ojasiewicz gradient inequality \index{{\L}ojasiewicz gradient inequality}  near every point $\varphi\in {\cal S}:=\{\varphi\in V,\ \nabla E(\varphi)=0\}.$\\
Then there exists $\varphi\in {\cal S}$  such that \index{convergence result}
$$\lim_{t\to+\infty}^{}\Vert u(t)-\varphi\Vert_V= 0.$$
Moreover, let $\theta$ be any {\L}ojasiewicz\index{Lojasiewicz} exponent of $E$ at $\varphi$. Then we have
\begin{equation}\label{DecayEstimaFirstOrderAbst}
\Vert u(t) - \varphi \Vert_H= \left\{ \begin{array}{ll}
O(e^{-\delta t} )  & \hbox{ for some }\delta>0 \text{ if } \theta =\frac12, \\[2mm]
O(t^{-\theta/(1-2\theta)} )  & \text{if } 0<\theta<{1\over2}.
\end{array} \right.
\end{equation}\end{thm}
\begin{proof} [Proof.]  We define the function $z$ by $z(t):=E (u(t))$ for all $t\geq 0$. Since $u\in C^1(\R_+, V)$ and $E\in C^1(V,\R)$, then by chain rule, $z$ is differentiable and
\begin{equation}\label{inegaliteenergieE} z'(t) =-\Vert u' (t)\Vert_H^2,\quad \forall t\geq0.
\end{equation} 
Integrating this last equation and by using $(i)$, we get $u'\in L^2 (\R_+ ;H)$. Now, since the range of $u$ is precompact in $V$, and $u$ is uniformly Holder  continuous on the half-line with values in $H$, it is also uniformly continuous with values in $V$ and   $u' = -\nabla E  (u(t))$  is uniformly continuous with values in $V'$. Then by applying Lemma \ref{L1-conv} to the numerical function $\Vert u' (t)\Vert_{V'}^2$ , we obtain that $u'(t)$ tends to $0$ in $V'$ as $t$ tends to infinity, hence also in $H$  by compactness. We conclude that $\omega(u_0)\subset {\cal S}$.
Moreover, since the function $z$ is bounded and decreasing, the limit  $K:=\displaystyle\lim_{t\to\infty} E(u(t))$ exists.  Replacing $E$ by $E-K$ we may assume $K = 0$.\\
 
 If $z(t_0)=0$ for some $t_0\geq0$, then $z(t)=0$ for every $t\geq t_0$, and therefore, $u$ is constant for $t\geq t_0$. In this case, there remains nothing to prove. Then we can assume that $z(t)>0$ for all $t\geq 0$.
Define $\Gamma:=\omega(u_0)$. It is clear that $\Gamma$ is compact and connected. Let  $\varphi\in \Gamma$, then there exists $t_n\to+\infty$ such that $\Vert u(t_n)- \varphi\Vert_V\longrightarrow 0$. Then we get
$$\lim_{n\to+\infty}^{}E(u(t_n))=E(\varphi)=K=0.$$
On the other hand, by assumption $(ii)$, $E$ satisfies the {\L}ojasiewicz \index{Lojasiewicz} gradient  inequality \eqref{eqloja-Simon} at every point $\varphi\in {\cal S}$.  Applying Lemma \ref{InegLojsABSRAITLemme} with $W=V$, $X=V'$,  and ${\cal G}=\nabla E$ we obtain,
$$ \exists \sigma, c>0,\ \exists \theta\in(0,\frac12]/\ \left[\hbox{dist}(u,\Gamma)\leq \sigma\Longrightarrow\Vert \nabla E(u)\Vert_{V'}\geq c\vert E(u)\vert^{1-\theta}\right].$$
Now since $\Gamma=\omega(u_0)$, by Theorem \ref{thm 4.1.8.} iii),  there exists $T>0$ such that $\hbox{dist}(u,\Gamma)\leq \sigma$ for all $t\geq T$. Then we get  
\begin{equation}\label{ineqLojSimTraj}
\forall t\geq T\quad\Vert \nabla E(u)\Vert_{V'}\geq c\vert E(u)\vert^{1-\theta}.
\end{equation}
By combining \eqref{inegaliteenergieE} and \eqref{ineqLojSimTraj}, we get
\begin{equation}\label{inequationdifferentiellNew}z'(t) \le -c^2 (z(t))^{2 (1-\theta)  },\quad \forall t\geq T.
\end{equation} The end of the proof is identical to that of Theorem \ref{ThmCVAnalyticSGDimFini}, we obtain the convergence of $u(t)$ in $H$ and the convergence in $V$ follows by compactness \end{proof}

\begin{thm}\label{CvAbstraitSdOrder} Let $u\in C^1(\R_+,V)\cap C^2(\R_+,V')$ be a solution of \eqref{secorder} and assume that
\par (i) $\displaystyle \overline{\cup_{t\geq 1} \{ u(t),u'(t) \}}$ is compact in $V\times H$;
\par (ii)   if $K:V'\to V$ denotes the duality map, then the operator $K\circ E''(v)\in\cL (V)$ extends to a bounded linear operator on $H$ for every $v\in V$, and $K\circ E'' :V \to \cL (H)$ maps bounded sets into bounded sets; 
\par (iii) $E$ satisfies the {\L}ojasiewicz gradient inequality \index{{\L}ojasiewicz gradient inequality}  near every point $\varphi\in {\cal S}:=\{\varphi\in V,\ \nabla E(\varphi)=0\}.$\\
Then there exists $\varphi\in {\cal S}$  such that
$$\lim_{t\to+\infty}^{}\Vert u'\Vert_H+\Vert u(t)-\varphi\Vert_V= 0.$$
Moreover, let $\theta$ be any {\L}ojasiewicz\index{Lojasiewicz} exponent of $E$ at $\varphi$. Then we have \index{convergence result}
\begin{equation}\label{DecayEstimaFirstOrderAbst}
\Vert u(t) - \varphi \Vert_H= \left\{ \begin{array}{ll}
O(e^{-\delta t} )  & \hbox{ for some }\delta>0 \text{ if } \theta =\frac12, \\[2mm]
O(t^{-\theta/(1-2\theta)} )  & \text{if } 0<\theta<{1\over2}.
\end{array} \right.
\end{equation}\end{thm}
\begin{proof} [Proof.] Let
$$\cE(t):={1\over2}\Vert u'(t)\Vert_H^2+E(u(t)).$$
By the assumptions on $u$ and $E$, $\cE$ is differentiable everywhere and for all $t>0$
$$\cE'(t)=-\Vert u'(t)\Vert_H^2.$$
Hence $\cE$ is decreasing, and by using $(i)$ it is bounded. By integrating the last equality, we deduce that $u'\in L^2(\R^+,H)$. Since $H\hookrightarrow V'$ we deduce that $h(t):=\Vert u'(t)\Vert_{V'}^2$ is integrable. Moreover by assumption $(i)$ and the equation \eqref{secorder}, for almost every $t>0$ we find
$$\vert h'(t)\vert \leq2 \Vert u'(t)\Vert_{V'}\Vert u''(t)\Vert_{V'}\leq C$$
Hence the function $h$ is Lipschitz continuous and integrable which implies $\displaystyle \lim_{t\to\infty}^{}h(t)=0$. Since $u'$ is compact with values in $H$ we deduce
$$\lim_{t\to\infty}\Vert u'(t)\Vert_H=0 .$$
Let  \index{$\omega$-limit set}$(\varphi,\psi)\in\omega (u,u')$, and let $(t_n)_{n\in\N}\subset\R_+$ be an unbounded increasing sequence such that $\displaystyle\lim_{n\to\infty} (u(t_n),u'(t_n))=(\varphi,\psi)$. Obviously we get  $\psi=0$. On the other hand, since $\Vert u'\Vert_H\longrightarrow0$, it follows  that 
\begin{equation}\label{CvUniformeuH}\displaystyle\lim_{n\to\infty} \sup_{s\in[0,1]}\Vert u(t_n+s)-\varphi\Vert_H =0.\end{equation}
 Actually the same is true with values in $V$. In fact, assuming the contrary, there is $\delta>0$ such that
$$\forall n\in\N,\quad\sup_{s\in[0,1]}\Vert u(t_n+s)-\varphi\Vert_V\geq\delta.$$
Then we can find a sequence $(s_n)\subset[0,1]$ such that
$$\forall n\in\N,\quad \Vert u(t_n+s_n)-\varphi\Vert_V\geq\frac{\delta}{2}.$$
By compactness of $u$ in $V$, we can find $\psi\in V$ and subsequences still denoted $(t_n)$ and $(s_n)$ such that 
$$\Vert u(t_n+s_n)-\psi\Vert_V\longrightarrow 0$$
which imply that $\Vert \psi-\varphi\Vert_V\geq\frac{\delta}{2}.$ Now from \eqref{CvUniformeuH} we deduce that $\varphi=\psi$, a contradiction.\\
Therefore, $\displaystyle\lim_{n\to\infty} \nabla E (u(t_n+s))=\nabla E (\varphi )$ uniformly in $s\in [0,1]$. \\
By equation \eqref{secorder},
\begin{eqnarray*}
 \nabla E (\varphi ) & = & \int_0^1 \nabla E (\varphi ) \; ds \\
& = & \lim_{n\to\infty} \int_0^1 \nabla E (u(t_n+s)) \; ds \\
& = & \lim_{n\to\infty} \int_0^1 (-u''(t_n+s) -u'(t_n+s)) \; ds\\
&=&\lim_{n\to\infty} -u'(t_n+1)+u'(t_n)-u(t_n+1)+u(t_n)\\
&=& 0 .
\end{eqnarray*} Hence $\varphi\in {\cal S}$. Now since $\cE$ 
is bounded and decreasing, the limit  $K:=\displaystyle\lim_{t\to\infty}  \cE(t)=\displaystyle\lim_{t\to\infty} E(u(t))$ exists.  Replacing $E$ by $E-K$ we may assume $K = 0$.\\
Now let $\varepsilon$ be a positive real number, and as in \cite{MR1714129} let us define
for all
$t\ge 0$
\begin{equation}\label{DefFtLiapunovSdOrderAbst} Z(t)={1\over 2}\Vert u'\Vert_H^2+ E(u)+\varepsilon\langle \nabla E(u),\, u'\rangle_{V'}\end{equation} where $\langle\cdot,\cdot\rangle_{V'} $ denotes the inner product in $
V'.$ We note that $Z$  makes sense as a consequence of hypothesis $(i)$. We have for almost all $t\geq 0$:
$$Z'(t)=-\Vert u'\Vert_H^2+\varepsilon \{-\langle\nabla E(u),\, u'\rangle_{V'} -
\Vert \nabla E(u)\Vert_{V'}^2   +\langle(\nabla E(u))' ,u'\rangle_{V'}\}.$$
Then, using $(ii)$,  for almost all $t\ge 0$ we obtain for some $P>0$
$$Z'(t)\le (- 1 +P \varepsilon)\Vert
u'\Vert_H^2-\varepsilon  \langle\nabla E(u),u'\rangle_{V'} -\varepsilon
\Vert \nabla E(u)\Vert_{V'}^2.$$
Since we have by Cauchy-Schwarz inequality
$$ \langle \nabla E(u),\, u'\rangle_{V'}\leq {1\over 2}\Vert \nabla E(u)\Vert_{V'}^2 +{1\over 2}\Vert u'\Vert_{V'}^2 , $$
we deduce :
$$Z'(t)\leq (- 1 +P \varepsilon )\Vert u'\Vert_H^2+ {\varepsilon\over 2} \Vert u'\Vert_{V'}^2 - {\varepsilon\over 2}\Vert\nabla E (u)\Vert_{V'}^2.$$
By choosing $\varepsilon$ small enough, we see that there exists $c_1>0 $ such that for almost all $t\geq 0 $
\begin{equation}\label{ineqenergieSdordre}Z'(t)\leq  -  c_1 (\Vert u'\Vert_H^2+\Vert\nabla E(u)\Vert_{V'}^2) .
\end{equation}
Since $Z$ is nonincreasing with limit $0$, we have in particular $Z$ is nonnegative. As in the proof of the Theorem \ref{ThmCvSdOrderAbstr}   we can assume that $Z(t)>0$ for all $t\geq 0$.\\
Let $\Gamma=\{\varphi/\ (\varphi,0)\in \omega(u,u')\}$. Theorem \ref{thm 4.1.8.} ii) implies that $\Gamma$ is compact and connected.
Now by assumption $(iii)$, $E$ satisfies the {\L}ojasiewicz \index{Lojasiewicz} gradient  inequality \eqref{eqloja-Simon} at every point $\varphi\in {\cal S}$.  Applying Lemma \ref{InegLojsABSRAITLemme} with $W=V$, $X=V'$,  and ${\cal G}=\nabla E$ we obtain,
$$ \exists \sigma, c>0,\ \exists \theta\in(0,\frac12]/\ \left[\hbox{dist}(u,\Gamma)\leq \sigma\Longrightarrow\Vert \nabla E(u)\Vert_{V'}\geq c\vert E(u)\vert^{1-\theta}\right].$$
Now by the definition of  $\Gamma$ and using Theorem \ref{thm 4.1.8.} iii), we obtain that  there exists $T>0$ such that $\hbox{dist}(u,\Gamma)\leq \sigma$ for all $t\geq T$. Then we get  
\begin{equation}\label{ineqLojSimTrajSdordre}
\forall t\geq T\quad \Vert \nabla E(u)\Vert_{V'}\geq c\vert E(u)\vert^{1-\theta}.
\end{equation}
Using this last inequality together with Cauchy-Schwarz and Young inequalities, we get for all $t\geq T$
\begin{eqnarray}
\nonumber Z(t)^{2(1-\theta)}&\leq& C_2\{ \Vert u'\Vert_H^2 +\Vert\nabla E(u)\Vert_{V'}^2+\vert E(u)\vert\}^{2(1-\theta)}.\\
\label{ineqdiffsdordreabstra}&\leq&C_3\{ \Vert u'\Vert_H^2 +\Vert\nabla E(u)\Vert_{V'}^2\}
\end{eqnarray}
Combining the inequalities \eqref{ineqenergieSdordre}  and \eqref{ineqdiffsdordreabstra} we find for all $t\geq T$
$$Z'(t)\leq -\frac{c_1}{C_3}Z(t)^{2(1-\theta)}. $$
The conclusion follows easily \end{proof}

\section{Examples}
\subsection{A semilinear heat equation \index{heat equation}}\label{subheat}
 As a first application we study the asymptotic behaviour of the semilinear heat equation \index{heat equation}
\begin{equation} \label{heat2}
\left\{ \begin{array}{ll}
u_t -\Delta u + f(x,u) = 0 , & (t,x)\in\R_+\times\Omega ,\\[2mm]
u(t, \cdot )|_{\partial\Omega} = 0 , & t\in\R_+ ,\\[2mm]
u(0,x) = u_0 (x) , & x\in\Omega .
\end{array} \right.
\end{equation}
In equation \eqref{heat2} \index{heat equation} we assume that $\Omega\subset\R^N$ ($N\geq 1$) is a bounded domain. We assume  that   $f:\overline{ \Omega}\times \R\longrightarrow\R$ is continuously differentiable and if $N\geq 2$, we assume in addition that \index{growth condition}
\begin{equation} \label{HypotheseCroissancef}
\begin{array}{l}
\exists C>0, \alpha \geq0 \mbox{ such that } (N-2)\alpha \leq 2 \\[2mm]
\mbox{ and }\vert {\partial f\over \partial s}(x,s)\vert \leq C (1+\vert s\vert^{\alpha})\ \hbox{ a.e. on  }\ \Omega\times \R
\end{array}\end{equation}
With this condition on $f$, the energy \index{energy} functionnal $E$ given by 
$$\forall u\in H^1_0(\Omega),\  E(u)=\frac12\int_{\Omega}\vert \nabla u\vert^2\, dx+\int_{\Omega}F(u)\, dx, $$
where   $F(x,s):=\int_0^s f(x,r) \; dr$, is well defined.
By using Proposition 1.17.5 page 66 of \cite{MR1276944}, we know that $E$ is of class $C^2$ on  $H^1_0(\Omega)$ and
\begin{eqnarray*}
DE(u)&=& -\Delta u+f(x,u),\quad \forall u\in H^1_0(\Omega),\\
 D^2 E(u)\xi &=&-\Delta \xi+\frac{\partial f}{\partial s}(x,u)\xi,\quad \forall u,\xi\in H^1_0(\Omega).
\end{eqnarray*}
It is well known that $ D^2E(\varphi)$ is a semi-Fredholm\index{semi-Fredholm} operator for all $\varphi\in H^1_0(\Omega)\cap L^\infty(\Omega)$ (see example \ref{ExempleFredholm}). Let $d=\dim \ker DE(\varphi)$.
\begin{prop}\label{Loj-SimonLaplacien}Assume that  hypothesis \eqref{HypotheseCroissancef}  is satisfied.  Let $\varphi\in H^1_0(\Omega)\cap L^\infty(\Omega)$ be a critical point of $E$. Assume also that one of the following hypotheses is satisfied :
\begin{eqnarray}
\label{NoyauTriviale}& & d=0\\
\label{EgaliteDimLap11}& & d>0 \hbox{ and there exists }O\subset\R^d \hbox{ open, and }h\in C^1(O,V)/\\
\nonumber & &\varphi\in h(O)\subset (DE)^{-1}(0) \hbox{ and } h:O\longrightarrow h(O)\hbox{ is a diffeomorphism; }\\
\label{fAnalytique} & &  f\hbox{ is analytic \index{analytic} in } s , \hbox {uniformly with respect to } x \in\Omega 
\end{eqnarray}
Then there exist $\theta\in(0,{1\over2}]$ and $\sigma>0$ such that 
 \begin{equation}\label{InegaliteLojSimonLaplacien}
\forall u \in H^1_0(\Omega), \quad \Vert u -\varphi \Vert_{H^1_0(\Omega)}\ <\sigma\Longrightarrow
 \Vert -\Delta u+f(x,u) \Vert_{H^{-1}(\Omega)}\  \geq \vert E(u)-E(\varphi)\vert ^{1-\theta} .
\end{equation}
\end{prop}

\begin{proof}[{\bf Proof.}] Let $A:=D^2E(\varphi)$ and assume that $d=0$. 
Corollary \ref{SFredplusProj}  gives that $A=D^2E(\varphi)$ is an isomorphism from $H^1_0(\Omega)$ into $H^{-1}(\Omega)$. To conclude we have just to apply proposition \ref{generalNoyauNul.}.
Now assume \eqref{EgaliteDimLap11} holds. To apply Theorem \ref{generalEgaliteDimension}, we have just to remark that $A$ is a semi-Fredholm\index{semi-Fredholm} operator. 
For the  proof of \eqref{InegaliteLojSimonLaplacien} under hypothesis \eqref{fAnalytique}, we distinguish two cases :\\

\textbf{Case 1 :} $N\leq 3$. Let $Z=L^2(\Omega)$, by elliptic regularity \cite{MR0125307} we get that $W:=(\Pi+A)^{-1}(Z)\subset H^2(\Omega)$ where $\Pi$ is the orthogonal projection in $L^2(\Omega)$ on $N(A):= \ker A$. The functional $E:H^2(\Omega)\cap H^1_0(\Omega)\longrightarrow \R $ is clearly analytic \index{analytic} since it is the sum of a continuous quadratic functional  and a Nemytskii \index{Nemytskii operator} operator which is analytic \index{analytic} on the Banach algebra $H^2 (\Omega )\subset L^{\infty}(\Omega)$ (see example \ref{Nemytskiianalytic}.) By using Proposition \ref{Df}, we also obtain  that $DE:W\longrightarrow Z$ is analytic\index{analytic}. We can apply Theorem \ref{generalAnalytique.} to obtain  \eqref{InegaliteLojSimonLaplacien}.\\

\textbf{Case 2 :} $N\geq 4$. Let $p> \frac N2$ and $Z=L^p(\Omega)$. By elliptic regularity \cite{MR0125307}, we know that $W:=(\Pi+A)^{-1}(Z)\subset W^{2,p}(\Omega)$ which is a Banach algebra since $p>\frac N2.$ The end is the same as in the first case.
\end{proof}
\begin{rem}{\rm 1) The  result of proposition \ref{Loj-SimonLaplacien} remains true  for the general energy \index{energy} defined by :
\begin{equation} \label{energy3}
E(u) := \frac{1}{2} \sum_{i,j=1}^d \int_\Omega a_{ij} \frac{\partial u}{\partial x_i} \frac{\partial u}{\partial x_j}  + \int_\Omega F(x,u) , \quad u\in H^1_0 (\Omega ) ,
\end{equation}
where  $F(u)=\int_0^uf(s)\,ds$, $f$ satisfies \eqref{HypotheseCroissancef} and $a_{i,j}$ satisfies the following conditions :
\begin{enumerate}
\item \label{a1} $a_{ij} \in  C^1 (\bar{\Omega})$,
\item \label{a2} $a_{ij} = a_{ji}$, and
\item \label{a3} $\displaystyle \sum_{i,j=1}^d a_{ij} (x) \xi_i \xi_j \geq \gamma \| \xi\|^2$ for some $\gamma >0$ and every $\xi\in\R^d$, $t\in\R_+$, $x\in\Omega$,
\end{enumerate}
\noindent 2) A similar result holds true for Neumann boundary conditions}\end{rem}

\bigskip 

The following result is an immediate application of Theorem \ref{ThmCvSdOrderAbstr} using the Proposition \ref{Loj-SimonLaplacien}.  The smoothing effect of the heat \index{heat equation} equation implies (cf.\cite{MR0747194} ) that for each $\varepsilon > 0$ and $\alpha\in[0, 1)$,$$\bigcup_{t\geq
\varepsilon}\{u(t)\}\quad\hbox{is bounded in }\quad C^{1+\alpha}({\overline\Omega})$$ as soon as $u(t)$ is bounded in $L^\infty(\Omega)$ for $t\ge0$. In particular,  $\bigcup_{t\geq 0}\{u(t)\}$ is precompact in $H^1_0(\Omega)$ . 

\begin{thm} \label{heatmain}\index{heat equation}
Let $u\in C^1(\R_+, H^1_0(\Omega))$ be a  bounded solution of equation \eqref{heat2}. Assume that for all $\varphi\in S:=\{\varphi\in H^1_0(\Omega)/\ -\Delta \varphi+f(\varphi)=0\}$ we have $\varphi\in L^\infty(\Omega)$ and  one of the three conditions \eqref{NoyauTriviale},  \eqref{EgaliteDimLap11} or \eqref{fAnalytique} of Proposition \ref{Loj-SimonLaplacien} is satisfied. Then \index{convergence result}
$$\displaystyle\lim_{t\to\infty} \Vert u(t)-\varphi\Vert_{H^1}=0.$$
 Moreover, let $\theta$ be any {\L}ojasiewicz\index{Lojasiewicz} exponent of $E$ at $\varphi$. Then we have
\begin{equation*} 
\Vert u(t) - \varphi \Vert_{L^2}= \left\{ \begin{array}{ll}
O(e^{-\delta t} )  & \hbox{ for some }\delta>0 \text{ if } \theta =\frac12, \\[2mm]
O(t^{-\theta/(1-2\theta)} )  & \text{if } 0<\theta<{1\over2}.
\end{array} \right.
\end{equation*}\end{thm} 
\begin{rem}{\rm It has been shown in \cite{MR2328934}  that if $d\le 1$, convergence holds without any need of condition \eqref{EgaliteDimLap11} or \eqref{fAnalytique} . However, if $d = 1$ and convergence occurs, in general  the  convergence can be arbitrarily slow. The hypothesis $d\le 1$ provides convergence results in a wide framework, cf. e.g. \cite{MR1149371}, \cite{MR1205867}}.\end{rem}

\subsection{A semilinear wave equation}

As a next application we study the asymptotic behaviour of the semilinear wave equation \index{wave equation}
\begin{equation} \label{wave2}
\left\{ \begin{array}{ll}
u_{tt} + u_t - \Delta u + f(x,u) = 0 , & (t,x)\in\R_+\times \Omega ,\\[2mm]
u(t, \cdot )|_{\partial\Omega} = 0 , & t\in\R_+ , \\[2mm]
u(0,x) = u_0 (x) , u_t (0,x) = u_1 (x), & x\in\Omega .
\end{array} \right.
\end{equation}
We let $\Omega\subset\R^d$, $f\in C^1(\bar{\Omega}\times\R ;\R)$,  the spaces $H:=L^2 (\Omega )$ and $V:=H^1_0 (\Omega )$  as in Subsection \ref{subheat}. If $N\geq 2$, then we replace the growth condition \index{growth condition} \eqref{HypotheseCroissancef} by the following condition :
\begin{equation} \label{HypotheseCroissancef2}
\begin{array}{l}
\exists C>0, \alpha \geq0 \mbox{ such that } (N-2)\alpha < 2 \\[2mm]
\mbox{ and }\vert {\partial f\over \partial s}(x,s)\vert \leq C (1+\vert s\vert^{\alpha})\ \hbox{ a.e. on  }\ \Omega\times \R
\end{array}\end{equation}
 
 \begin{thm} \label{wavemain} Let $u$ be a solution of \eqref{wave2} such that \index{wave equation}
$$ \cup_{t\geq 0}^{}\{u(t,.),u_t(t,.)\}\hbox{ is bounded in } H^1_0(\Omega)\times L^2({\Omega}).$$
 Assume that for all $\varphi\in S:=\{\varphi\in H^1_0(\Omega)/\ -\Delta \varphi+f(\varphi)=0\}$ we have $\varphi\in L^\infty(\Omega)$ and  one of the three conditions \eqref{NoyauTriviale} or \eqref{EgaliteDimLap11} or \eqref{fAnalytique} of Proposition \ref{Loj-SimonLaplacien} is satisfied. 
 Then \index{convergence result}
 $$\displaystyle\lim_{t\to\infty}\Vert u_t\Vert_{L^2} +\Vert u(t)-\varphi\Vert_{H^1}=0.$$
 Moreover, let $\theta$ be any {\L}ojasiewicz\index{Lojasiewicz} exponent of $E$ at $\varphi$. Then we have \index{convergence result}
\begin{equation*} 
\Vert u(t) - \varphi \Vert_{L^2}= \left\{ \begin{array}{ll}
O(e^{-\delta t} )  & \hbox{ for some }\delta>0 \text{ if } \theta =\frac12, \\[2mm]
O(t^{-\theta/(1-2\theta)} )  & \text{if } 0<\theta<{1\over2}.
\end{array} \right.
\end{equation*}
\end{thm}

\begin{proof}[{\bf Proof.}] First \eqref{HypotheseCroissancef2} implies that the Nemytskii \index{Nemytskii operator} operator associated to $f$ is compact: $
H^1_0(\Omega)\rightarrow L^2({\Omega})$, then by the lemma \ref{Lemma Webb4.5.2.}, the orbit
$\cup_{t\geq 0}^{}\{u(t,.),u_t(t,.)\}$
 is precompact in $
 H^1_0(\Omega)\times L^2({\Omega})$. This is condition $(i)$ of theorem \ref{CvAbstraitSdOrder}. 
 Moreover, the duality mapping $K : H^{-1} (\Omega ) \to H^1_0 (\Omega )$ is given by $Kv = (-\Delta)^{-1} v$, so that $KE''(v) = I +(-\Delta )^{-1} f'(v)$. From this, the growth assumption  on $f$ \eqref{HypotheseCroissancef2}, and the Sobolev embedding theorem, it is not difficult to deduce that the condition $(ii)$ of Theorem \ref{CvAbstraitSdOrder} is satisfied.
\end{proof}

\chapter[Variants and additional results]{Variants and additional results}

In this chapter, we collect, most of the time  without proofs a few additional results which complement, mainly in the infinite dimensional framework and often at the price of additional technicalities,  the simple theory developed in the two previous chapters. For the proofs, the reader is invited to read the corresponding specialized papers
\section{Convergence in natural norms}  In the last chapter, we obtained convergence to \index{equilibrium point} equilibrium for some semi-linear parabolic and hyperbolic equations in the energy \index{energy} space. However the rate of convergence to \index{equilibrium point} equilibrium was  specified in $L^2({\Omega})$. In \cite{MR1829143}, it is shown that the same decay occurs in 
$H^1_0(\Omega)$ for the wave equation and in $L^\infty ({\Omega})$ with an arbitrarily small loss for the heat equation. This loss is most probably artificial but this becomes only important when the {\L}ojasiewicz\index{Lojasiewicz} exponent of $\varphi$ is exactly known, which is possible only in exceptional cases. 
\section{Convergence without growth restriction for the heat equation \index{heat equation}} In \cite {MR1609269}, the second author gave a proof of the Simon convergence theorem (cf. \cite{{MR0727703}} in the framework of Sobolev spaces instead of $C^\alpha$ spaces which were used by L. Simon. His proof is quite similar to that of our main parabolic result, but uses more complicated spaces. The advantage is that no growth restriction is assumed for the nonlinear perturbative term. 
\section{More general applications} 
\subsection{Systems} 
Let $V=(H^1_0(\Omega))^n$, $H=(L^2(\Omega))^n$, $V'=(H^{-1}(\Omega))^n$ and we define the function $E:(H^1_0(\Omega))^n\longrightarrow\R$ by
$$\forall u=(u_1,\cdots, u_n)\in (H^1_0(\Omega))^n,\quad E(u)=\frac12\sum_{i=1}^n\int_{\Omega}\vert \nabla u_i\vert^2\, dx+\int_{\Omega}F(u)\, dx.$$
 When $N\geq 2,$ we assume  that  
\begin{equation}\label{HypotheseCroissancefSyst} 
 \Vert \nabla_s^2 F(x,s)\Vert \leq C (1+\Vert s\Vert^{\alpha})\ \hbox{ a.e. on  }\ \Omega\times \R
\end{equation} 
 for some $C \geq 0$ and $\alpha \geq 0$ such that $ (N-2)\alpha <2$.
By using Proposition 1.17.5 page 66 of \cite{MR1276944}, we know that $E$ is of class $C^2$ on  $H^1_0(\Omega)$ and
\begin{eqnarray*}
DE(u )&=&(-\Delta u_1+f_1(x,u),\cdots, -\Delta u_n+f_n(x,u))   \\
D^2 E(u)(\xi)&=&<-\Delta \xi_1+\frac{\partial f_1}{\partial s_1}(x,u)\xi_1,\cdots,-\Delta \xi_n+\frac{\partial f_1}{\partial s_n}(x,u)\xi_n) \ \forall \xi\in (H^1_0(\Omega))^n.
\end{eqnarray*}
It is well known that $\dim \ker D^2E(\varphi)$ is finite for all $\varphi\in (H^1_0(\Omega))^n\cap(L^\infty(\Omega))^n$. Let $d=\dim \ker DE(\varphi)$.
\begin{prop}\label{Loj-SimonLaplacien-syst}Assume that  hypothesis \eqref{HypotheseCroissancefSyst}   is satisfied.  Let $\varphi\in (H^1_0(\Omega))^n\cap (L^\infty(\Omega))^n$ be a critical point of $E$. Assume also that one of the following hypotheses is satisfied :
\begin{eqnarray*}
 & & d=0\\
  & & d>0 \hbox{ and there exists }O\subset\R^d \hbox{ open, and }h\in C^1(O,V)/\varphi\in h(O)\subset (DE)^{-1}(0)\\
\nonumber & & \hbox{ and } h:O\longrightarrow h(O)\hbox{ is a diffeomorphism; }\\
  & &  f\hbox{ is analytic \index{analytic} in } s , \hbox {uniformly with respect to } x \in\Omega 
\end{eqnarray*}
Then there exist $\theta\in(0,{1\over2}]$ and $\sigma>0$ such that 
 \begin{equation}\label{InegaliteLojSimonLaplaciensystem}
\forall u \in (H^1_0(\Omega))^n, \quad \Vert u -\varphi \Vert_{H^1_0(\Omega)}\ <\sigma\Longrightarrow
 \Vert DE(u) \Vert_{(H^{-1}(\Omega))^n}\  \geq \vert E(u)-E(\varphi)\vert ^{1-\theta} 
\end{equation}
\end{prop}

\subsection{Fourth order operators} 
Let $V=H^2_0(\Omega)$, $H=L^2(\Omega)$, $V'=H^{-2}(\Omega)$ and we define the function $E:H^2_0(\Omega)\longrightarrow\R$ by
$$\forall u\in H^2_0(\Omega),\quad E(u)=\frac12\int_{\Omega}\vert \Delta u\vert^2\, dx+\int_{\Omega}F(u)\, dx$$
where  $F(u)=\int_0^uf(s)\,ds.$ When $N\geq 4,$ we assume  that $ f(x, 0)\in L^{\infty}(\Omega) $  { and }
\begin{equation}\label{HypotheseCroissancefN4} 
 \vert {\partial f\over \partial s}(x,s)\vert \leq C (1+\vert s\vert^{\alpha})\ \hbox{ a.e. on  }\ \Omega\times \R
\end{equation} 
 for some $C \geq 0$ and $\alpha \geq 0$ such that $ (N-4)(\alpha+1) < N+4$.
By using Proposition 1.17.5 page 66 of \cite{MR1276944}, we know that $E$ is of class $C^2$ on  $H^2_0(\Omega)$ and
\begin{eqnarray*}
<DE(u),\psi>_{H^{-2}\times H^2_0}&=&<\Delta^2 u+f(x,u),\psi>_{H^{-2}\times H^2_0}\forall \psi\in H^2_0(\Omega),\\
<D^2E(u)\xi,\psi>_{H^{-2}\times H^2_0}&=&<\Delta^2 \xi+\frac{\partial f}{\partial s}(x,u)\xi,\psi>_{H^{-2}\times H^2_0}\forall \psi\in H^2_0(\Omega).
\end{eqnarray*}
It is well known that $\dim \ker E'(\varphi)$ is finite for all $\varphi\in H^2_0(\Omega)$. Let $d=\dim \ker E'(\varphi)$.
\begin{prop}Assume that  hypothesis \eqref{HypotheseCroissancefN4}  is satisfied.  Let $\varphi\in H^2_0(\Omega)\cap L^\infty(\Omega)$ be a critical point of $E$. Assume also that one of the following hypotheses is satisfied :
\begin{eqnarray}
\label{NoyauTrivBilap}& & d=0\\
\label{EgaDimLapBil} & & d>0 \hbox{ and there exists }O\subset\R^d \hbox{ open, and }h\in C^1(O,V)/\\
\nonumber & & \varphi\in h(O)\subset (DE)^{-1}(0) \hbox{ and } h:O\longrightarrow h(O)\hbox{ is a diffeomorphism; }\\
\label{fAnalytiqueBiLaplacien} & &  f\hbox{ is analytic \index{analytic} in } s , \hbox {uniformly with respect to } x \in\Omega 
\end{eqnarray}
Then there exist $\theta\in(0,{1\over2}]$ and $\sigma>0$ such that \index{convergence result}
 \begin{equation}\label{InegaliteLojSimonBiLaplacien}
\forall u \in H^2_0(\Omega), \quad \Vert u -\varphi \Vert_{H^2_0(\Omega)}\ <\sigma\Longrightarrow
 \Vert \Delta^2 u+f(x,u) \Vert_{H^{-2}(\Omega)}\  \geq \vert E(u)-E(\varphi)\vert ^{1-\theta} 
\end{equation}
\end{prop}
\begin{rem} {\rm In virtue of remark \ref{PhiNestpasCritiqueTriviale}, if $\varphi$ is not a critical point of $E$, \eqref{InegaliteLojSimonLaplacien} is  just the consequence of the fact that $E\in C^1(V,V').$ In this case we don't have to assume any assumption.}
\end{rem}

\begin{proof}[{\bf Proof.}] The proof of \eqref{InegaliteLojSimonBiLaplacien} under hypotheses \eqref{NoyauTrivBilap} and\eqref{EgaDimLapBil} is the same as in the proposition \ref{Loj-SimonLaplacien}. Now assume that \eqref{fAnalytiqueBiLaplacien} holds. As in the proof of the proposition \ref{Loj-SimonLaplacien}, we distinguish two cases :\\
\textbf{Case 1 :} $N\leq 3$. Let $Z=L^2(\Omega)$, by elliptic regularity \cite{MR0125307}, we know that $W:=(\Pi+A)^{-1}(Z)\subset H^4(\Omega)$ where $\Pi$ is the orthogonal projection in $L^2(\Omega)$ on $N:= \ker A$. It is also clear that $N\subset Z=L^2(\Omega)$.
The functional $E:  H^2_0(\Omega)\longrightarrow \R $ is clearly analytic \index{analytic} since it is the sum of a continuous quadratic functional  and a Nemytskii \index{Nemytskii operator} operator which is analytic \index{analytic} on the Banach algebra $H^2(\Omega )\subset L^{\infty}(\Omega)$ (see Example \ref{Nemytskiianalytic} .) By using Proposition \ref{Df} ), we also obtain  that $DE:W\longrightarrow Z$ is analytic\index{analytic}. We can apply Theorem \ref{generalAnalytique.} to obtain  \eqref{InegaliteLojSimonLaplacien}.\\
\textbf{Case 2 :} $N\geq 4$. Let $p>\max(2, \frac N4)$ and $Z=L^p(\Omega)$. By elliptic regularity \cite{MR0125307}, we know that $W:=(\Pi+A)^{-1}(Z)\subset W^{4,p}(\Omega)$ which is a Banach algebra since $p>\frac N4.$ The end is the same as in the first case.
\end{proof}

\section{The wave equation with nonlinear damping\index{wave equation}} In \cite{MR2429439}, L. Chergui succeeded to generalize Theorem \ref{chergui} to the semilinear wave equation with nonlinear localized damping 
\begin{equation} \label{wavechergui}
\left\{ \begin{array}{ll}
u_{tt} + |u_t|^{\alpha} u_t - \Delta u + f(x,u) = 0 , & (t,x)\in\R_+\times \Omega ,\\[2mm]
u(t, \cdot )|_{\partial\Omega} = 0 , & t\in\R_+ , \\[2mm]
u(0,x) = u_0 (x) , u_t (0,x) = u_1 (x), & x\in\Omega .
\end{array} \right.
\end{equation}
One of the difficulties to do that is the proof of compactness of the trajectories in the energy \index{energy} space.  His result has been extended, under natural hypotheses,  to possibly nonlocal damping terms in \cite{MR3095203}.
\section{Some explicit decay rates under additional conditions} The {\L}ojasiewicz\index{Lojasiewicz} exponent of an equilibrium \index{equilibrium point} point is generally difficult to find, even for 2-dimensional ODE systems. However in some exceptional case, it turns out, for semilinear problems involving a power non-linearity, to be computable explicitely. This was done in \cite{MR2019030} with application to the exact decay of the solution when the limit is $0$, and in \cite{MR2774060} under a positivity condition of the \index{energy} energy. The last result allows for a continuum of equilibria to exist, but only for Neuman boundary conditions. 
\section{More information about decay rates}  All our convergence results contain an estimate of the difference between the limiting equilibrium \index{equilibrium point} and the solution. The question naturally arises of the optimality of this estimate. Actually, even when the equation has a single \index{equilibrium point} equilibrium playing the role of a universal attractor of all solutions, the situation can be rather complicated. If the decay estimate obtained for instance by Liapunov's \index{Liapunov} direct method or $\L$ojasiewicz method is optimal for all solutions other than the rest point itself, it means that all non-trivial solutions tend to the quilibrium at the same rate, a circumstance which tends to be the exception rather than the general rule. As an illustration, let us consider the simple ODE 

$$ u''+ u'  + u^3 = 0$$  

Apart from the zero solution, it is true (although not completely trivial to prove, cf. e.g.\cite{{MR2145567}} that here are only two possible rates of decay: as $t^{-\frac{1}{2}}$ or as $e^{-t}$. Actually the first case corresponds to solutions behaving as those of $ u'' + u^3 = 0$ and is shared by most solutions, while the ranges of exponentially decaying solutions lie on a separatrix made of two curves symmetric with respect to the origin$(0, 0)$ having the rough shape of spirals. \\

Analogous properties have been established by the first author for the slightly more complicated equation
$ u''+ c \vert u'\vert ^{\alpha}u' + \vert u\vert ^{\beta}u = 0 $  where $c, \alpha, \beta $
are positive constants.  If $\alpha > \alpha_0 : = {\beta\over{\beta +2}}$, all trajectories are oscillatory  up to infinity and tend to $0$ at the same rate. If $\alpha < \alpha_0 $, all trajectories have a finite number of zeroes on $[0, \infty)$ and there are two different rates of decay at infinity . For the details , cf \cite{MR2763360}.\\

In a series of papers, the exact decay rate of solutions have been thoroughly studied for more complicated second order ODE and for  infinite-dimensional abstract problems containing semilinear parabolic and hyperbolic equations. We refer to \cite{MR2900650, MR2983445, GGH1, GGH2, GGH3} for the details.

\section{The asymptotically autonomous \index{autonomous} case}\index{autonomous}  It is natural to ask whether the convergence results are robust under a perturbative source which dies off sufficiently fast for $t$ large. Such results were obtained in \cite{HuTa01},  \cite{MR1978032}, \cite{MR2350247}, \cite{MR2732191} and  \cite {MR2802889}.
\section{Non convergence for heat and wave equations} \index{non convergence result} Non convergence results for parabolic and hyperbolic equations with smooth non-analytic nonlinearities were proved by  \cite {MR1370152}, \cite {MR1942223} and \cite{MR2018329}. Although such negative results may look natural since 2 dimensional ODE systems already produce such bad phenomena, the question is whether or not the fact that the generating function is scalar forces the system to behave like a scalar equation. The answer is negative but the proof is non-trivial. 

\bibliographystyle{amsplain}

\providecommand{\bysame}{\leavevmode\hbox to3em{\hrulefill}\thinspace}

\printindex
\end{document}